\documentclass{amsart}
\usepackage[utf8]{inputenc}
\usepackage{amsmath,amsthm, amssymb}
\usepackage{tikz,color}
\usepackage{hyperref}
\usepackage{todonotes}
\usepackage{cleveref}
\usepackage{caption}
\usepackage{subcaption}
\usetikzlibrary{calc,math}
\usepackage{diagbox}
\usetikzlibrary{intersections}

\newtheorem{theorem}{Theorem}[section]
\newtheorem{corollary}[theorem]{Corollary}
\newtheorem{proposition}[theorem]{Proposition}
\newtheorem{lemma}[theorem]{Lemma}
\newtheorem{conjecture}[theorem]{Conjecture}

\theoremstyle{definition}
\newtheorem{definition}[theorem]{Definition}
\newtheorem{example}[theorem]{Example}
\newtheorem{question}[theorem]{Question}
\newtheorem{remark}[theorem]{Remark}

\newcommand{\R}{\mathbb{R}}
\newcommand{\Z}{\mathbb{Z}}
\newcommand{\N}{\mathbb{N}}
\newcommand{\HH}{\mathcal{H}}
\newcommand{\permutahedralTiling}[1]{\mathcal{PT}_{#1}}
\newcommand{\heawoodGraph}[1]{H_{#1}}
\newcommand{\heawoodComplex}[1]{\mathcal{HC}_{#1}}
\newcommand{\affineGraph}[1]{\widetilde{G}_{#1}}
\newcommand{\HHaffine}{\widetilde{\mathcal{H}}}
\newcommand{\complexAffine}[1]{\widetilde{\mathcal{C}}_{#1}}
\newcommand{\latticeAffine}[1]{\widetilde{\mathcal{L}}_{#1}} % weight lattice 
\newcommand{\sublatticeAffine}[1]{\widetilde{\mathcal{S}}_{#1}} % sublattice
\newcommand{\fundamentalvectorsAffine}[1]{\widetilde F_{#1}} 
\newcommand{\fundamentalvectors}[1]{F_{#1}} 
\newcommand{\lattice}[1]{\mathcal{L}_{#1}} % weight lattice 
\newcommand{\torus}[1]{\mathcal{T}_{#1}} % torus 
\newcommand{\sublattice}[1]{\mathcal{S}_{#1}} % sublattice
\newcommand{\x}{\mathbf{x}}
\newcommand{\y}{\mathbf{y}}
\newcommand{\kk}{\mathbf{k}}
\newcommand{\B}{\mathbf{B}} %ordered partition
\newcommand{\perm}[1]{\operatorname{Perm}_{#1}}
\newcommand{\fundamentaltile}[1]{P_{#1}}
\newcommand{\conv}{\operatorname{conv}}
\newcommand{\vertices}{\operatorname{Vert}}
\newcommand{\equivalent}[1]{\sim_{#1}}
\newcommand{\Zrowspan}[1]{\mathbb{Z}\text{-rowspan}(#1)}
\newcommand{\Zspan}[1]{\mathbb{Z}\text{-span}(#1)}
\newcommand{\parallelepipedDomainAffine}[1]{\widetilde {\mathcal{D}}^{\operatorname{par}}_{#1}}
\newcommand{\permutahedronDomainAffine}[1]{\widetilde {\mathcal{D}}^{\operatorname{perm}}_{#1}}

\newcommand{\permutahedronDomain}[1]{{\mathcal{D}}^{\operatorname{perm}}_{#1}}

\DeclareRobustCommand{\stirling}{\genfrac{\{}{\}}{0pt}{}} %Stirling numbers of the second kind

\definecolor{blue}{rgb}{0, 0.445, 0.695}
\definecolor{bluishgreen}{rgb}{0, 0.626, 0.456}
\definecolor{red}{rgb}{0.896, 0.395, 0}
\definecolor{purple}{rgb}{0.783, 0.464, 0.640}
\definecolor{skyblue}{rgb}{0.359, 0.752, 0.973}
\definecolor{orange}{rgb}{0.999, 0.706, 0.0}
\definecolor{yellow}{rgb}{0.937, 0.890, 0.258}

\definecolor{olive}{RGB}{116,141,19}
\definecolor{green}{RGB}{108,208,48}
\definecolor{teal}{RGB}{47,77,62}
\definecolor{turquoise}{RGB}{86,235,211}
\definecolor{lightblue}{RGB}{150,178,153}
\definecolor{blue2}{RGB}{25,50,191}
\definecolor{indigo}{RGB}{142,128,251}
\definecolor{indigo2}{RGB}{114,32,246}
\definecolor{lightpurple}{RGB}{243,197,250}
\definecolor{purple2}{RGB}{105,66,131}
\definecolor{magenta}{RGB}{206,43,188}
\definecolor{brown}{RGB}{110,57,13}

\newcommand{\defn}[1]{{\color{green!50!black}\emph{#1}}}

\makeatletter
\def\hyper@x#1,#2\relax{#1}
\def\hyper@y#1,#2\relax{#2}
\def\hyper@coords#1{#1}

\newif\ifhyper@vertical

\def\hyper@disc@computer#1#2{%
  \edef\hyper@toscan{(#1)}
  \tikz@scan@one@point\hyper@coords\hyper@toscan
  \edef\hyper@sx{\the\pgf@x}
  \edef\hyper@sy{\the\pgf@y}
  \edef\hyper@toscan{(#2)}
  \tikz@scan@one@point\hyper@coords\hyper@toscan
  \edef\hyper@ex{\the\pgf@x}
  \edef\hyper@ey{\the\pgf@y}
  \pgfmathsetmacro{\hyper@det}{\hyper@sx * \hyper@ey - \hyper@sy * \hyper@ex}
  \pgfmathparse{\hyper@det == 0 ? "\noexpand\hyper@verticaltrue" : "\noexpand\hyper@verticalfalse"}
  \pgfmathresult
  \ifhyper@vertical
  \edef\hyper@cmd{-- (\tikztotarget)}
  \else
  \pgfmathsetmacro{\hyper@mx}{(\hyper@ex + \hyper@sx)/2}
  \pgfmathsetmacro{\hyper@my}{(\hyper@ey + \hyper@sy)/2}
  \pgfmathsetmacro{\hyper@dx}{\hyper@ex - \hyper@sx}
  \pgfmathsetmacro{\hyper@dy}{\hyper@ey - \hyper@sy}
  \pgfmathsetmacro{\hyper@dradius}{\pgfkeysvalueof{/tikz/hyperbolic disc radius}}
  \pgfmathsetmacro{\hyper@t}{((\hyper@dradius)^2 - \hyper@sx * \hyper@ex - \hyper@sy * \hyper@ey)/(2 * (\hyper@sx * \hyper@ey - \hyper@sy * \hyper@ex))}
  \pgfmathsetmacro{\hyper@radius}{sqrt((\hyper@t)^2 + .25) * veclen(\hyper@dx,\hyper@dy)}
  \pgfmathsetmacro{\hyper@cx}{\hyper@mx + \hyper@t * \hyper@dy}
  \pgfmathsetmacro{\hyper@cy}{\hyper@my - \hyper@t * \hyper@dx}
  \pgfmathsetmacro{\hyper@sangle}{atan2(\hyper@sy-\hyper@cy,\hyper@sx - \hyper@cx)}
  \pgfmathsetmacro{\hyper@eangle}{atan2(\hyper@ey-\hyper@cy,\hyper@ex - \hyper@cx)}
  \pgfmathsetmacro{\hyper@eangle}{\hyper@eangle > \hyper@sangle + 180 ? \hyper@eangle - 360 : \hyper@eangle}
  \edef\hyper@cmd{arc[radius=\hyper@radius pt, start angle=\hyper@sangle, end angle=\hyper@eangle]}
  \fi
}

\tikzset{%
  hyperbolic disc radius/.initial={1cm},
  hyperbolic disc/.style={
    to path={
      \pgfextra{\hyper@disc@computer\tikztostart\tikztotarget}
      \hyper@cmd
    }
  },
}

\makeatother

\title[Generalized Heawood graphs]{Generalized Heawood graphs and \\ triangulations of tori}
\author{Cesar Ceballos}
\address{CC: TU Graz, Institut f\"ur Geometrie, Kopernikusgasse 24, 8010 Graz, Austria.}
\email{cesar.ceballos@tugraz.at}
\author{Joseph Doolittle}
\address{JD: TU Graz, Institut f\"ur Geometrie, Kopernikusgasse 24, 8010 Graz, Austria.}
\email{jdoolittle@tugraz.at}
\thanks{Both authors were supported by the Austrian Science Fund FWF, Project P 33278. 
Our work was also supported by the ANR-FWF International Cooperation Project PAGCAP, funded by the FWF Project I 5788. }
\date{\today}

\begin{document}
\begin{abstract}
    The Heawood graph is a remarkable graph that played a fundamental role in the development of the theory of graph colorings on surfaces in the 19th and 20th centuries. 
    Based on permutahedral tilings, we introduce a generalization of the classical Heawood graph indexed by a sequence of positive integers.
    We show that the resulting generalized Heawood graphs are toroidal graphs, which are dual to higher dimensional triangulated tori. We also present explicit combinatorial formulas for their $f$-vectors and study their automorphism groups.  
\end{abstract}

\maketitle

\tableofcontents
\section{Introduction}

The Heawood graph is a remarkable graph which played a fundamental role in the historical development of the theory of map colorings on surfaces. 
The four color theorem is an important result in this area, and perhaps one of the most well known results in mathematics in general. 
It states that for any map on a sphere, for example Europe, there is a coloring of that map with four colors, such that each region (or country) has one color and any two adjacent regions\footnote{Two regions are adjacent if they share a common boundary curve segment, not just a point.} have different colors.  

The four color problem has an interesting history, dating back to 1852 when the graduate student Francis Guthrie noticed that four colors where sufficient to color the map of England, and asked his younger brother Frederick Guthrie whether the same holds for any map on a sphere~\cite{guthrie_1880}. 
Frederick Guthrie posed this problem to Augustus De Morgan, a professor at University College London, who sent a letter to William Rowan Hamilton regarding this problem in October 23, 1852, see e.g.~\cite[Chapter~6]{biggs_graphtheory_1968}. 
The four color theorem was finally proved more than a hundred years later in 1976 by Kenneth Appel and Wolfgang Haken~\cite{AppelHaken_fourColorTheorem_1976,AppelHaken_fourColor_discharging_1977,AppelHaken_fourColor_reducibility_1977,AppelHaken_fourColor_ScientificAmerican_1977} after many false proofs and false counterexamples, and it is the first major result in mathematics that was proved using a computer.

One famous false proof of the four color theorem was given by Alfred Kempe in 1879~\cite{kempe_fourColor_1879}. 
His proof was announced in \emph{Nature}~\cite{kempe_fourColor_Nature_1880} and was regarded as an established fact for more than a decade. 
In 1890, Percy John Heawood found a gap in Kempe's proof, and modified his argument to show that five colors are sufficient to color a map on a sphere~\cite{heawood_torus_1984}. 
This became known as the five color theorem. 

In the same paper~\cite{heawood_torus_1984}, Heawood investigated coloring of maps on other surfaces. 
He showed that $N_p$ colors are sufficient to color a map on the oriented surface of genus $p\geq 1$, where 
\begin{align}\label{eq_heawoodnumber}
N_p = \left\lfloor \frac{7+\sqrt{1+48p}}{2} \right\rfloor
\end{align}
and $\lfloor x \rfloor$ is the largest integer not greater than $x$. 
For instance, it is possible to color any map on a torus (genus $p=1$ surface) using seven colors. 
Heawood also showed that for $p=1$ the number seven is tight, by showing a map of the torus where seven colors are necessary: a map consisting of seven regions for which any two regions are adjacent to each other. 

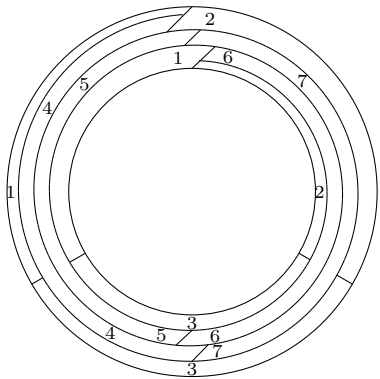
\begin{figure}[htb]
    \centering
        \begin{tikzpicture}%
    	[x={(240:2cm)},
    	y={(0:2cm)},
    	z={(120:2cm)},
        scale=0.82]

    \foreach \x in {1,...,358} %each is 3 degrees
    {
    %one rotation is 1/4cm, each degree is 1/4*1/360 cm
    \draw  (90-\x*3:2.125cm+\x*0.002083cm) to [bend left=1.5] (87-\x*3:2.125cm+0.002083cm+\x*0.002083cm); %3 degree arc, half lost on each end
    }

    \draw (90:2cm) arc (90:450:2cm);
    \draw (90:3cm) arc (90:450:3cm);

    \draw (90:2cm) to (90-123*3:2.125cm+123*0.002083cm);
    \draw (90:3cm) to (90-237*3:2.125cm+237*0.002083cm);
    
    \draw (330:2cm) to (90-40*3:2.125cm+40*0.002083cm);
    \draw (330:3cm) to (90-280*3:2.125cm+280*0.002083cm);
    
    \draw (210:2cm) to (90-80*3:2.125cm+80*0.002083cm);
    \draw (210:3cm) to (90-320*3:2.125cm+320*0.002083cm);

    \draw (90-119*3:2.125cm+119*0.002083cm) to (90-241*3:2.125cm+241*0.002083cm);
    
    \draw (90-60*3:2.125cm+60*0.002083cm) to (90-182*3:2.125cm+182*0.002083cm);
    \draw (90-300*3:2.125cm+300*0.002083cm) to (90-178*3:2.125cm+178*0.002083cm);

    \node[font=\scriptsize] at (90-118*3:1.94cm+118*0.002083cm) {1};
    \node[font=\scriptsize] at (90-362*3:2.06cm+362*0.002083cm) {2};
    \node[font=\scriptsize] at (90-125*3:2cm+125*0.002083cm) {6};
    \node[font=\scriptsize] at (90-255*3:2cm+255*0.002083cm) {7};
    \node[font=\scriptsize] at (90-30*3:2cm+30*0.002083cm) {2};
    \node[font=\scriptsize] at (90-177*3:2cm+177*0.002083cm) {6};
    \node[font=\scriptsize] at (90-297*3:2cm+297*0.002083cm) {7};
    \node[font=\scriptsize] at (90-60*3:2cm+60*0.002083cm) {3};
    \node[font=\scriptsize] at (90-420*3:2cm+420*0.002083cm) {3};
    \node[font=\scriptsize] at (90-184*3:2cm+184*0.002083cm) {5};
    \node[font=\scriptsize] at (90-310*3:2cm+310*0.002083cm) {4};
    \node[font=\scriptsize] at (90-450*3:2cm+450*0.002083cm) {1};
    \node[font=\scriptsize] at (90-225*3:2cm+225*0.002083cm) {5};
    \node[font=\scriptsize] at (90-340*3:2cm+340*0.002083cm) {4};

    \end{tikzpicture}   
    \caption{Reproduction of Heawood's map on a torus from~1890. 
    The inner and outer circle are identified to produce a torus.}
    \label{fig_Heawood_newFigure}
\end{figure}

We have reproduced Heawood's original map in~\Cref{fig_Heawood_newFigure}. 
The original source can be accessed from the SUB G\"ottinger Digitalisierungszentrum as Fig. 16 in the scanned document: 

\url{https://gdz.sub.uni-goettingen.de/id/PPN600494829_0024?tify=%7B%22pages%22%3A%5B405%5D%2C%22pan%22%3A%7B%22x%22%3A0.41%2C%22y%22%3A0.544%7D%2C%22view%22%3A%22toc%22%2C%22zoom%22%3A0.621%7D}

A reproduction of Heawood's figure also appears in~\cite[Fig. 7.6]{biggs_graphtheory_1968}.

The fact that the number $N_p$ is tight for a genus $p$ orientable surface became known as the Heawood's Conjecture, and was finally proved in 1968~\cite{ringel_youngs_torusmap_1968}. 
The case of genus~$p=1$ (the torus) is often refereed to as the seven color theorem. 
For more details and historical information about these map coloring problems we refer to~\cite{BookRingelMapColorTheorem} and~\cite[Chapters~6 and~7]{biggs_graphtheory_1968}.  

The more symmetric representations of Heawood's map in~\Cref{fig_Heawood_Figure_Leech_Ungar_Coxter} were described independently by John Leech in~\cite[Figures~2 and~3]{leech_seven_1955} and Peter Ungar in~\cite[Page~342]{Ungar_diagrams_1953}; Coxeter's bottom right figure~\cite[Figure~10]{coxeter_four_1959} is a 3D illustration of Ungar's model. These illustrations have inspired beautiful mathematical artwork models such as the ones presented in~\cite{goldstine_beading_2013,goldstine_capturing_2014,chas_crochet_2018,goldstine_eightheptagons_2020}.

\begin{figure}[htb]
    \begin{subfigure}[b]{0.45\textwidth}
        \centering
            \begin{tikzpicture}%
    	[x={(240:2cm)},
    	y={(0:2cm)},
    	z={(120:2cm)},
        scale=0.12]

    \coordinate (w1) at (3,0,0);    
    \coordinate (w2) at (0,3,0);    
    \coordinate (w3) at (0,0,3);   
    
    \def\metatilepositions{
    0/2/-1,
    1/0/-2,
    2/-1/0,
    0/0/0
    }
    
    \def\tilepositions{
    0/0/0/1,
    1/0/0/5,
    0/1/0/2,
    1/1/0/6,
    0/0/1/3,    
    0/1/1/4,    
    1/0/1/7
    }
    
    \foreach \e/\f/\g in \metatilepositions{
        \begin{scope}[shift={($\e*(3,0,0)+\f*(0,3,0)+\g*(0,0,3)$)}]
    \foreach \a/\b/\c/\d in \tilepositions{
        \begin{scope}[shift={($\a*(3,0,0)+\b*(0,3,0)+\c*(0,0,3)$)}]
            \draw[thick,draw=blue] (1,2,3) -- (2,1,3) -- (3,1,2) -- (3,2,1) -- (2,3,1) -- (1,3,2) -- (1,2,3);      
            \node at (0,0,0) {\d};
        \end{scope}
        }      
        \draw[ultra thick,draw=black] (1,0,5) -- (0,1,5) -- (0,2,4) -- (-1,3,4) --  (-1,4,3) -- (0,4,2) -- (0,5,1) -- (1,5,0) -- (2,4,0) -- (3,4,-1) -- (4,3,-1) -- (4,2,0) -- (5,1,0) -- (5,0,1) -- (4,0,2) -- (4,-1,3) -- (3,-1,4) -- (2,0,4) -- (1,0,5);
        \end{scope} 
    }
    \end{tikzpicture}%
        \caption{Leech~\cite[Figure~2]{leech_seven_1955}}
        \label{fig_Heawood_Figure_Leech}
    \end{subfigure}    
    \begin{subfigure}[b]{0.45\textwidth}
        \centering
        \begin{tikzpicture}[scale=0.18,every node/.style={font=\tiny}]
\draw (0,0) -- (28,0) -- (28,9) -- (0,9) -- (0,0);
\draw (28,4.5) -- (25,3) -- (27,6) -- (21,3) -- (23,6) -- (17,3) -- (19,6) -- (13,3) -- (15,6) -- (9,3) -- (11,6) -- (5,3) -- (7,6) -- (1,3) -- (3,6) -- (0,4.5);
\draw (27,6) -- (28,7);
\draw (23,6) -- (26,9);
\draw (19,6) -- (22,9);
\draw (25,3) -- (22,0);
\draw (15,6) -- (18,9);
\draw (21,3) -- (18,0);
\draw (11,6) -- (14,9);
\draw (17,3) -- (14,0);
\draw (7,6) -- (10,9);
\draw (13,3) -- (10,0);
\draw (3,6) -- (6,9);
\draw (9,3) -- (6,0);
\draw (5,3) -- (2,0);
\draw (1,3) -- (0,2);
\draw (26,0) -- (28,2);
\draw (0,7) -- (2,9);
\node at (2,1.8){1};
\node at (2,7.2){2};
\node at (6,1.8){2};
\node at (6,7.2){3};
\node at (10,1.8){3};
\node at (10,7.2){4};
\node at (14,1.8){4};
\node at (14,7.2){5};
\node at (18,1.8){5};
\node at (18,7.2){6};
\node at (22,1.8){6};
\node at (22,7.2){7};
\node at (26,1.8){7};
\node at (26,7.2){1};

\node[font=\tiny] at (0.5,8.4){1};
\node[font=\tiny] at (27.5,0.6){1};
\node[font=\tiny] at (0.5,3.8){7};
\node[font=\tiny] at (27.5,5.2){2};
\end{tikzpicture}
        \caption{Leech~\cite[Figure~3]{leech_seven_1955}}
    \end{subfigure}
    \begin{subfigure}[b]{0.45\textwidth}
        \centering
        \begin{tikzpicture}[scale=0.38,every node/.style={font=\tiny}]
\def\xstep{14/19}
\def\ystep{10/19}

\draw (0,0) -- (14,0) -- (14,10) -- (0,10) -- (0,0);
\draw (0,10) -- (5,0);
\draw (2,10) -- (7,0);
\draw (4,10) -- (9,0);
\draw (6,10) -- (11,0);
\draw (8,10) -- (13,0);
\draw (10,10) -- (14,2);
\draw (12,10) -- (14,6);
\draw (0,2) -- (1,0);
\draw (0,6) -- (3,0);
\node at (14-1*\xstep,10-1*\ystep){i};
\node at (14-3*\xstep,10-3*\ystep){v};
\node at (14-5*\xstep,10-5*\ystep){r};
\node at (14-7*\xstep,10-7*\ystep){o};
\node at (14-9*\xstep,10-9*\ystep){y};
\node at (14-11*\xstep,10-11*\ystep){g};
\node at (14-13*\xstep,10-13*\ystep){b};
\node at (14-15*\xstep,10-15*\ystep){i};
\node at (14-17*\xstep,10-17*\ystep){v};
\node at (14-18.5*\xstep,10-18.5*\ystep){r};
\end{tikzpicture}
        \caption{Ungar~\cite[Page~342]{Ungar_diagrams_1953}}
    \end{subfigure}
    \begin{subfigure}[b]{0.45\textwidth}
        \centering
        \begin{tikzpicture}[scale=0.45,every node/.style={font=\tiny}]

\def\fo{1/14}
        \begin{scope}[shift={(3,0)}]
        \draw (1,0) arc (0:360:1);
        \draw (3,0) arc (0:360:3);
        \draw[variable=\t,domain=0:0.5,samples=40] plot ({(2+cos(\t*360))*cos(\t*360*5*\fo+7*\fo*360+0*\fo*360)},{(2+cos(\t*360))*sin(\t*360*5*\fo+0*\fo*360)});
        \draw[variable=\t,domain=0:0.5,samples=40] plot ({(2+cos(\t*360))*cos(\t*360*5*\fo+7*\fo*360+2*\fo*360)},{(2+cos(\t*360))*sin(\t*360*5*\fo+2*\fo*360)});
        \draw[variable=\t,domain=0:0.5,samples=40] plot ({(2+cos(\t*360))*cos(\t*360*5*\fo+7*\fo*360+4*\fo*360)},{(2+cos(\t*360))*sin(\t*360*5*\fo+4*\fo*360)});
        \draw[variable=\t,domain=0:0.5,samples=40] plot ({(2+cos(\t*360))*cos(\t*360*5*\fo+7*\fo*360+6*\fo*360)},{(2+cos(\t*360))*sin(\t*360*5*\fo+6*\fo*360)});
        \draw[variable=\t,domain=0:0.5,samples=40] plot ({(2+cos(\t*360))*cos(\t*360*5*\fo+7*\fo*360+8*\fo*360)},{(2+cos(\t*360))*sin(\t*360*5*\fo+8*\fo*360)});
        \draw[variable=\t,domain=0:0.5,samples=40] plot ({(2+cos(\t*360))*cos(\t*360*5*\fo+7*\fo*360+10*\fo*360)},{(2+cos(\t*360))*sin(\t*360*5*\fo+10*\fo*360)});
        \draw[variable=\t,domain=0:0.5,samples=40] plot ({(2+cos(\t*360))*cos(\t*360*5*\fo+7*\fo*360+12*\fo*360)},{(2+cos(\t*360))*sin(\t*360*5*\fo+12*\fo*360)});
        \node at (180-0*\fo*360:2){g};
        \node at (180-2*\fo*360:2){y};
        \node at (180-4*\fo*360:2){o};
        \node at (180-6*\fo*360:2){r};
        \node at (180+2*\fo*360:2){b};
        \node at (180+4*\fo*360:2){i};
        \node at (180+6*\fo*360:2){v};
        \end{scope}
        \begin{scope}[shift={(-3,0)}]
        \draw (1,0) arc (0:360:1);
        \draw (2,0) arc (0:360:2);
        \draw (3,0) arc (0:360:3);
        \draw[variable=\t,domain=-0.25:0,samples=40] plot ({(2+cos(\t*360))*cos(\t*360*5*\fo+0*\fo*360)},{(2+cos(\t*360))*sin(\t*360*5*\fo+0*\fo*360)});
        \draw[variable=\t,domain=0.5:0.75,samples=40] plot ({(2+cos(\t*360))*cos(\t*360*5*\fo+0*\fo*360)},{(2+cos(\t*360))*sin(\t*360*5*\fo+0*\fo*360)});
        \draw[variable=\t,domain=-0.25:0,samples=40] plot ({(2+cos(\t*360))*cos(\t*360*5*\fo+2*\fo*360)},{(2+cos(\t*360))*sin(\t*360*5*\fo+2*\fo*360)});
        \draw[variable=\t,domain=0.5:0.75,samples=40] plot ({(2+cos(\t*360))*cos(\t*360*5*\fo+2*\fo*360)},{(2+cos(\t*360))*sin(\t*360*5*\fo+2*\fo*360)});
        \draw[variable=\t,domain=-0.25:0,samples=40] plot ({(2+cos(\t*360))*cos(\t*360*5*\fo+4*\fo*360)},{(2+cos(\t*360))*sin(\t*360*5*\fo+4*\fo*360)});
        \draw[variable=\t,domain=0.5:0.75,samples=40] plot ({(2+cos(\t*360))*cos(\t*360*5*\fo+4*\fo*360)},{(2+cos(\t*360))*sin(\t*360*5*\fo+4*\fo*360)});
        \draw[variable=\t,domain=-0.25:0,samples=40] plot ({(2+cos(\t*360))*cos(\t*360*5*\fo+6*\fo*360)},{(2+cos(\t*360))*sin(\t*360*5*\fo+6*\fo*360)});
        \draw[variable=\t,domain=0.5:0.75,samples=40] plot ({(2+cos(\t*360))*cos(\t*360*5*\fo+6*\fo*360)},{(2+cos(\t*360))*sin(\t*360*5*\fo+6*\fo*360)});
        \draw[variable=\t,domain=-0.25:0,samples=40] plot ({(2+cos(\t*360))*cos(\t*360*5*\fo+8*\fo*360)},{(2+cos(\t*360))*sin(\t*360*5*\fo+8*\fo*360)});
        \draw[variable=\t,domain=0.5:0.75,samples=40] plot ({(2+cos(\t*360))*cos(\t*360*5*\fo+8*\fo*360)},{(2+cos(\t*360))*sin(\t*360*5*\fo+8*\fo*360)});
        \draw[variable=\t,domain=-0.25:0,samples=40] plot ({(2+cos(\t*360))*cos(\t*360*5*\fo+10*\fo*360)},{(2+cos(\t*360))*sin(\t*360*5*\fo+10*\fo*360)});
        \draw[variable=\t,domain=0.5:0.75,samples=40] plot ({(2+cos(\t*360))*cos(\t*360*5*\fo+10*\fo*360)},{(2+cos(\t*360))*sin(\t*360*5*\fo+10*\fo*360)});
        \draw[variable=\t,domain=-0.25:0,samples=40] plot ({(2+cos(\t*360))*cos(\t*360*5*\fo+12*\fo*360)},{(2+cos(\t*360))*sin(\t*360*5*\fo+12*\fo*360)});
        \draw[variable=\t,domain=0.5:0.75,samples=40] plot ({(2+cos(\t*360))*cos(\t*360*5*\fo+12*\fo*360)},{(2+cos(\t*360))*sin(\t*360*5*\fo+12*\fo*360)});
        \node at (0*\fo*360:1.5){b};
        \node at (2*\fo*360:1.5){g};
        \node at (4*\fo*360:1.5){y};
        \node at (6*\fo*360:1.5){o};
        \node at (8*\fo*360:1.5){r};
        \node at (10*\fo*360:1.5){v};
        \node at (12*\fo*360:1.5){i};
        \node at (0*\fo*360:2.5){y};
        \node at (2*\fo*360:2.5){o};
        \node at (4*\fo*360:2.5){r};
        \node at (6*\fo*360:2.5){v};
        \node at (8*\fo*360:2.5){i};
        \node at (10*\fo*360:2.5){b};
        \node at (12*\fo*360:2.5){g};
        \end{scope}
\end{tikzpicture}
        \caption{Coxeter~\cite[Figure~10]{coxeter_four_1959}}
    \end{subfigure}
    \caption{Reproductions of more symmetric representations of Heawood's map.}
    \label{fig_Heawood_Figure_Leech_Ungar_Coxter}
\end{figure}
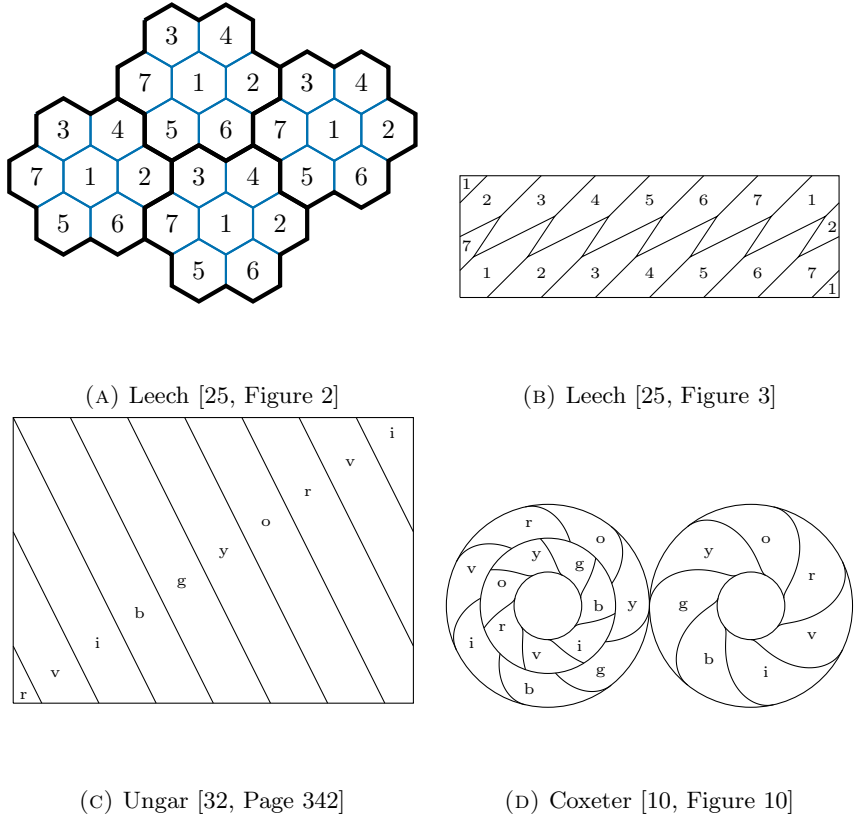
 
\begin{figure}[htb]%change to p when figures are back
    \begin{subfigure}[b]{0.48\textwidth}
    \centering
     HeawoodFigure111
    %\heawoodFigure{1}{1}{1}{0.20}
    \caption{$\kk=(1,1,1)$}
    \label{fig_heawoods_dimension2_a}
    \end{subfigure}
    \begin{subfigure}[b]{0.48\textwidth}
    \centering
     HeawoodFigure121
    %\heawoodFigure{1}{2}{1}{0.20}
    \caption{$\kk=(1,2,1)$}
    \end{subfigure}
    \par\medskip \medskip
    \begin{subfigure}[b]{0.49\textwidth}
    \centering
     HeawoodFigure222
    %\heawoodFigure{2}{2}{2}{0.15}
    \caption{$\kk=(2,2,2)$}
    \end{subfigure}
    \begin{subfigure}[b]{0.49\textwidth}
    \centering
     HeawoodFigure223
    %\heawoodFigure{2}{2}{3}{0.15}
    \caption{$\kk=(2,2,3)$}
    \end{subfigure}
    \par\medskip \medskip 
    \begin{subfigure}[b]{0.49\textwidth}
    \centering
     HeawoodFigure312
    % \heawoodFigure{3}{1}{2}{0.15}
    \caption{$\kk=(3,1,2)$}
    \label{fig_heawoods_dimension2_e}
    \end{subfigure}
    \begin{subfigure}[b]{0.49\textwidth}
    \centering
     HeawoodFigure321
    %\heawoodFigure{3}{2}{1}{0.15}
    \caption{$\kk=(3,2,1)$}
    \end{subfigure}
    \caption{Examples of the Heawood graph $\heawoodGraph{\kk}$ in dimension 2. The opposite sides (with the same color) are identified, making this graph a toroidal graph. The torus is the gray hexagon with opposite edges identified.}
    \label{fig_heawoods_dimension2}
\end{figure}

The Heawood graph is defined as the graph of Heawood's map: its vertices are the common points of three pairwise adjacent regions, and the edges are the lines connecting these points. 
It is a toroidal and distance-transitive graph on 14 vertices and 21 edges.
Our favorite representation of Heawood's graph is Leech's highly symmetric representation using regular hexagons on~\Cref{fig_Heawood_Figure_Leech}.
Note that here, the graph is the graph induced by the edge graph of the seven hexagons inside one of the fundamental domains bounded by the black-bold lines, where the boundary is identified according to how these fundamental domains glue together. 
This identification is illustrated on~\Cref{fig_heawoods_dimension2_a}.

The main purpose of this paper is to introduce a generalization $\heawoodGraph{\kk}$ of Heawood's graph that extends Leech's representation. 
Our generalization is indexed by a sequence $\kk=(k_1,\dots,k_{d+1})\in \N^{d+1}$ of positive integers for some~$d\geq 2$, and recovers the classical Heawood graph when $\kk=(1,1,1)$. 
As in the classical case, we will show that $\heawoodGraph{\kk}$ is a toroidal graph which is naturally embedded in a~$d$-dimensional torus.

When there are three parameters, the generalized Heawood graph $\heawoodGraph{(k_1,k_2,k_3)}$ is a 2-dimensional generalization of the classical Heawood graph. 
It is obtained by gluing together $\prod (k_i+1) - \prod k_i$ regular hexagons: From a ``central" hexagon one adds $k_i$ hexagons pointing in the direction at angle $(i-2) \frac{2\pi}{3}$ for $i=1,2,3$; then fill the ``big hexagon'' that they generate with other small hexagons.
\footnote{The factor $(i-2)$ can be changed to $i$ and the result would be just a rotation giving the same graph. We have chosen a factor of $(i-2)$ by convenience to match notions arising from permutahedra later on.} 
Several examples are illustrated in~\Cref{fig_heawoods_dimension2}.
We also provide three different choices of fundamental domain in~\Cref{fig_fundamentaldomains}, where the torus can be visualized in its more common rectangular presentation.
\footnote{The generators of the middle parallelogramm are $(k_1+1)w_1-k_2w_2$ and $(k_2+1)w_2-k_3w_3$, using the notation of~\Cref{sec_affineArrangement_permutahedralTiling}. 
The generators of a fundamental Parallelepiped in higher dimensions are $(k_i+1)w_i-k_{i+1}w_{i+1}$ for $i=1,\dots,d$, see~\Cref{prop_torus} and its proof.}

\begin{figure}[htb]
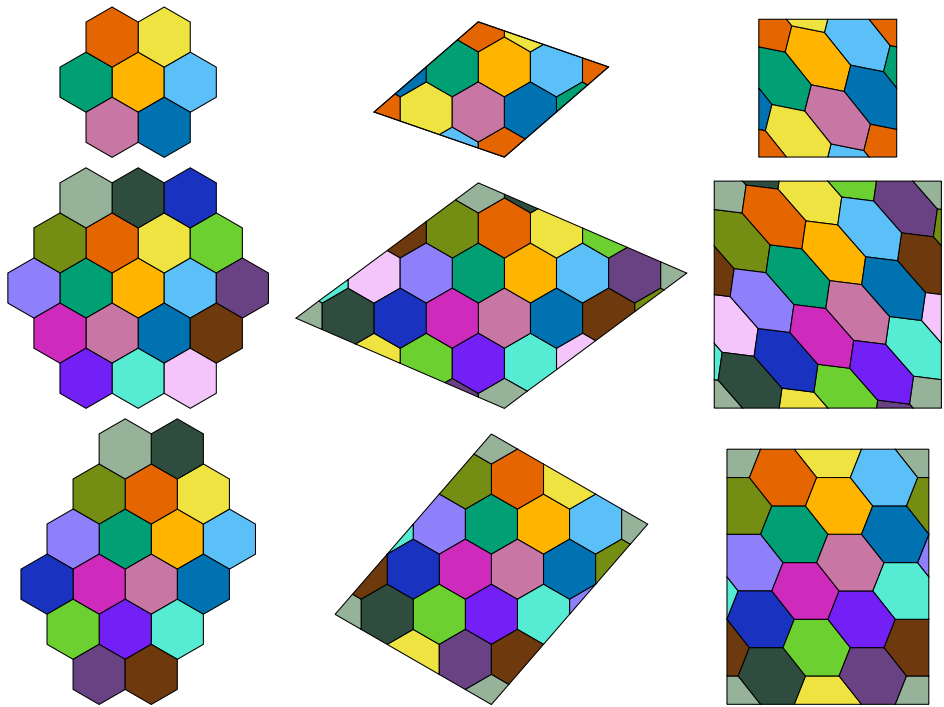

    \centering
    \include{figures/fig_fundamentaldomains}
    \caption{Different presentations of the fundamental domain for the Heawood graphs $\heawoodGraph{(1,1,1)}$, $\heawoodGraph{(2,2,2)}$ and %$\heawoodGraph{(1,2,3)}$
    $\heawoodGraph{(3,1,2)}$.}
    \label{fig_fundamentaldomains}
\end{figure}

The case $d=3$ gives 3-dimensional generalizations of the Heawood graph. The smallest choice of parameters is $\heawoodGraph{(1,1,1,1)}$, which is obtained by gluing $15=2^4-1^4$ polytopes that are 3-dimensional permutahedra, see~\Cref{fig_permutahedraltiling3D}.
The boundary of the result is identified to itself to form the complex into a 3-dimensional torus (see~\Cref{sec_heawoodComplex}).

\begin{figure}[htb]
    \centering
    % This is the figure which shows a 1111 permutahedra stacked together. Temporarily removed so that compile times are quicker.
    \include{figures/figure1111collapsed}
    \caption{The Heawood graph $\heawoodGraph{(1,1,1,1)}$ is the edge graph of this portion of the 3-dimensional permutahedral tiling after properly identifying its boundary (see~\Cref{sec_heawoodComplex}), making it into a toroidal graph.}
    \label{fig_permutahedraltiling3D}
\end{figure}

Our motivation for this work arose in our study of Gil Kalai's conjecture  on reconstruction of simplicial spheres~\cite{blind_puzzles_1987,kalai_blog_2009}.
The classical Heawood graph is the dual graph of a minimal triangulation of the 2-dimensional torus, which we used in~\cite{ceballosdoolittle_subword_2022} to provide examples of non-reconstructable triangulated 2-tori.  
Except for the triangulation dual to the classical Heawood graph $\heawoodGraph{(1,1,1)}$, the techniques from~\cite{ceballosdoolittle_subword_2022} can not be applied to any other $\heawoodGraph{\kk}$, because their automorphism group coincides with the automorphism group of their dual graph (the generalized Heawood graph). 

One special object of interest is the dual triangulation of $\heawoodGraph{(1,1,\dots,1)}$.
This triangulation consists of $2^{d+1}-1$ vertices and appeared in the work of Wolfgang K\"uhnel and Gunter Lassmann from the 1980's in~\cite{kuhnel_lassmann_threedtorus_1984,kuhnel_lassmann_symmetrictorii_1988}.
Interestingly, it is conjectured to be a minimal triangulation of the $d$-dimensional torus~\cite[Conjecture 21]{lutz_triangulated_2005} (see \Cref{conj_minimaltriangulation}). 
This conjecture holds in dimensions $d=1,2$ ($d=1$ is trivial and $d=2$ corresponds to the known classical Heawood graph case), and is supported by strong computational evidence in dimension $d=3$~\cite{lutz_triangulated_2005}.

\section{The generalized Heawood graph}\label{sec_heawood_graph}
The purpose of this section is to introduce the
generalized Heawood graph $\heawoodGraph{\kk}$ indexed by a sequence $\kk=(k_1,\dots,k_{d+1})\in \N^{d+1}$ of positive integers for some~$d\geq 2$. 
This is obtained by making some identifications on an infinite {graph $\affineGraph{d}$}, which is the graph of the $d$-dimensional permutahedral tiling, similar to the \(2\)-dimensional approach in \cite{senechal_tilingtorus_1988}. 
Before explaining this identification, we provide a direct definition of the graph $\affineGraph{d}$. 
The connection to permutahedral tilings will come to light in later sections. 

The \defn{graph $\affineGraph{d}$} is an infinite graph whose vertices $\vertices(\affineGraph{d})$ are the elements of the affine subspace 
\[
\{
\x=(x_1,\dots,x_{d+1}):\ x_1+\dots +x_{d+1}= 1+\dots +(d+1)
\}\subset \R^{d+1}
\]
%$\x=(x_1,\dots,x_{d+1})\in \Z^{d+1}$ 
whose coordinate entries are integers containing all the numbers $1,2,\dots, d+1$ mod~$(d+1)$. 
For instance, all permutations of $[d+1]$ satisfy this property. 
Two vertices $\x,\y$ of $\affineGraph{d}$ are connected by an edge if $\y-\x=e_j-e_i$ for some~$i\neq j$, where $e_1,\dots,e_{d+1}$ denote the standard basis vectors in $\R^{d+1}$.
\Cref{fig_permutahedral_tiling_graph_2} shows a portion of the graph of the permutahedral tiling in dimension $d=2$; the blue hexagon is the convex hull of all permutations of $[3]$. 

\begin{figure}[htb]
    \centering
    %This figure is of the 2d grid, labeled by extended permutations.
        \begin{tikzpicture}%
    	[x={(240:2cm)},
    	y={(0:2cm)},
    	z={(120:2cm)},
        scale=0.3]

    \coordinate (w1) at (3,0,0);    
    \coordinate (w2) at (0,3,0);    
    \coordinate (w3) at (0,0,3);    
        
    \def\tilepositions{
    0/0/0 ,
    1/0/0 ,
    2/0/0 ,
    0/1/0 ,
    1/1/0 ,
    2/1/0 ,
    0/2/0 ,
    1/2/0 ,
    2/2/0 ,
    0/3/0 ,
    1/3/0 ,
    2/3/0 ,
    0/0/1 ,    
    0/1/1 ,    
    0/2/1 ,    
    0/3/1 ,    
    0/0/2 ,    
    0/1/2 ,    
    0/2/2 ,    
    0/3/2 ,
    1/0/1 ,
    2/0/1 ,
    1/0/2 ,
    2/0/2 
    }

    \node (132) at (1,3,2) {};
    \node (213) at (2,1,3) {};
    \node (231) at (2,3,1) {};
    \node (312) at (3,1,2) {};
    \node (321) at (3,2,1) {};

    % Central hexagon
    \draw[draw=white, fill=blue!10] (1,2,3) -- (2,1,3) -- (3,1,2) -- (3,2,1) -- (2,3,1) -- (1,3,2) -- cycle;
    
    %fundamental tile:
    \foreach \a/\b/\c in \tilepositions{
        \begin{scope}[shift={($\a*(3,0,0)+\b*(0,3,0)+\c*(0,0,3)$)}]
            \node (123) at (1,2,3) {
            \pgfmathparse{int(1+2*\a-1*\b-1*\c)}\ifthenelse{-1 < \pgfmathresult}{\pgfmathresult}{\pgfmathparse{int(neg(1+2*\a-1*\b-1*\c))}$\bar\pgfmathresult$}%
            \pgfmathparse{int(2-1*\a+2*\b-1*\c)}\ifthenelse{-1 < \pgfmathresult}{\pgfmathresult}{\pgfmathparse{int(neg(2-1*\a+2*\b-1*\c))}$\bar\pgfmathresult$}%
            \pgfmathparse{int(3-1*\a-1*\b+2*\c)}\ifthenelse{-1 < \pgfmathresult}{\pgfmathresult}{\pgfmathparse{int(neg(3-1*\a-1*\b+2*\c))}$\bar\pgfmathresult$}%
            };
            \node (132) at (1,3,2) {
            \pgfmathparse{int(1+2*\a-1*\b-1*\c)}\ifthenelse{-1 < \pgfmathresult}{\pgfmathresult}{\pgfmathparse{int(neg(1+2*\a-1*\b-1*\c))}$\bar\pgfmathresult$}%
            \pgfmathparse{int(3-1*\a+2*\b-1*\c)}\ifthenelse{-1 < \pgfmathresult}{\pgfmathresult}{\pgfmathparse{int(neg(3-1*\a+2*\b-1*\c))}$\bar\pgfmathresult$}%
            \pgfmathparse{int(2-1*\a-1*\b+2*\c)}\ifthenelse{-1 < \pgfmathresult}{\pgfmathresult}{\pgfmathparse{int(neg(2-1*\a-1*\b+2*\c))}$\bar\pgfmathresult$}%
            };
            % \node (123) at (1,2,3) {};
            \node (213) at (2,1,3) {};
            % \node (132) at (1,3,2) {};
            \node (231) at (2,3,1) {};
            \node (312) at (3,1,2) {};
            \node (321) at (3,2,1) {};
        
            % \draw[blue] (123) -- (213) -- (312) -- (321) -- (231) -- (132) -- (123);      
            \draw[blue] (213) -- (123) -- (132) -- (231);      
        \end{scope}
        }            

    \end{tikzpicture}   
    \caption{The graph $\affineGraph{2}$ of the permutahedral tiling for $d=2$. 
    Commas and parenthesis are omited for simplicity. 
    An overlined number $\overline k$ represent the negative number $-k$. For instance, $\overline{1}43$ represents the vertex $(-1,4,3)$.}
    \label{fig_permutahedral_tiling_graph_2}
\end{figure}
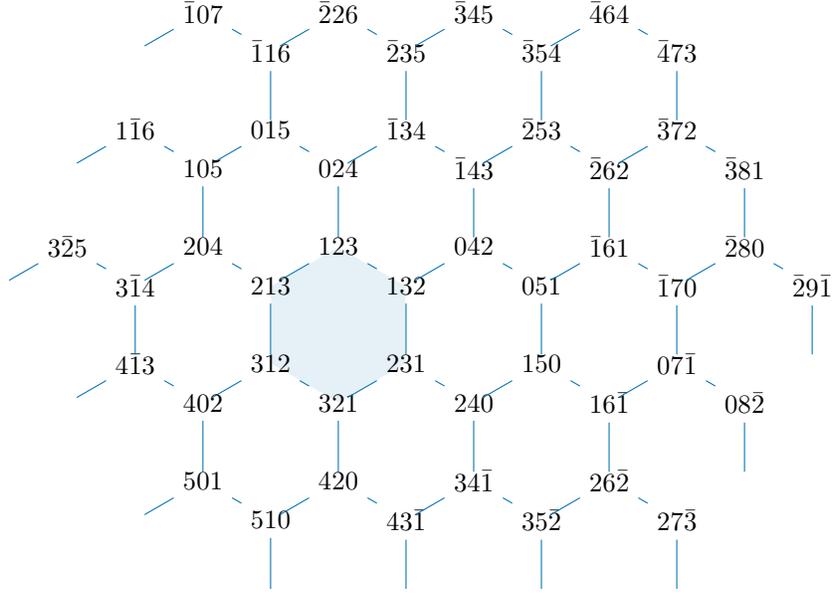

For $\kk=(k_1,\dots,k_{d+1})\in \N^{d+1}$ we denote by $M_\kk$ the matrix 

\begin{align}\label{eq_matrix}
M_\kk=
\begin{pmatrix}
    k_1+1  & -k_2 &  &  &&&\\ 
       & k_2+1 & -k_3 &  &&& \\
      &   & \ddots  & \ddots  &&&  \\
    &  &   & \ddots  & \ddots  &&  \\
   & &  &   & k_d+1  & -k_{d+1}    \\
   -k_1 &  &   &   && k_{d+1}+1    \\
\end{pmatrix}
\end{align}
and let $w_1,\dots,w_{d+1}\in \Z^{d+1}$ be the vectors
\begin{align}\label{eq_vector_wi}
w_i = (d+1)e_i - \sum_{j=1}^{d+1} e_j.
\end{align}
Equivalently, $w_i$ is the vector with $i$th coordinate equal to $d$ and all other coordinates equal to $-1$.

Note that if $\x\in \vertices(\affineGraph{d})$ then $\x+w_i\in \vertices(\affineGraph{d})$. 
Moreover, if $\x,\y\in\vertices(\affineGraph{d})$ are connected by an edge then $\x+w_i$ and $\y+w_i$ are connected by an edge as well. 
In other words, the graph $\affineGraph{d}$ is invariant under translations by the vectors~$w_1,\dots,w_{d+1}$.   

We denote by $\lattice{d}$ the lattice of integer linear combinations of the \(w_i\)
\begin{align}\label{eq_lattice_one}
\lattice{d} := 
\{
a_1w_1+\dots + a_{d+1}w_{d+1} :\ a_1,\dots,a_{d+1}\in 
\Z^{d+1}\},
\end{align}
and by $\sublattice{\kk}\subset \lattice{d}$ the sublattice 
\begin{align}\label{eq_sublattice}
\sublattice{\kk} := 
\left\{ 
a_1w_1+\dots + a_{d+1}w_{d+1} :\   
\begin{array}{l}
    (a_1,\dots,a_{d+1})=(b_1,\dots,b_{d+1}) M_{\kk} \\ 
    \text{for some } b_1,\dots,b_{d+1}\in \Z
\end{array}
\right\}.
\end{align}
That is, $\sublattice{\kk}$ is the set of linear combinations of $w_1,\dots ,w_{d+1}$ whose coefficient vector $(a_1,\dots,a_{d+1})$ is an integer linear combination of the rows of $M_\kk$. 

Using this sublattice, we can produce a finite graph by making some identifications on the vertices and edges of the infinite graph $\affineGraph{d}$. 

We say that $\x,\y\in \vertices(\affineGraph{d})$ are \defn{$\kk$-equivalent}, in which case we write $\x\equivalent{\kk} \y$, if 
\begin{align}
\y=\x+v \text{ for some } v\in \sublattice{\kk}.
\end{align}
Similarly, two edges of $\affineGraph{d}$ are {$\kk$-equivalent} if one is a translation of the other by a vector in $\sublattice{d}$. 

\begin{definition}[Generalized Heawood graph]
\label{def_heawoodGraph}
Let $\kk=(k_1,\dots,k_{d+1})\in \N^{d+1}$ be a sequence of positive integers for some $d\geq 2$. 
The \defn{Heawood graph} $\heawoodGraph{\kk}$ is the graph whose vertices and edges are the $\kk$-equivalent classes of vertices and edges of $\affineGraph{d}$, respectively. 
In other words,~$\heawoodGraph{\kk}$ is the graph obtained by identifying vertices and edges of $\affineGraph{d}$ up to translation by vectors in $\sublattice{\kk}$.     
\end{definition}

\begin{example}[Classical Heawood graph]
The classical Heawood graph is obtained when \(d=2\) and $\kk=(1,1,1)$, and is illustrated in~\Cref{fig_figure111}. 
The lattice $\lattice{2}$ consists of integer linear combinations of the vectors 
\begin{align*}
w_1&=(2,-1,-1) \\
w_2&=(-1,2,-1) \\
w_3&=(-1,-1,2)
\end{align*}
The associated matrix is
\begin{align*}
M_{(1,1,1)}=
\begin{pmatrix}
    2  & -1 & \\ 
       & 2 & -1 \\
     -1 &   & 2   \\
\end{pmatrix}
\end{align*}
The sublattice $\sublattice{(1,1,1)}$ consists of integer linear combinations of the rows of this matrix, when considered as vectors of coefficients of the $w_i$'s:
\begin{align*}
2w_1-w_2\\
2w_2-w_3\\
2w_3-w_1
\end{align*}
Figure~\ref{fig_figure111} shows a tiling of the plane, where each fundamental tile consists of seven hexagons: one hexagon in the center together with its six surrounding hexagons. The barycenters of the central hexagons correspond exactly to elements of the sublattice $\sublattice{(1,1,1)}$. 
The equivalence relation $\cong_\kk$ then identifies vertices and edges via translations that transform one fundamental tile into another.   
\end{example}

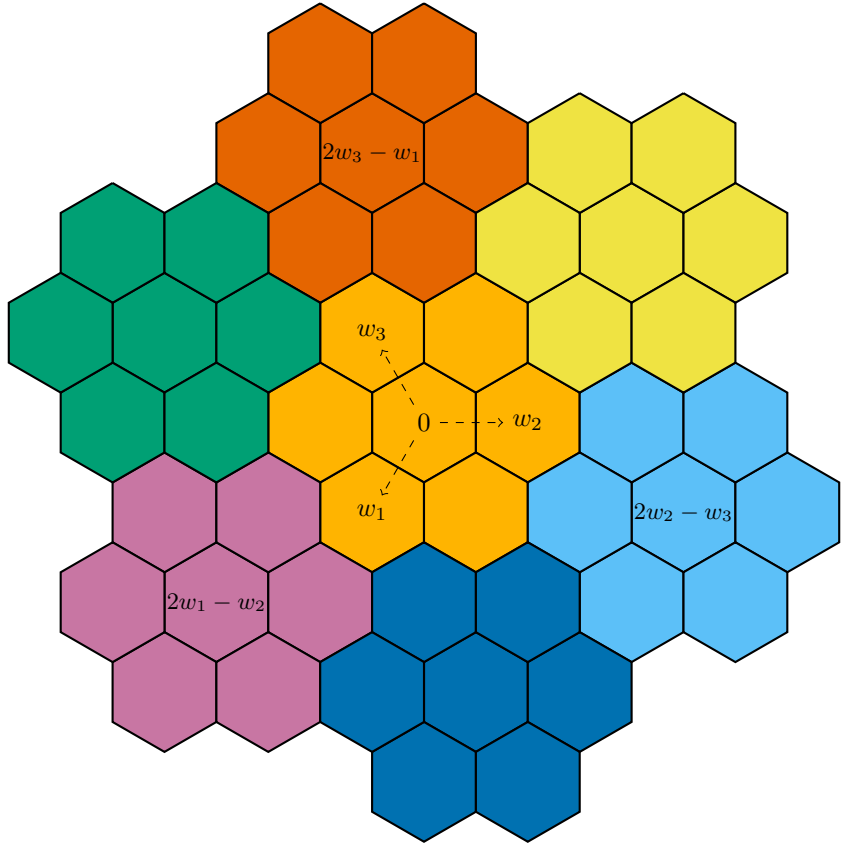
\begin{figure}[htb]
    \centering
    %7 copies of the fundamental tile
        \begin{tikzpicture}%
    	[x={(240:2cm)},
    	y={(0:2cm)},
    	z={(120:2cm)},
        scale=0.23]

    \coordinate (w1) at (3,0,0);    
    \coordinate (w2) at (0,3,0);    
    \coordinate (w3) at (0,0,3);    

    % \node at (0,0,0) {0};
    \draw (1,2,3) -- (2,1,3) -- (3,1,2) -- (3,2,1) -- (2,3,1) -- (1,3,2) -- (1,2,3); 

    \def\tilepositions{
    0/0/0,
    1/0/0,
    0/1/0,
    1/1/0,
    0/0/1,    
    0/1/1,    
    1/0/1
    }

    %fundamental tile:
    \foreach \a/\b/\c in \tilepositions{
        \begin{scope}[shift={($\a*(3,0,0)+\b*(0,3,0)+\c*(0,0,3)$)}]
            \draw[thick,draw=black,fill=orange] (1,2,3) -- (2,1,3) -- (3,1,2) -- (3,2,1) -- (2,3,1) -- (1,3,2) -- (1,2,3);      
            %\node at (0,0,0) {\a\b\c};
        \end{scope}
        }            

    % shifts of the tiles
    \def\tileshifts{%0/0/0/brown, 
    2/-1/0/purple, 0/2/-1/skyblue, -1/0/2/red, -2/1/0/yellow, 0/-2/1/bluishgreen, 1/0/-2/blue}
    
    \foreach \d\e\f\tilecolor in \tileshifts{
        \begin{scope}[shift={($\d*(3,0,0)+\e*(0,3,0)+\f*(0,0,3)$)}]        
        
            \foreach \a/\b/\c in \tilepositions{
                \begin{scope}[shift={($\a*(3,0,0)+\b*(0,3,0)+\c*(0,0,3)$)}]
                    \draw[thick,draw=black,fill=\tilecolor] (1,2,3) -- (2,1,3) -- (3,1,2) -- (3,2,1) -- (2,3,1) -- (1,3,2) -- (1,2,3);      
                    % \node at (0,0,0) {\a\b\c};
                \end{scope}
                }            
            % \node at (0,0,0) {\d\e\f};
            % \node at (3,0,0) {\tilecolor};
        \end{scope}
        }

    \node (n0) at (0,0,0) {$0$};
    \node (n1) at (w1) {$w_1$};
    \node (n2) at (w2) {$w_2$};
    \node (n3) at (w3) {$w_3$};
    \draw[dashed,->] (n0) -- (n1);
    \draw[dashed,->] (n0) -- (n2);
    \draw[dashed,->] (n0) -- (n3); 
        
    \node at ($2*(w1)-1*(w2)+0*(w3)$) {\small $2w_1-w_2$};
    \node at ($0*(w1)+2*(w2)-1*(w3)$) {\small $2w_2-w_3$};
    \node at ($-1*(w1)+0*(w2)+2*(w3)$) {\small $2w_3-w_1$};
    \end{tikzpicture}

% \definecolor{orange}{rgb}{0.898, 0.621, 0.0}
% \definecolor{skyblue}{rgb}{0.336, 0.703, 0.910}
% \definecolor{bluishgreen}{rgb}{0, 0.617, 0.449}
% \definecolor{yellow}{rgb}{0.937, 0.890, 0.258}
% \definecolor{blue}{rgb}{0, 0.445, 0.695}
% \definecolor{red}{rgb}{0.832, 0.367, 0}
% \definecolor{purple}{rgb}{0.797, 0.473, 0.652}
    \caption{The classical Heawood graph $\heawoodGraph{(1,1,1)}$ as a quotient of the graph of the permutahedral tiling in dimension two.}
    \label{fig_figure111}
\end{figure}

\begin{example}[$\kk=(2,3,2)$ for $d=2$]
Another example, where $\kk=(2,3,2)$, is illustrated in \Cref{fig_fundamentaltile_translates}. 
The vectors $w_1,w_2,w_3$ are the same as in the previous example, and the matrix is

\begin{align*}
M_{(2,3,2)}=
\begin{pmatrix}
    3  & -3 & \\ 
       & 4 & -2 \\
     -2 &   & 3   \\
\end{pmatrix}
\end{align*}
The sublattice $\sublattice{(2,3,2)}$ consists of the integer linear combinations of 
\begin{align*}
3w_1-3w_2\\
4w_2-2w_3\\
3w_3-2w_1
\end{align*}
Figure~\ref{fig_fundamentaltile_translates}  shows a tiling of the plane, where each fundamental tile gives a copy of the Heawood graph $\heawoodGraph{(2,3,2)}$, which is obtained by identifying the boundary of this tile to itself, via translations that transform one fundamental tile into another.  
\end{example}

Our aim is to prove some structural and enumerative properties of the generalized Heawood graph. Our first result is the following.

\begin{theorem}\label{thm_heawoodgraph}
    The generalized Heawood graph $\heawoodGraph{\kk}$ is a vertex-transitive
    graph with~$d! D_\kk$ many vertices and $\frac{(d+1)!}{2}D_\kk$ many edges, where
    \begin{align}\label{eq_determinat}
    D_{\kk}= \det M_\kk        
    = \prod (k_i+1) -\prod{k_i}. 
    \end{align}
\end{theorem}

Similarly to the classical case, the generalized Heawood graph is the dual graph of a triangulated torus, for which a simple combinatorial formula for its number of faces can be explicitly given.   

We denote by $\stirling{n}{k}$ the Stirling number of the second kind, which counts the number of ways to partition a set of $n$ objects into $k$ non-empty subsets. 
These numbers can be explicitly calculated as 
\begin{align}
    \stirling{n}{k} = 
    \frac{1}{k!} \sum_{i=1}^{k} (-1)^{k-i}{\binom{k}{i}} i^n. 
\end{align}

\begin{theorem}\label{thm_maintheorem}
%Let $\kk=(k_1,\dots,k_{d+1})\in \N^{d+1}$ for some $d\geq 2$.
The generalized Heawood graph $\heawoodGraph{\kk}$ is the dual graph of a triangulation of a $d$-dimensional torus with $f$-vector $(f_0,f_1,\dots ,f_d)$ determined by
\begin{align}
f_i=i! \, \stirling{d+1}{i+1} \, D_{\kk}.    
\end{align}
In particular, 
\begin{align}
    f_0 = D_\kk, \qquad
    f_d = d!\, D_\kk, \qquad
    f_{d-1} = \frac{(d+1)!}{2}\, D_\kk.
\end{align}
\end{theorem}

Table~\ref{table_modified_stirling} shows the factor
$c(i,d):=f_i/D_\kk$ 
for some small values. 

\begin{table}[htb]
    \centering
    \begin{tabular}{|l|cccccc|}
    \hline
    \diagbox{$d$}{$i$} & $0$ & $1$ & $2$ & $3$ & $4$ & $5$\\
    \hline
    $2$ & $1$ & $3$ & $2$ &&&\\
    $3$ & $1$ & $7$ & $12$  & $6$ &&\\
    $4$ & $1$ & $15$ & $50$ & $60$ & $24$ &\\
    $5$ & $1$ & $31$ & $180$ & $390$ & $360$ & $120$\\
    \hline
    \end{tabular}
    \caption{The factor $c(i,d)$ for some small values of $i$ and $d$.}
    \label{table_modified_stirling}
\end{table}

\begin{example}[$d=2$]
We consider the classical Heawood graph, when $\kk=(1,1,1)$. 
The factor $D_{(1,1,1)}=7$ counts the number of hexagons in~\Cref{fig_heawoods_dimension2_a}. The $f$-vector of its dual 2-dimensional triangulated torus is 
\[
(1\cdot 7, 3\cdot 7, 2\cdot 7)=(7,21,14).
\]
Interpreting this in the graph setting, we have 7 hexagons, 21 edges, and 14 vertices. 

When $d=2$, with a general \(\kk\), we have $D_\kk$ many hexagons, $3D_\kk$ many edges, and $2D_\kk$ many vertices. 
Table~\ref{table_counting_dim2} shows these numbers for all the examples in~\Cref{fig_heawoods_dimension2}.

\begin{table}[htb]
    \centering
    \begin{tabular}{|l|ccc|}
    \hline
    \diagbox{$\kk$}{$i$} & 0 & 1 & 2 \\
    \hline
    $(1,1,1)$ & $1\cdot 7$ & $3\cdot 7$ & $2\cdot 7$\\
    $(1,2,1)$ & $1\cdot 10$ & $3\cdot 10$ & $2\cdot 10$\\
    $(2,2,2)$ & $1\cdot 19$ & $3\cdot 19$ & $2\cdot 19$\\
    $(2,2,3)$ & $1\cdot 24$ & $3\cdot 24$ & $2\cdot 24$\\
    $(3,2,1)$ & $1\cdot 18$ & $3\cdot 18$ & $2\cdot 18$\\
    $(3,1,2)$ & $1\cdot 18$ & $3\cdot 18$ & $2\cdot 18$\\
    \hline
    \end{tabular}
    \caption{The number of hexagons, edges, and vertices for the Heawood graphs in~\Cref{fig_heawoods_dimension2}.}
    \label{table_counting_dim2}
\end{table}
\end{example}

\section{Background on the permutahedral tiling}

% Our generalizations of the Heawood graph will be obtained by making some identifications on the graph of a tiling of space with permutahedra. In this section, we start by providing some background. 
In order to prove these results, it is useful to have some further background about the permutahedral tiling. 

\subsection{The braid arrangement and the permutahedron}
For a fixed a positive integer $d\geq 2$, 
we consider the collection of hyperplanes
$\{ H_{ij} \}_{1\leq i < j \leq d+1}$ given by 
\[
H_{ij}= 
\{
\x\in \R^{d+1} :\ x_j-x_i=0
\}. 
\]

All these hyperplanes intersect in a one dimensional subspace spanned by the vector $(1,\dots,1)\in \R^{d+1}$, which is orthogonal to the $d$-dimensional subspace 
\[
V=\{\x\in \R^{d+1} :\ x_1+\dots x_{d+1}=0\} \subset \R^{d+1}.
\] 

The \defn{braid arrangement} $\HH_d$ is the restriction of the arrangement $\{ H_{ij} \}_{1\leq i < j \leq d+1}$ to $V$. 
An example in dimension $d=2$ is shown in Figure~\ref{fig_braid_arrangement}.

The hyperplane $H_{ij}$ can also be described as the orthogonal complement of the vector $e_j-e_i$.%, where $e_1,\dots ,e_{d+1}$ denote the standard basis vectors of $\R^{d+1}$. 
The collection $\Phi=\{e_j-e_i\}_{1\leq i\neq j \leq d+1}$ is known as the \defn{root system of type $A_d$}. 
It has a natural decomposition $\Phi=\Phi^+\bigsqcup \Phi^-$ into \defn{positive roots} and \defn{negative roots}, where~$\Phi^+$ (resp. $\Phi^-$) is the set of roots $e_j-e_i$ with $i<j$ (resp. $i>j$). 
The \defn{simple roots} are $\alpha_i=e_{i+1}-e_i$ for $i=1,\dots d$.

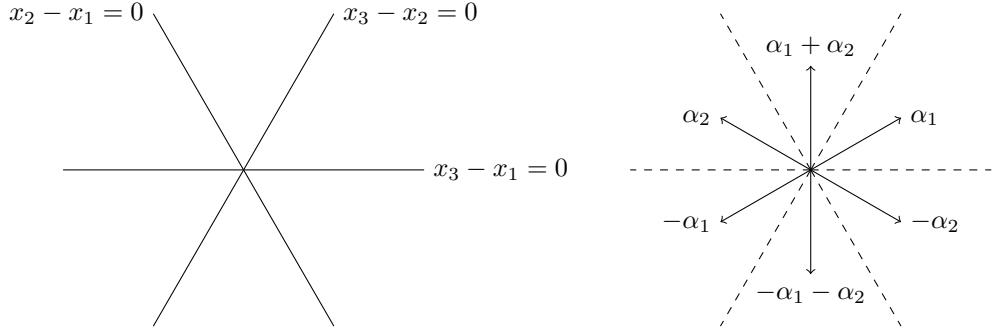
\begin{figure}[htb]
    \centering
       \begin{tabular}{cc}
    \begin{tikzpicture}%
    	[x={(240:2cm)},
    	y={(0:2cm)},
    	z={(120:2cm)},
        scale=0.4]
    
    \coordinate (0) at (0,0,0);
    \coordinate (w1) at (3,0,0);    
    \coordinate (w2) at (0,3,0);    
    \coordinate (w3) at (0,0,3);    
    \coordinate (-w1) at (-3,0,0);
    \coordinate (-w2) at (0,-3,0);    
    \coordinate (-w3) at (0,0,-3);   

    \coordinate (alpha1) at (-1,1,0);
    \coordinate (alpha2) at (0,-1,1);
    
     \draw (0) -- (w1);
     \draw (0) -- (w2);
     \draw (0) -- (w3);
     \draw (0) -- (-w1);
     \draw (0) -- (-w2);
     \draw (0) -- (-w3);
    
     \node[left] at (w3) {$x_2-x_1=0$};
     \node[right] at (-w1) {$x_3-x_2=0$};
     \node[right] at (w2) {$x_3-x_1=0$}; 

    \end{tikzpicture}  
    &
    \quad
    \begin{tikzpicture}
        	[x={(240:2cm)},
    	y={(0:2cm)},
    	z={(120:2cm)},
        scale=0.4]

    \coordinate (0) at (0,0,0);
    \coordinate (w1) at (3,0,0);    
    \coordinate (w2) at (0,3,0);    
    \coordinate (w3) at (0,0,3);    
    \coordinate (-w1) at (-3,0,0);
    \coordinate (-w2) at (0,-3,0);    
    \coordinate (-w3) at (0,0,-3);   

     \draw[dashed] (0) -- (w1);
     \draw[dashed] (0) -- (w2);
     \draw[dashed] (0) -- (w3);
     \draw[dashed] (0) -- (-w1);
     \draw[dashed] (0) -- (-w2);
     \draw[dashed] (0) -- (-w3);

    \coordinate (alpha1) at (-1,1,0);
    \coordinate (alpha2) at (0,-1,1);  
    \coordinate (alpha12) at ($(alpha1)+(alpha2)$);
    \coordinate (minusalpha1) at ($-1*(alpha1)$);
    \coordinate (minusalpha2) at ($-1*(alpha2)$);  
    \coordinate (minusalpha12) at ($-1*(alpha1)-1*(alpha2)$);

    \draw[->] (0) -- (alpha1);
    \draw[->] (0) -- (alpha2);
    \draw[->] (0) -- (alpha12);
    \draw[->] (0) -- (minusalpha1);
    \draw[->] (0) -- (minusalpha2);
    \draw[->] (0) -- (minusalpha12);

    \node[right] at (alpha1) {$\alpha_1$};    
    \node[left] at (alpha2) {$\alpha_2$};    
    \node[above] at (alpha12) {$\alpha_1+\alpha_2$};    
    \node[right] at (minusalpha2) {$-\alpha_2$};    
    \node[left] at (minusalpha1) {$-\alpha_1$};    
    \node[below] at (minusalpha12) {$-\alpha_1-\alpha_2$};    
    \end{tikzpicture}
    \end{tabular}
    \caption{The braid arrangement $\HH_2$ and the root system of type~$A_2$.}
    \label{fig_braid_arrangement}
\end{figure}

The braid arrangement $\HH_d$ decomposes the space $V$ into $(d+1)!$ connected components. 
Each of these components is a simplicial cone determined by a set of inequalities of the form 
\[
x_{i_1} \leq x_{i_2} \leq \dots \leq x_{i_{d+1}}, 
\]
where $[i_1,i_2,\dots,i_{d+1}]$ is a permutation of $[d+1]$.

More generally, the braid arrangement induces a simplicial cone decomposition of~$V$ consisting of all full dimensional cones mentioned above, together with all of their faces. 
The cones in this decomposition are naturally indexed by ordered partitions of $[d+1]$ as we now describe. 
An \defn{order partition} of $[d+1]$ is a tuple $[B_1,\dots ,B_k]$ such that $[d+1]=\bigsqcup_{i=1}^k B_i$, and the sets $B_i$ are called the \defn{blocks}.
The ordered partition $[B_1,\dots ,B_k]$
indexes the cone determined by the following set of equalities and inequalities
\begin{align*}
   x_i=x_{i'} & \quad \text{ if } i,i'\in B_j \text{ for some } j, \\
   x_i\leq x_{i'} & \quad \text{ if } i\in B_{j} \text{ and } i'\in B_{j'} \text{ with } j\leq j'.
\end{align*}
In particular, the ordered partition consisting of only one block corresponds to the origin, since  $x_1=\dots=x_{d+1}=0$; and the ordered partitions consisting of $d+1$ blocks correspond to permutations and index the full dimensional cones. 
In general, an ordered partition consisting of $k$ blocks corresponds to a $k-1$ dimensional cone.  

Furthermore, the cone associated to an ordered partition $\B =[B_1,\dots, B_k]$ is a face of the cone of another ordered partition $\B' =[B_1',\dots, B_\ell']$ if and only if~$\B$ is \defn{refined} by $\B'$, meaning that the blocks of $\B$ are unions of consecutive blocks of $\B'$. 

An example is illustrated in Figure~\ref{fig_braid_arrangement_cells}, where all the cells of the simplicial cone decomposition are labeled by ordered partitions.
For instance, the ordered partition $[2,3,1]$ corresponds to the 2-dimensional cone $x_2\leq x_3 \leq x_1$, while $[23,1]$ corresponds to the 1-dimensional cone $x_2=x_3\leq x_1$. Since $[23,1]$ is refined by~$[2,3,1]$ ($23$ is the union of the consecutive blocks $2$ and $3$), then the cone labeled $[23,1]$ is a face of the cone labeled $[2,3,1]$.
Note that we abuse of notation by writing~$23$ for the set $\{2,3\}$, and $1$ for the set $\{1\}$ for simplicity. We follow this convention throughout the paper.

\begin{figure}[htb]
    \centering
    \begin{tabular}{cc}
    \begin{tikzpicture}%
    	[x={(240:2cm)},
    	y={(0:2cm)},
    	z={(120:2cm)},
        scale=0.4]
    
    \coordinate (0) at (0,0,0);
    \coordinate (w1) at (3,0,0);    
    \coordinate (w2) at (0,3,0);    
    \coordinate (w3) at (0,0,3);    
    \coordinate (-w1) at (-3,0,0);
    \coordinate (-w2) at (0,-3,0);    
    \coordinate (-w3) at (0,0,-3);   
    
    \coordinate (alpha1) at (-1.7,1.7,0);
    \coordinate (alpha2) at (0,-1.7,1.7);  
    \coordinate (alpha12) at ($(alpha1)+(alpha2)$);
    \coordinate (minusalpha1) at ($-1*(alpha1)$);
    \coordinate (minusalpha2) at ($-1*(alpha2)$);  
    \coordinate (minusalpha12) at ($-1*(alpha1)-1*(alpha2)$);

    \node at (alpha12) {$[1,2,3]$};
    \node at (alpha2) {$[2,1,3]$};
    \node at (minusalpha1) {$[2,3,1]$};
    \node at (minusalpha12) {$[3,2,1]$};
    \node at (minusalpha2) {$[3,1,2]$};
    \node at (alpha1) {$[1,3,2]$};
    
    % \node (n0) at (0) {$[123]$};
    \node (n0) at (0) {};
    
     \draw (n0) -- (w1);
     \draw (n0) -- (w2);
     \draw (n0) -- (w3);
     \draw (n0) -- (-w1);
     \draw (n0) -- (-w2);
     \draw (n0) -- (-w3);
    \end{tikzpicture}  
         & 
    \qquad     
    \begin{tikzpicture}%
    	[x={(240:2cm)},
    	y={(0:2cm)},
    	z={(120:2cm)},
        scale=0.4]
    
    \coordinate (0) at (0,0,0);
    \coordinate (w1) at (3,0,0);    
    \coordinate (w2) at (0,3,0);    
    \coordinate (w3) at (0,0,3);    
    \coordinate (-w1) at (-3,0,0);
    \coordinate (-w2) at (0,-3,0);    
    \coordinate (-w3) at (0,0,-3);   

    \coordinate (alpha1) at (-1.85,1.85,0);
    \coordinate (alpha2) at (0,-1.85,1.85);  
    \coordinate (alpha12) at ($(alpha1)+(alpha2)$);
    \coordinate (minusalpha1) at ($-1*(alpha1)$);
    \coordinate (minusalpha2) at ($-1*(alpha2)$);  
    \coordinate (minusalpha12) at ($-1*(alpha1)-1*(alpha2)$);

    \node (n0) at (0) {$[123]$};
    % \node (n0) at (0) {};
    
     \draw (n0) -- (w1) node [midway,fill=white] {$[23,1]$};
     \draw (n0) -- (w2) node [midway,fill=white] {$[13,2]$};
     \draw (n0) -- (w3) node [midway,fill=white] {$[12,3]$};
     \draw (n0) -- (-w1) node [midway,fill=white] {$[1,23]$};
     \draw (n0) -- (-w2) node [midway,fill=white] {$[2,13]$};
     \draw (n0) -- (-w3) node [midway,fill=white] {$[3,12]$};

    \node at (minusalpha12) {};
    \end{tikzpicture}  
    \end{tabular}
        
    \caption{Ordered partition labeling of the cells of the simplicial cone decomposition determined by the braid arrangement $\HH_2$.} 
    % Left: the two dimensional cells correspond to permutations, and the zero dimensional cell to the partition consisting of one block. Right: the one dimensional cells.}
    \label{fig_braid_arrangement_cells}
\end{figure}
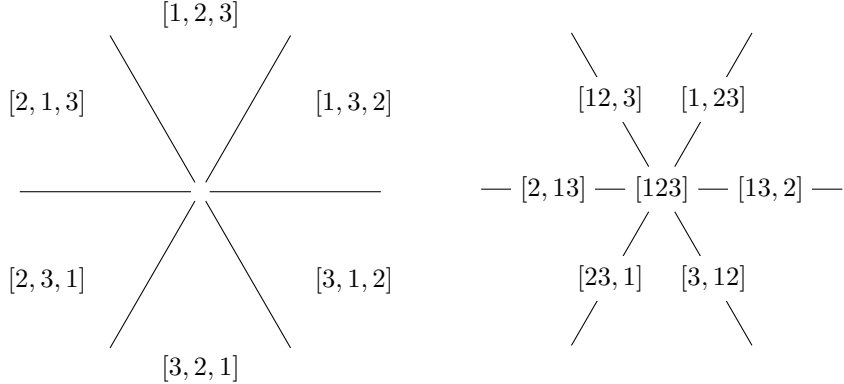

The braid arrangement $\HH_d$ is closely related to a polytope called the permutahedron.
The \defn{permutahedron} $\perm{d}$ is the convex hull of all permutations of $[d+1]$:
\[
\perm{d} =
\conv \{
(\sigma(1),\dots,\sigma({d+1})): \text{ for } \sigma \in \mathfrak S_{d+1} 
\} \subseteq \R^{d+1}.
\]
It is a $d$-dimensional polytope and has $(d+1)!$ vertices.
The facet description of the permutahedron $\perm{d}$ can be given as the subset of the affine hyperplane in~$\R^{d+1}$ determined by 
\[
x_1+\dots +x_{d+1}= 1+\dots +(d+1),
\]
that satisfies the inequalities
\begin{align}\label{eq_facets_permutahedron}
\sum_{a\in A} x_a \geq 1+\dots+|A|,    
\end{align}
for all non-empty proper subsets $A\subset [d+1]$. 

And example of the permutahedron $\perm{2}$ is illustrated in~\Cref{fig_perm_2D}. 
The vertex~$(3,1,2)$ is labeled $312$ for simplicity. 
We use this convention for the coordinates of the vertices in examples and figures throughout the paper. 
When needed, we may use round parenthesis $(\cdot)$ to represent coordinates, and square brackets $[\cdot]$ to represent ordered partitions and permutations in one-line notation, to avoid confusion.

\begin{figure}[htb]
\centering
\begin{tikzpicture}%
	[x={(240:2cm)},
	y={(0:2cm)},
	z={(120:2cm)},
    scale=0.35]

    \coordinate (0) at (0,0,0);
    \coordinate (w1) at (3,0,0);    
    \coordinate (w2) at (0,3,0);    
    \coordinate (w3) at (0,0,3);    
    \coordinate (-w1) at (-3,0,0);
    \coordinate (-w2) at (0,-3,0);    
    \coordinate (-w3) at (0,0,-3);   
    
    \coordinate (123) at (1,2,3);
    \coordinate (213) at (2,1,3);
    \coordinate (312) at (3,1,2);
    \coordinate (321) at (3,2,1);
    \coordinate (231) at (2,3,1);
    \coordinate (132) at (1,3,2);
    
    \node (n123) at (1,2,3) {123};
    \node (n213) at (2,1,3) {213};
    \node (n312) at (3,1,2) {312};
    \node (n321) at (3,2,1) {321};
    \node (n231) at (2,3,1) {231};
    \node (n132) at (1,3,2) {132};
    
    \draw (n123) -- (n213) -- (n312) -- (n321) -- (n231) -- (n132) -- (n123);

         \draw[dashed] (0) -- (w1);
         \draw[dashed] (0) -- (w2);
         \draw[dashed] (0) -- (w3);
         \draw[dashed] (0) -- (-w1);
         \draw[dashed] (0) -- (-w2);
         \draw[dashed] (0) -- (-w3);
         
\end{tikzpicture}    
% \quad
\begin{tikzpicture}%
	[x={(240:2cm)},
	y={(0:2cm)},
	z={(120:2cm)},
    scale=0.35]

\coordinate (0) at (0,0,0);
\coordinate (w1) at (3,0,0);    
\coordinate (w2) at (0,3,0);    
\coordinate (w3) at (0,0,3);    
\coordinate (-w1) at (-3,0,0);
\coordinate (-w2) at (0,-3,0);    
\coordinate (-w3) at (0,0,-3);    

\coordinate (123) at (1,2,3);
\coordinate (213) at (2,1,3);
\coordinate (312) at (3,1,2);
\coordinate (321) at (3,2,1);
\coordinate (231) at (2,3,1);
\coordinate (132) at (1,3,2);

\node (n123) at (1,2,3) {$[1,2,3]$};
\node (n213) at (2,1,3) {$[2,1,3]$};
\node (n312) at (3,1,2) {$[2,3,1]$};
\node (n321) at (3,2,1) {$[3,2,1]$};
\node (n231) at (2,3,1) {$[3,1,2]$};
\node (n132) at (1,3,2) {$[1,3,2]$};

\draw (n123) -- (n213) -- (n312) -- (n321) -- (n231) -- (n132) -- (n123);

\node at (w1) {};
\end{tikzpicture}   
% \quad
\begin{tikzpicture}%
	[x={(240:2cm)},
	y={(0:2cm)},
	z={(120:2cm)},
    scale=0.35]

\coordinate (0) at (0,0,0);
\coordinate (w1) at (3,0,0);    
\coordinate (w2) at (0,3,0);    
\coordinate (w3) at (0,0,3);    
\coordinate (-w1) at (-3,0,0);
\coordinate (-w2) at (0,-3,0);    
\coordinate (-w3) at (0,0,-3);    

\coordinate (123) at (1,2,3);
\coordinate (213) at (2,1,3);
\coordinate (312) at (3,1,2);
\coordinate (321) at (3,2,1);
\coordinate (231) at (2,3,1);
\coordinate (132) at (1,3,2);

\draw (123) -- (213) -- (312) -- (321) -- (231) -- (132) -- (123);

\node at ($0.75*(w1)$) {$[23,1]$};
\node at ($0.75*(w2)$) {$[13,2]$};
\node at ($0.75*(w3)$) {$[12,3]$};

\node at ($0.75*(-w1)$) {$[1,23]$};
\node at ($0.75*(-w2)$) {$[2,13]$};
\node at ($0.75*(-w3)$) {$[3,12]$};

\node at (0,0,0) {$[123]$};

\node at (w1) {};
\end{tikzpicture}   
    \caption{The permutahedron $\perm{2}$ and its normal fan (left). Its faces labeled by ordered partitions (middle and right).}
    \label{fig_perm_2D}
\end{figure}
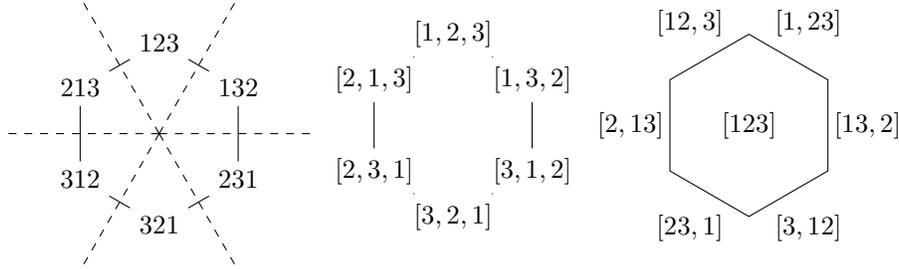

The non-empty faces of $\perm{d}$ are in bijection with ordered partitions of $[d+1]$.  
More precisely, an ordered partition $[B_1,\dots ,B_k]\subset [d+1]$ is the face of~$\perm{d}$ that satisfies 
\[
\sum_{a\in B_i} x_a = (b_{i-1}+1) + (b_{i-1}+2)+\dots + b_i, 
\]
for $1\leq i \leq k$, where $b_i:=|B_1|+\dots |B_i|$ and $b_0:=0$. 
Equivalently, $[B_1,\dots ,B_k]$ denotes the face given by the convex hull of all permutations $(x_1,\dots x_{d+1})$ of $[d+1]$ such that the set $\{x_a\}_{a\in B_i}=\{b_{i-1}+1,\dots , b_i\}$.

The ordered partition consisting of only one block corresponds to the full dimensional face (the polytope $\perm{d}$ itself). An ordered partition $[A,B]$ consisting of two blocks corresponds to the facet determined by~\eqref{eq_facets_permutahedron}. 
% And the ordered partitions whose blocks consist singleton elements correspond to the vertices of $\perm{d}$. Such ordered partitions can be thought as permutations in one line notation $[\sigma(1),\dots,\sigma(d+1)]$, and the corresponding vertex of $\perm{d}$ has coordinates given by the inverse permutation $(\sigma^{-1}(1),\dots , \sigma^{-1}(d+1))$. 
The ordered partitions whose blocks consist of singleton elements can be thought as permutations in one line notation, and correspond to the vertices of $\perm{d}$ whose coordinates are given by the inverse permutation: 

\smallskip
\begin{center}
\begin{tabular}{ccc}
    Ordered partition & $\longleftrightarrow$ & Vertex of $\perm{d}$\\
    $[\sigma(1),\dots,\sigma(d+1)]$ & $\longleftrightarrow$ & $(\sigma^{-1}(1),\dots , \sigma^{-1}(d+1))$
\end{tabular}
\smallskip
\end{center}
For instance, $[2,3,1]$ corresponds to the vertex $(3,1,2)$ determined by 
\[
x_2=1, \quad x_3=2, \quad x_1=3, 
\]
and $[23,1]$ corresponds to the facet satisfying 
\[
x_2+x_3=1+2, \quad x_1=3.
\]
This facet, an edge in this case, has two vertices $(3,1,2)$ and $(3,2,1)$, which are labeled by their inverse permutations $[2,3,1]$ and $[3,2,1]$ respectively. 

A face labeled by an ordered partition $\B =[B_1,\dots, B_k]$ is contained in a face labeled by an ordered partition $\B' =[B_1',\dots, B_\ell']$ if and only if~$\B$ \defn{refines} $\B'$.
%, just like the normal fan implies. 
For instance, the vertices $[2,3,1]$ and $[3,2,1]$ are contained in the edge $[23,1]$ because the block $23$ is the union of the first two blocks in both cases.

As a consequence, 
the normal fan of $\perm{d}$ is the simplicial cone decomposition determined by~$\HH_d$, and the face of $\perm{d}$ labeled by an ordered partition $[B_1,\dots , B_k]$ is dual to the cone with the same labeling. 
We also say the $\perm{d}$ is \defn{dual} to the simplicial cone decomposition determined by~$\HH_d$, in the sense that their posets of faces are dual to each other. 

Combinatorially, there is a small difference on how to describe the edges of the permutahedron depending on whether we use the coordinates of the vertices or the permutation of the corresponding ordered partition. 
In terms of vertex coordinates, two vertices of $\perm{d}$ are connected by an edge if they can be obtained from each other by swapping two consecutive \emph{values}:
\[
(\dots, k ,\dots , k+1,\dots ) \longleftrightarrow (\dots, k+1 ,\dots , k,\dots ),
\]
or equivalently, if the vector connecting one vertex to the other is of the form $e_j-e_i$ for some $i\neq j$.

In terms of ordered partitions, two permutations are connected by and edge in $\perm{d}$ if they can be obtained from each other by swapping two consecutive \emph{positions}:
\[
[\dots, i,j,\dots] \longleftrightarrow [\dots, j,i,\dots].
\]

\subsection{The affine arrangement and the permutahedral tiling}\label{sec_affineArrangement_permutahedralTiling}
Now, we consider the collection of affine hyperplanes 
\[
H^k_{ij}= 
\{
\x\in \R^{d+1} :\ x_j-x_i=k
\} 
\]
for $1\leq i < j \leq d+1$ and $k\in \Z$. It contains the hyperplanes $H^0_{ij}=H_{ij}$ of the braid arrangement and all their ``integer translates". 

The \defn{affine Coxeter arrangement} $\HHaffine_d$ of type $\widetilde A_d$ is the restriction of this arrangement to the hyperplane $V$.
For $d=2$, this is just the arrangement of affine hyperplanes of the triangular lattice, which is illustrated on the left of~\Cref{fig_triangularlattice_hexagonallattice}.

\begin{figure}[htb]
\centering
\begin{tabular}{cc}
    \begin{tikzpicture}%
    	[x={(240:2cm)},
    	y={(0:2cm)},
    	z={(120:2cm)},
        scale=0.15]
    
    \coordinate (0) at (0,0,0);
    \coordinate (w1) at (3,0,0);    
    \coordinate (w2) at (0,3,0);    
    \coordinate (w3) at (0,0,3);    
    \coordinate (-w1) at (-3,0,0);
    \coordinate (-w2) at (0,-3,0);    
    \coordinate (-w3) at (0,0,-3);    
    
    \def\boundaryvertices{
    %This is the list of boundary vertices of the the hexagon paired to draw the edges of the triangular tiling 
    %diagonals:
    3/0/3 /0/0/3,
    3/0/2 /0/1/3,
    3/0/1 /0/2/3,
    3/0/0 /0/3/3,
    3/1/0 /0/4/3,
    3/2/0 /0/4/2,
    3/3/0 /0/4/1,
    3/4/0 /0/4/0,
    %antidiagonals
    3/0/3 /3/0/0,
    2/0/3 /3/1/0,
    1/0/3 /3/2/0,
    0/0/3 /3/3/0,
    0/1/3 /3/4/0,
    0/2/3 /2/4/0,
    0/3/3 /1/4/0,
    0/4/3 /0/4/0,
    %horizontal
    3/0/0 /3/4/0,
    3/0/1 /2/4/0,
    3/0/2 /1/4/0,
    3/0/3 /0/4/0,
    2/0/3 /0/4/1,
    1/0/3 /0/4/2,
    0/0/3 /0/4/3
    }
    
    %draws translates of hyperplanes H_ij:
        \foreach \a/\b/\c /\d/\e/\f in \boundaryvertices{
            \draw[blue] ($\a*(w1)+\b*(w2)+\c*(w3)$) -- ($\d*(w1)+\e*(w2)+\f*(w3)$);
            }
    %draws hyperpalnes H_ij:
    \draw[ultra thick] ($3*(w1)+0*(w2)+0*(w3)$) -- ($0*(w1)+3*(w2)+3*(w3)$);
    \draw[ultra thick] ($3*(w1)+0*(w2)+3*(w3)$) -- ($0*(w1)+4*(w2)+0*(w3)$);
    \draw[ultra thick] ($0*(w1)+0*(w2)+3*(w3)$) -- ($3*(w1)+3*(w2)+0*(w3)$);
    
    \draw[ultra thick, blue] (w1) -- (-w2) -- (w3) -- (-w1) -- (w2) -- (-w3) -- cycle;
    %  \draw (0) -- (w1);
    %  \draw (0) -- (w2);
    %  \draw (0) -- (w3);
    %  \draw (0) -- (-w1);
    %  \draw (0) -- (-w2);
    %  \draw (0) -- (-w3);
    \node[fill=white, draw=black, circle, inner sep=0.5pt] at (w1) {$\widetilde w_1$};
    \node[fill=white, draw=black, circle, inner sep=0.5pt] at (w2) {$\widetilde w_2$};
    \node[fill=white, draw=black, circle, inner sep=0.5pt] at (w3) {$\widetilde w_3$};
    \end{tikzpicture} 
    &
    \begin{tikzpicture}%
    	[x={(240:2cm)},
    	y={(0:2cm)},
    	z={(120:2cm)},
        scale=0.15]

    \coordinate (w1) at (3,0,0);    
    \coordinate (w2) at (0,3,0);    
    \coordinate (w3) at (0,0,3);    
        
    \def\tilepositions{
    0/0/0 ,
    1/0/0 ,
    2/0/0 ,
    0/1/0 ,
    1/1/0 ,
    2/1/0 ,
    0/2/0 ,
    1/2/0 ,
    2/2/0 ,
    0/3/0 ,
    1/3/0 ,
    2/3/0 ,
    0/0/1 ,    
    0/1/1 ,    
    0/2/1 ,    
    0/3/1 ,    
    0/0/2 ,    
    0/1/2 ,    
    0/2/2 ,    
    0/3/2 ,
    1/0/1 ,
    2/0/1 ,
    1/0/2 ,
    2/0/2 
    }

    %fundamental tile:
    \foreach \a/\b/\c in \tilepositions{
        \begin{scope}[shift={($\a*(3,0,0)+\b*(0,3,0)+\c*(0,0,3)$)}]
            \draw[blue,fill=brown!15] (1,2,3) -- (2,1,3) -- (3,1,2) -- (3,2,1) -- (2,3,1) -- (1,3,2) -- (1,2,3);      
        \end{scope}
        }            

    % middle hexagon:
        \draw[ultra thick,blue] (1,2,3) -- (2,1,3) -- (3,1,2) -- (3,2,1) -- (2,3,1) -- (1,3,2) -- (1,2,3);      
    
    \node (n0) at (0,0,0) {$0$};
    \node (n1) at (w1) {$w_1$};
    \node (n2) at (w2) {$w_2$};
    \node (n3) at (w3) {$w_3$};
    \draw[dashed,->] (n0) -- (n1);
    \draw[dashed,->] (n0) -- (n2);
    \draw[dashed,->] (n0) -- (n3);   

    \node at ($3*(w1)$) {};
    
    \end{tikzpicture}   
\end{tabular}
    \caption{A finite piece of the simplicial complex $\complexAffine{2}$ determined by the affine Coxeter arrangement of type $\widetilde A_2$ (left). A finite piece of its dual tiling of space by permutahedra $\permutahedralTiling{2}$ (right).}
    \label{fig_triangularlattice_hexagonallattice}
\end{figure}
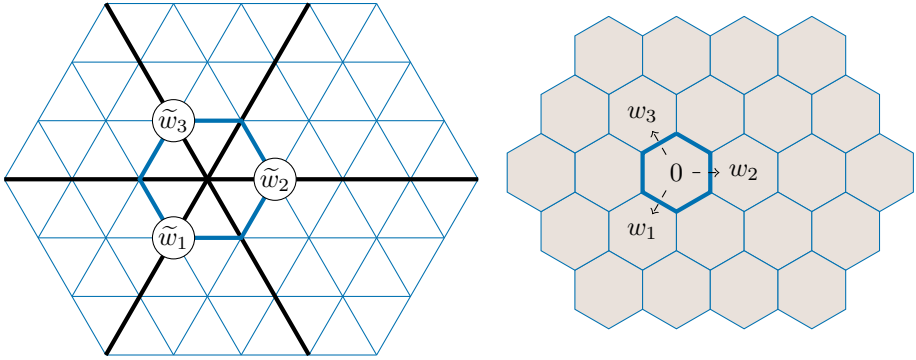

In general, the arrangement $\HHaffine_d$ decomposes the space $V$ into an infinite number of simplices, giving rise to an infinite simplicial complex that we denote by $\complexAffine{d}$. The vertices of this complex are the elements of
\begin{align}\label{eq_lattice_affine_one}
\latticeAffine{d} :=
\{
\x\in V :\ x_j-x_i \in \Z, \text{ for all } 1\leq i < j \leq d+1 
\}.    
\end{align}
This set is a lattice, which is known as the \defn{weight lattice} of type $A_d$. It is generated by integer linear combinations of the vectors $\widetilde w_1, \dots ,\widetilde w_{d+1}$ given by
\begin{align}
    (d+1) \widetilde w_i = (d+1) e_i - \sum_{j=1}^{d+1} e_j,
\end{align}
which satisfy the relation
\begin{align}
\widetilde w_1 + \dots + \widetilde w_{d+1} = 0.    
\end{align}
The weight lattice is isomorphic to the quotient 
\begin{align}
    \latticeAffine{d} \cong \Z^{d+1}/ (1,\dots,1).
\end{align}

For instance, for $d=2$ the vertices of the complex $\complexAffine{2}$ are the vertices of the triangular tiling, and are spanned by the vectors
\begin{align*}
    \widetilde w_1 &= \frac{1}{3} (2,-1,-1) \\
    \widetilde w_2 &= \frac{1}{3} (-1,2,-1) \\
    \widetilde w_3 &= \frac{1}{3} (-1,-1,2)
\end{align*}
subject to the relation $\widetilde w_1+\widetilde w_2+\widetilde w_3=0$, see~\Cref{fig_triangularlattice_hexagonallattice} (left).

For the case $d=2$, the dual of the triangular lattice is the hexagonal lattice, which is illustrated on the right of~\Cref{fig_triangularlattice_hexagonallattice}.
This generalizes to higher dimensions, where the dual of the complex $\complexAffine{d}$ is a combinatorial structure known as the permutahedral tiling. 
We will next describe this tilining geometrically.

In order to define the permutahedral tiling, we denote by 
\begin{align}\label{eq_lattice_two}
\lattice{d}:=(d+1)\latticeAffine{d}    
=
\{
\x\in V :\ x_j-x_i \in (d+1)\Z, \text{ for all } 1\leq i < j \leq d+1 
\}
\end{align}
the $(d+1)$-dilation of the lattice $\latticeAffine{d}$. It is generated by integer linear combinations of $w_1,\dots,w_{d+1}$, where 
\begin{align}
    w_i := (d+1)\widetilde w_i = (-1,\dots, -1, d, -1, \dots, -1)
\end{align}
is the vector with $i$th coordinate equal to $d$ and all other coordinates equal to $-1$. These vectors also satisfy the relation
\begin{align}
    w_1+ \dots + w_{d+1}=0.
\end{align}

The lattice $\lattice{d}$ is the same lattice we defined before in Equation~\eqref{eq_lattice_one}, and the vectors $w_i$ are the same vectors we defined in Equation~\eqref{eq_vector_wi}; see Section~\ref{sec_heawood_graph}.

The \defn{permutahedral tiling} $\permutahedralTiling{d}$ is the infinite tiling of the affine subspace
\begin{align}
\{
\x\in \R^{d+1}:\ x_1+\dots+x_{d+1}=1+\dots+(d+1)
\}    
\end{align}
whose tiles are translates
\begin{align}
    \perm{d} + v 
\end{align}
 of the permutahedron, for $v\in \lattice{d}$.
 An example for $d=2$ of this translation is illustrated on~\Cref{fig_perm_translation_2D}. A bigger portion of the tiling with more tiles is shown on the right of~\Cref{fig_triangularlattice_hexagonallattice}.
 In~\Cref{fig_permutahedron3D}, we illustrate the $d=3$ dimensional permutahedron together with the four translation directions $w_1,w_2,w_3,w_4$.

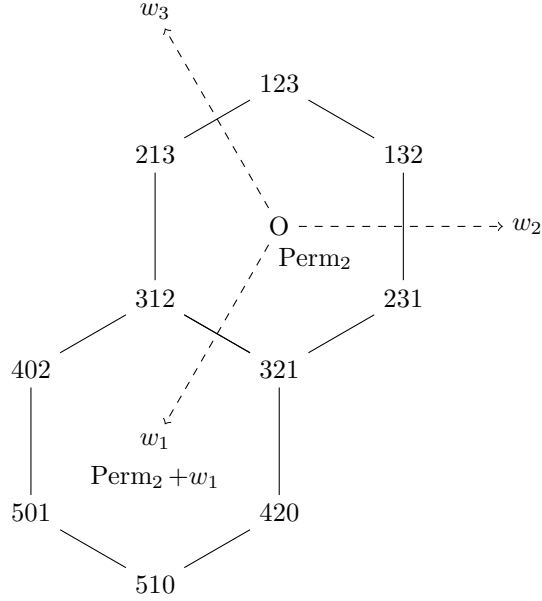
\begin{figure}[htb]
    \centering
\begin{tikzpicture}%
	[x={(240:2cm)},
	y={(0:2cm)},
	z={(120:2cm)},
    scale=0.55]

\coordinate (w1) at (3,0,0);    
\coordinate (w2) at (0,3,0);    
\coordinate (w3) at (0,0,3);    

\coordinate (123) at (1,2,3);
\coordinate (213) at (2,1,3);
\coordinate (312) at (3,1,2);
\coordinate (321) at (3,2,1);
\coordinate (231) at (2,3,1);
\coordinate (132) at (1,3,2);

\node (n123) at (1,2,3) {123};
\node (n213) at (2,1,3) {213};
\node (n312) at (3,1,2) {312};
\node (n321) at (3,2,1) {321};
\node (n231) at (2,3,1) {231};
\node (n132) at (1,3,2) {132};

\draw (n123) -- (n213) -- (n312) -- (n321) -- (n231) -- (n132) -- (n123);

\node (o) at (0,0,0) {O};
\node (n1) at (w1) {$w_1$};
\node (n2) at (w2) {$w_2$};
\node (n3) at (w3) {$w_3$};
\draw[dashed,->] (o) -- (n1);
\draw[dashed,->] (o) -- (n2);
\draw[dashed,->] (o) -- (n3);

\coordinate (123w1) at ($(123)+(w1)$);
\coordinate (213w1) at ($(213)+(w1)$);
\coordinate (312w1) at ($(312)+(w1)$);
\coordinate (321w1) at ($(321)+(w1)$);
\coordinate (231w1) at ($(231)+(w1)$);
\coordinate (132w1) at ($(132)+(w1)$);

\node (n213w1) at (213w1) {402};
\node (n312w1) at (312w1) {501};
\node (n321w1) at (321w1) {510};
\node (n231w1) at (231w1) {420};
\draw (n312) -- (n213w1) -- (n312w1) -- (n321w1) -- (n231w1) -- (n321) -- (n312);

\node at ($0.15*(w1)+0.22*(w2)+0*(w3)$) {$\perm{2}$};
\node[anchor=north, outer sep=6pt] at ($1*(w1)+0*(w2)+0*(w3)$) {$\perm{2}+w_1$};

\end{tikzpicture}    
    \caption{Example of two tiles in the permutahedral tiling in dimension 2.}
    \label{fig_perm_translation_2D}
\end{figure}

The poset of non-empty faces of the complex $\complexAffine{d}$ is dual to the poset of non-empty faces of the permutahedral tiling $\permutahedralTiling{d}$. A vertex of $\complexAffine{d}$ given by
\[
v=a_1 \widetilde w_1 + \dots a_{d+1}\widetilde w_{d+1} 
\]
corresponds to the full dimensional cell~$\perm{v}$ of $\permutahedralTiling{d}$, given by
\[
\perm{v}:=
\perm{d} + 
(a_1  w_1 + \dots a_{d+1} w_{d+1} 
)
=
\perm{d}+(d+1)v.
\]
The faces of $\complexAffine{d}$ that contain $v$ are dual to the faces of $\perm{v}$. 
In \Cref{fig_fundamentaltile_translates}, we  label $\perm{v}$ simply by \(v\), omitting parenthesis, omitting commas, and using a line over a number to represent its negative. 
For  example, \(101\) represents \(\perm{1\widetilde w_1+ 0 \widetilde w_2 +1\widetilde w_3}\) and \(3\overline{3}0\) represents \(\perm{3\widetilde w_1+ -3 \widetilde w_2 +0\widetilde w_3}\) .  

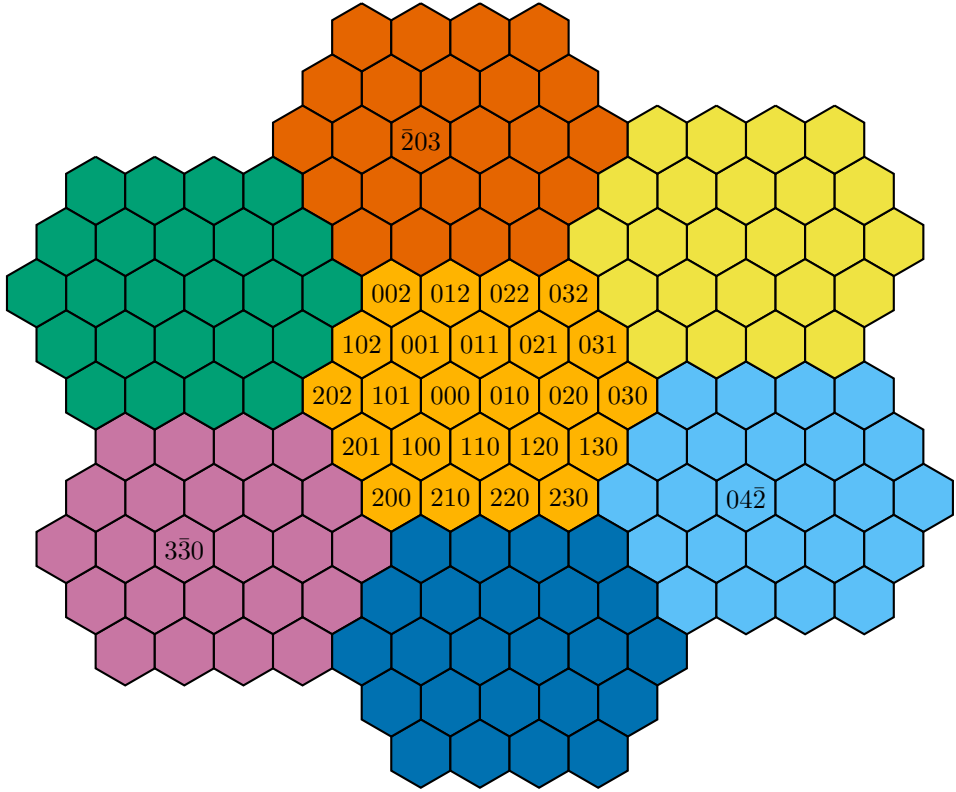
\begin{figure}[htb]
    \centering
        \begin{tikzpicture}%
    	[x={(240:2cm)},
    	y={(0:2cm)},
    	z={(120:2cm)},
        scale=0.131]

    \coordinate (w1) at (3,0,0);    
    \coordinate (w2) at (0,3,0);    
    \coordinate (w3) at (0,0,3);    

    % \node at (0,0,0) {0};
    \draw (1,2,3) -- (2,1,3) -- (3,1,2) -- (3,2,1) -- (2,3,1) -- (1,3,2) -- (1,2,3); 

    \def\tilepositions{
    0/0/0,
    1/0/0,
    2/0/0,
    0/1/0,
    1/1/0,
    2/1/0,
    0/2/0,
    1/2/0,
    2/2/0,
    0/3/0,
    1/3/0,
    2/3/0,
    0/0/1,    
    0/1/1,    
    0/2/1,    
    0/3/1,    
    0/0/2,    
    0/1/2,    
    0/2/2,    
    0/3/2,
    1/0/1,
    2/0/1,
    1/0/2,
    2/0/2 
    }

    %fundamental tile:
    \foreach \a/\b/\c in \tilepositions{
        \begin{scope}[shift={($\a*(3,0,0)+\b*(0,3,0)+\c*(0,0,3)$)}]
            \draw[thick,draw=black,fill=orange] (1,2,3) -- (2,1,3) -- (3,1,2) -- (3,2,1) -- (2,3,1) -- (1,3,2) -- (1,2,3);      
            \node at (0,0,0) {\a\b\c};
        \end{scope}
        }            

    % shifts of the tiles
    \def\tileshifts{%0/0/0/brown, 
    3/-3/0/purple, 0/4/-2/skyblue, -2/0/3/red, 3/1/-2/blue, -2/4/1/yellow, 1/-3/3/bluishgreen}
    
    \foreach \d\e\f\tilecolor in \tileshifts{
        \begin{scope}[shift={($\d*(3,0,0)+\e*(0,3,0)+\f*(0,0,3)$)}]        
        
            \foreach \a/\b/\c in \tilepositions{
                \begin{scope}[shift={($\a*(3,0,0)+\b*(0,3,0)+\c*(0,0,3)$)}]
                    \draw[thick,draw=black,fill=\tilecolor] (1,2,3) -- (2,1,3) -- (3,1,2) -- (3,2,1) -- (2,3,1) -- (1,3,2) -- (1,2,3);      
                    % \node at (0,0,0) {\a\b\c};
                \end{scope}
                }            
            % \node at (0,0,0) {\d\e\f};
            % \node at (3,0,0) {\tilecolor};
        \end{scope}
        }
    % \node at (0,0,0) {$000$};
    \node at ($3*(w1)-3*(w2)+0*(w3)$) {$3\bar 3 0$};
    \node at ($0*(w1)+4*(w2)-2*(w3)$) {$0 4\bar 2$};
    \node at ($-2*(w1)+0*(w2)+3*(w3)$) {$\bar 2 0 3$};
    \end{tikzpicture}
    \caption{The Heawood graph $\heawoodGraph{(2,3,2)}$ as a quotient of the graph of the permutahedral tiling in dimension two. 
    Translations of the fundamental tile $\fundamentaltile{(2,3,2)}$ give a tiling of space.}
    \label{fig_fundamentaltile_translates}
\end{figure}

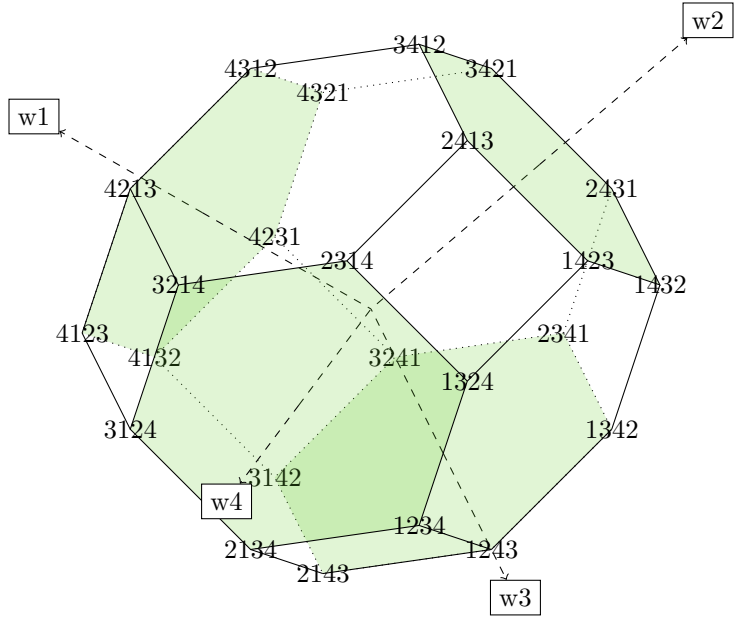
\begin{figure}[htb]
    \centering
\begin{tikzpicture} [scale=0.8]
    %tikz doesn't support 4d coordinate systems. So I pieced it together with this
    \node[coordinate] (p1) at (1,0){};
    \node[coordinate] (p2) at (0,1){};
    \node[coordinate] (p3) at (-0.4,-0.2){};
    \node[coordinate] (x1) at ($(p1)+(p2)+(p3)$){};
    \node[coordinate] (x2) at ($(p1)-(p2)-(p3)$){};
    \node[coordinate] (x3) at ($(p3)-(p1)-(p2)$){};
    \node[coordinate] (x4) at ($(p2)-(p1)-(p3)$){};
    %a sample node
    %\node (1100) at ($1*(x1)+1*(x2)+0*(x3)+0*(x4)$) {hi};

    %I cannot get this to work. I've tried more than I expected.
    %\newcommand{\cord}[4]{${#1}*(x1)+{#2}*(x2)+{#3}*(x3)+{#4}*(x4)$};
    
    \node[coordinate] (1234)  at ($1*(x1)+2*(x2)+3*(x3)+4*(x4)$) {1234};
    \node[coordinate] (1243)  at ($1*(x1)+2*(x2)+4*(x3)+3*(x4)$) {1243};
    \node[coordinate] (2143)  at ($2*(x1)+1*(x2)+4*(x3)+3*(x4)$) {2143};
    \node[coordinate] (2134)  at ($2*(x1)+1*(x2)+3*(x3)+4*(x4)$) {2134};
    \node[coordinate] (2314)  at ($2*(x1)+3*(x2)+1*(x3)+4*(x4)$) {2314};
    \node[coordinate] (2341)  at ($2*(x1)+3*(x2)+4*(x3)+1*(x4)$) {2341};
    \node[coordinate] (2431)  at ($2*(x1)+4*(x2)+3*(x3)+1*(x4)$) {2431};
    \node[coordinate] (2413)  at ($2*(x1)+4*(x2)+1*(x3)+3*(x4)$) {2413};
    \node[coordinate] (4213)  at ($4*(x1)+2*(x2)+1*(x3)+3*(x4)$) {4213};
    \node[coordinate] (4231)  at ($4*(x1)+2*(x2)+3*(x3)+1*(x4)$) {4231};
    \node[coordinate] (4321)  at ($4*(x1)+3*(x2)+2*(x3)+1*(x4)$) {4321};
    \node[coordinate] (4312)  at ($4*(x1)+3*(x2)+1*(x3)+2*(x4)$) {4312};
    \node[coordinate] (4132)  at ($4*(x1)+1*(x2)+3*(x3)+2*(x4)$) {4132};
    \node[coordinate] (4123)  at ($4*(x1)+1*(x2)+2*(x3)+3*(x4)$) {4123};
    \node[coordinate] (1423)  at ($1*(x1)+4*(x2)+2*(x3)+3*(x4)$) {1423};
    \node[coordinate] (1432)  at ($1*(x1)+4*(x2)+3*(x3)+2*(x4)$) {1432};
    \node[coordinate] (1342)  at ($1*(x1)+3*(x2)+4*(x3)+2*(x4)$) {1342};
    \node[coordinate] (1324)  at ($1*(x1)+3*(x2)+2*(x3)+4*(x4)$) {1324};
    \node[coordinate] (3124)  at ($3*(x1)+1*(x2)+2*(x3)+4*(x4)$) {3124};
    \node[coordinate] (3142)  at ($3*(x1)+1*(x2)+4*(x3)+2*(x4)$) {3142};
    \node[coordinate] (3412)  at ($3*(x1)+4*(x2)+1*(x3)+2*(x4)$) {3412};
    \node[coordinate] (3421)  at ($3*(x1)+4*(x2)+2*(x3)+1*(x4)$) {3421};
    \node[coordinate] (3241)  at ($3*(x1)+2*(x2)+4*(x3)+1*(x4)$) {3241};
    \node[coordinate] (3214)  at ($3*(x1)+2*(x2)+1*(x3)+4*(x4)$) {3214};

    \def\backedges{ (4123) / (4213) , (4321) / (4312) , (4132) / (4231) , (3142) / (3241)  }

    \def\frontedges{(2134) / (1234) , (2143) / (1243) , (2341) / (1342) , (2431) / (1432) , (3421) / (3412) , (3124) / (3214) , (2314) / (1324) , (2413) / (1423) , (2143) / (2134) , (2134) / (3124) , (2341) / (2431) , (2431) / (3421)}

\tikzset{vertex/.style={}}

    \node[fill=purple!0!white,draw=black] (w1) at ($3.5*(x1)+-0.5*(x2)+3.5*(x3)+3.5*(x4)$) {w1};
    \node[fill=purple!0!white,draw=black] (w3) at ($3.5*(x1)+3.5*(x2)+3.5*(x3)+-0.5*(x4)$) {w3};
    \node[coordinate] (1111) at ($2.5*(x1)+2.5*(x2)+2.5*(x3)+2.5*(x4)$) {};
    \node[coordinate] (hw1) at ($3*(x1)+1*(x2)+3*(x3)+3*(x4)$) {};
    \node[coordinate] (hw2) at ($3*(x1)+3*(x2)+1*(x3)+3*(x4)$) {};
    \node[coordinate] (hw3) at ($3*(x1)+3*(x2)+3*(x3)+1*(x4)$) {};
    \node[coordinate] (hw4) at ($1*(x1)+3*(x2)+3*(x3)+3*(x4)$) {};

    \draw[dashed,->] (hw1) -- (w1);
    \draw[dashed,->] (hw3) -- (w3);
    \draw[fill=green,fill opacity=0.2,dotted] (4321) -- (4231) -- (3241) -- (2341) -- (2431) -- (3421) -- (4321);%w1
    \draw[fill=green,fill opacity=0.2,dotted] (2134) -- (2143) -- (3142) -- (4132) -- (4123) -- (3124) -- (2134);%w3
    \draw[dashed] (hw3) -- (1111);
    \draw[dashed] (hw1) -- (1111);
    \draw[dashed] (hw2) -- (1111);
    \draw[dashed] (hw4) -- (1111);

    \foreach \x/\y in \backedges{
       \draw[dotted] \x -- \y;
    }
    
    \foreach \x/\y in \frontedges{
       \draw \x -- \y;
    }
\node[vertex] at (3142){4132};
\node[vertex] at (4132){4231};
\node[vertex] at (3241){3142};
\node[vertex] at (4231){3241};
\node[vertex] at (2134){4213};
\node[vertex] at (2143){4123};
\node[vertex] at (3124){4312};
\node[vertex] at (4123){4321};
\node[vertex] at (2341){2143};
\node[vertex] at (2431){1243};
\node[vertex] at (3421){1342};
\node[vertex] at (4321){2341};

    \draw[fill=green,fill opacity=0.2] (1234) -- (1243) -- (1342) -- (1432) -- (1423) -- (1324) -- (1234);%w4
    \draw[fill=green,fill opacity=0.2] (4312) -- (4213) -- (3214) -- (2314) -- (2413) -- (3412) -- (4312);%w2
    \node[fill=purple!0!white,draw=black] (w2) at ($3.5*(x1)+3.5*(x2)+-0.5*(x3)+3.5*(x4)$) {w2};
    \node[fill=purple!0!white,draw=black] (w4) at ($-0.5*(x1)+3.5*(x2)+3.5*(x3)+3.5*(x4)$) {w4};
    \draw[dashed,->] (hw2) -- (w2);
    \draw[dashed,->] (hw4) -- (w4);
    
\node[vertex] at (1234){3214};
\node[vertex] at (1243){3124};
\node[vertex] at (1324){2314};
\node[vertex] at (1342){2134};
\node[vertex] at (1423){1324};
\node[vertex] at (1432){1234};

\node[vertex] at (2314){2413};
\node[vertex] at (2413){1423};
\node[vertex] at (3214){3412};
\node[vertex] at (3412){1432};
\node[vertex] at (4213){3421};
\node[vertex] at (4312){2431};

\end{tikzpicture}
    \caption{The 3-dimensional permutahedron.}
    \label{fig_permutahedron3D}
\end{figure}

\begin{figure}[htb]
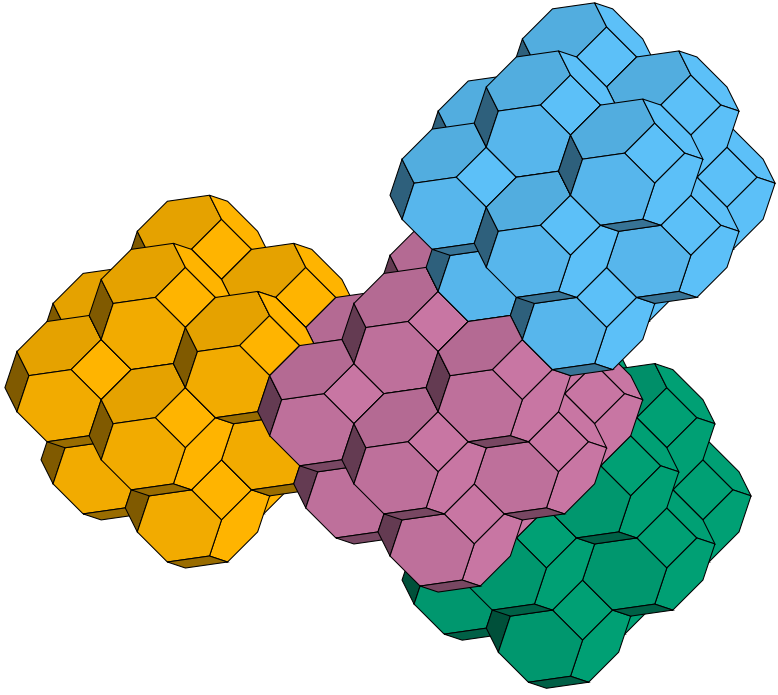

\centering
   %This is the figure which shows four 1111 permutahedra stacked together. Temporarily removed so that compile times are quicker.
 \include{figures/figureFour1111}
\caption{The fundamental tile  \(\fundamentaltile{(1,1,1,1)}\), and three translations of it with offsets $2w_1-w_2$, $2w_2-w_3$, and $2w_3-w_4$.}
\label{fig_figureFour1111}
\end{figure}

\begin{figure}[htb]
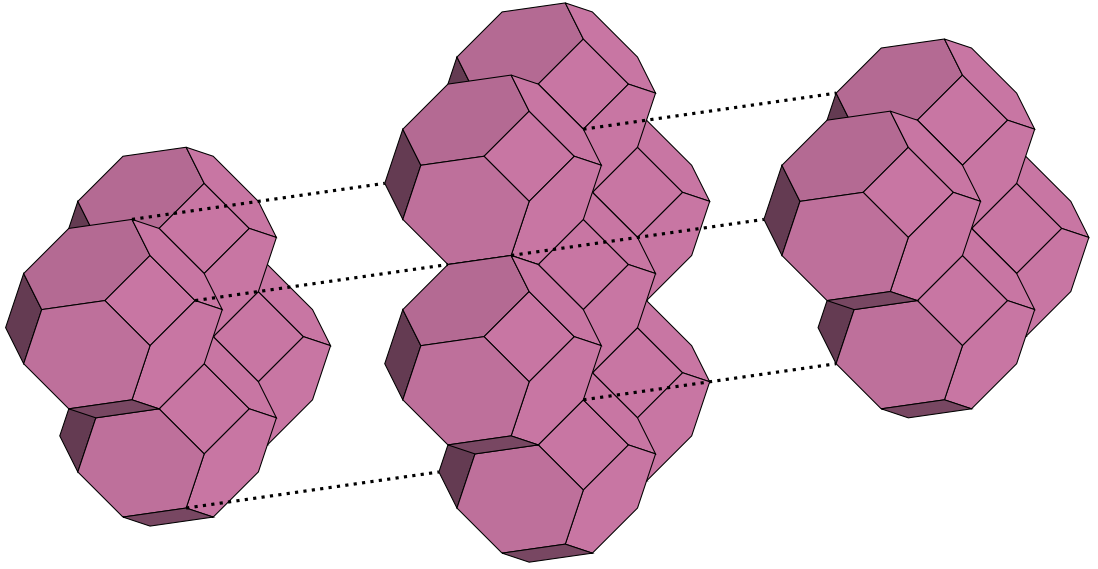

    \centering
%   This is the figure which shows an expanded 1111 permutahedra. Temporarily removed so that compile times are quicker.
 \include{figures/figure1111expand}
    \caption{The fundamental tile \(\fundamentaltile{(1,1,1,1)}\), expanded so that the interior is visible.}
    \label{fig_figure1111expand}
\end{figure}

\begin{figure}[htb]
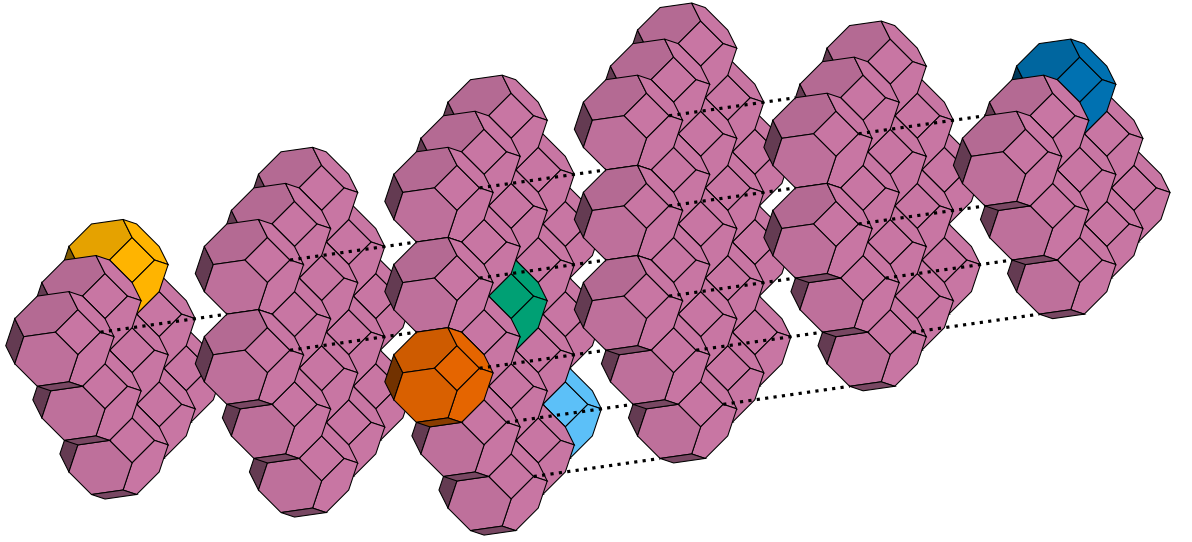

    \centering
%This is the figure which shows a 2223 collection of permutahedra. Temporarily removed so that compile times are quicker.
 \include{figures/figure2223expand}
    \caption{The fundamental tile \(\fundamentaltile{(2,2,2,3)}\), expanded so that the interior is visible. The differently colored permutahedra indicate the \(w_i\) directions.}
    \label{fig_figure2223expand}
\end{figure}

\section{The triangulated torus and the Heawood complex}
Using the machinery in the previous section, we can now proceed to introduce the triangulated torus stated in~\Cref{thm_maintheorem}, as a quotient of the infinite simplicial complex~$\complexAffine{d}$. 
Its dual is a polytopal complex that we call the Heawood complex, which can be obtained as a quotient of the infinite complex of faces of the permutahedral tiling~$\permutahedralTiling{d}$. 

\subsection{The triangulated torus}
As above, let $\kk = (k_1,k_2,\ldots,k_{d+1}) \in \N^{d+1}$ be a sequence of positive integers for some $d\geq 1$, and $M_\kk$ be the matrix in Equation~\eqref{eq_matrix}. 
Consider the lattice of integer linear combinations  
\begin{align}\label{eq_lattice_affine_two}
\latticeAffine{d} := 
\{
a_1\widetilde w_1+\dots + a_{d+1} \widetilde w_{d+1} :\ a_1,\dots,a_{d+1}\in 
\Z^{d+1}\}.
\end{align}
This is the weight lattice from Equation~\eqref{eq_lattice_affine_one}, and its elements are the vertices of the infinite simplicial complex $\complexAffine{d}$. 

We define $\sublatticeAffine{\kk}\subset \latticeAffine{d}$ as the \defn{sublattice}
\begin{align}\label{eq_sublatticeAffine}
\sublatticeAffine{\kk} := 
\left\{ 
a_1\widetilde w_1+\dots + a_{d+1}\widetilde w_{d+1} :\   
\begin{array}{l}
    (a_1,\dots,a_{d+1}) = (b_1,\dots,b_{d+1}) M_{\kk}  \\ 
    \text{for some } b_1,\dots,b_{d+1}\in \Z
\end{array}
\right\}.
\end{align}
Its elements are integer linear combinations of the vectors $\widetilde w_1,\dots,\widetilde w_{d+1}$, whose vector of coefficients is an integer linear combination of the rows of the matrix $M_\kk$. 

We say that two faces $F,F'\in \complexAffine{d}$ are \defn{$\kk$-equivalent}, and write \mbox{$F\equivalent{\kk} F'$}, if 
\begin{align}
F'=F+v \text{ for some } v\in \sublatticeAffine{\kk}.
\end{align}
That is, the face $F'$ can be obtained by translating the face $F$ by a vector $v\in \sublatticeAffine{d}$. 

\begin{definition}[The torus]
\label{def_torus}
Let $\kk=(k_1,\dots,k_{d+1})\in \N^{d+1}$ be a sequence of positive integers for some $d\geq 2$. The \defn{quotient complex} 
\[
\torus{\kk}= \complexAffine{d} / \sublatticeAffine{\kk}
\]
is the simplicial complex of $\kk$-equivalent classes of faces of $\complexAffine{d}$. 
In other words,~$\torus{\kk}$ is the simplicial complex obtained by identifying faces of $\complexAffine{d}$ up to translation by vectors in $\sublatticeAffine{\kk}$.   
A face class is contained in another if there is an element representative of the first that is contained in an element representative of the second.
\end{definition}

\begin{definition}(The fundamental vectors)
\label{def_fundamentalvectors}
For $\kk=(k_1,\dots,k_{d+1})\in \N^{d+1}$ we define the \defn{fundamental vectors} with respect to $\kk$ as the elements of the set
\begin{align}\label{eq_fundamentalvectorsAffine}
\fundamentalvectorsAffine{\kk} = 
\left\{a_1\widetilde w_1+\dots + a_{d+1} \widetilde w_{d+1} :
\begin{array}{l}
    \ 0 \leq a_i \leq k_i\in \Z^{d+1}  \\ 
    \text{at least one } a_i = 0
\end{array}
\right\}.
\end{align}
\end{definition}

\begin{lemma}\label{lem_quotient_size_affine}
The quotient $\latticeAffine{d}/\sublatticeAffine{\kk}$ is finite. Its number of elements is
\[
|\latticeAffine{d}/\sublatticeAffine{\kk}| =
\det M_\kk
= \prod (k_i+1) -\prod{k_i}
. 
\]
The fundamental vectors in $\fundamentalvectorsAffine{\kk}$ are element representatives of the classes of $\latticeAffine{d}/\sublatticeAffine{\kk}$. 
\end{lemma}

\begin{proof}
Since \(\latticeAffine{d} \subseteq V\), with \(V\) a \(d\)-dimensional space, and any \(d\) rows of \(M_\kk\) are linearly independent, the quotient is finite.
To count the elements, we first give a list of elements which contains a representative of every element of the quotient.
We then remove overcounted elements, and show that each of the remaining elements is in a distinct equivalence class.

The following set contains a representative of every element of the quotient.

\[A= \left\{a_1\widetilde w_1+\dots + a_{d+1} \widetilde w_{d+1} :\ 0 \leq a_i \leq k_i\in \Z^{d+1} \right\}.\]

Suppose \(v = b_1\widetilde w_1+\dots + b_{d+1} \widetilde w_{d+1} \in \latticeAffine{d}\). 
If for some \(i, b_i \not \in [0,k_i]\), there is an appropriate \(k \in \Z \) so that the representation of \(v+ k M_{\kk,i}\), where \(M_{\kk,i}\) is the \(i\)th row of \(M_{\kk}\), is in the range \([0,k_i]\). 
Cyclically correcting entries until they all lie within the specified set is a terminating process, since \(M_{\kk,i}\) has less effect on the \((i+1)\)th entry than \(M_{\kk,i+1}\) does.
This shows that every element of \(\latticeAffine{d}\) has a representative in this set.

We then remove elements in this set which have multiple representatives.
We remove all elements of the form 
\[B=\{a_1\widetilde w_1+\dots + a_{d+1} \widetilde w_{d+1} :\ 1 \leq a_i \leq k_i \in 
\Z^{d+1} \}.\]
To any such element, we can repeatedly subtract the all ones vector, \((1,\ldots, 1)\), until one of the entries is equal to \(0\).
The resulting vector is still in the first set, and differs from the starting element only by an integer combination of the rows of~\(M_\kk\).

We finally show that the following set has no repeated equivalence classes.
\[\fundamentalvectorsAffine{\kk}= A\smallsetminus B =
\left\{a_1\widetilde w_1+\dots + a_{d+1} \widetilde w_{d+1} :
\begin{array}{l}
    \ 0 \leq a_i \leq k_i\in \Z^{d+1}  \\ 
    \text{at least one } a_i = 0
\end{array}
\right\}.\]

Suppose \(s,t \in \fundamentalvectorsAffine{\kk}\) and \(t-s \in \sublatticeAffine{\kk}\). 
Since the rows of \(M_{\kk}\) are linearly independent, there is a unique choice of \(b_1,\dots,b_{d+1} \in \Z\) so that \(t-s = (b_1,\dots,b_{d+1}) M_{\kk}  \).

Consider the cyclic sequence of the \(b_i\)s.
If all the \(b_i\)s are equal, then \(t-s\) is a scalar multiple of the all ones vector.
Since $s,t\in \fundamentalvectorsAffine{\kk}$, this is possible only if $s=t$. 
%The only such vector in \(\fundamentalvectorsAffine{\kk}\) is the all zero vector, so \(s=t\).

If not all \(b_i\)s are equal, without loss of generality, assume some of the \(b_i\) are positive.
Let \(b_j\) be positive.
If \(b_{j-1} \geq b_j\), then \(b_{j-1}\) is positive, and we repeat using \(b_{j-1}\) in place of \(b_j\).
This repetition must end, because not all the \(b_i\) are equal and we are considering them cyclically.
Therefore we find a \(b_j\) so that \(b_j > 0\) and \(b_{j-1} < b_j\).
We now consider the value of \(t-s\) in the \(j\)th coordinate,  
\[
(t-s)_j = -k_j b_{j-1} + (k_j+1) b_j = b_j + (b_j -b_{j-1}) k_j.
\]
Since \(b_j > 0\) and \(b_j > b_{j-1}\), this coordinate is greater or equal to \(1 + k_j\).

However, the \(j\)th coordinate of \(t\) is at most \(k_j\), and the \(j\)th coordiante of \(s\) is at least \(0\), so the \(j\)th coordinate of \(t-s\) is at most \(k_j\).
This contradiction shows that the only way for \(t-s\) to be in \(\sublatticeAffine{\kk}\) is for them to be equal.
Therefore \(\fundamentalvectorsAffine{\kk}\) has a unique representative from each equivalence class of \(\latticeAffine{d}/\sublatticeAffine{\kk}\).

Counting the elements of \(\fundamentalvectorsAffine{\kk}\), we have \(\prod (k_i+1)\) elements in the first set $A$, and remove the \(\prod (k_i)\) elements of the second set $B$.
Together, these are \(\prod (k_i+1) -\prod{k_i} = \det M_\kk\) elements.
\end{proof}

\begin{remark}
% In Lemma~\ref{lem_quotient_size_affine}, the determinant of \(M_\kk\) is the same as the number of points of the lattice quotient. 
% From our proof, it is not obvious why this should be the case.
% Taking a different perspective, the number of elements of \(\latticeAffine{d}/\sublatticeAffine{\kk}\) is the number of lattice points of \(\Z^{d+1}/ (1,\dots,1) / \sublatticeAffine{\kk} \).
% If \((1,\dots,1)  \in \sublatticeAffine{\kk}\), then this can be simplified to \(\Z^{d+1}/\sublatticeAffine{\kk} \), where it is more clear that \(\det M_\kk\) counts the elements.
In Lemma~\ref{lem_quotient_size_affine}, the determinant of \(M_\kk\) is the same as the number of points of the lattice quotient. 
From our proof, it is not obvious why this should be the case.
Taking a different perspective, the number of elements of \(\latticeAffine{d}/\sublatticeAffine{\kk}\) is the number of lattice points of \(\Z^{d+1}/ (1,\dots,1) / \Zrowspan{M_\kk} \), where $\Zrowspan{M_\kk}$ denotes the set of integer linear combinations of the rows of $M_\kk$.
If the vector \((1,\dots,1)  \in \Zrowspan{M_\kk}\), then this can be simplified to \(\Z^{d+1}/\Zrowspan{M_\kk} \), where it is more clear that \(\det M_\kk\) counts the elements.
This is not the case if \((1,\dots,1)  \notin \Zrowspan{M_\kk}\); For example, compare the determinants and number of hexagons within Figure~\ref{fig_fostercensus}, and check if \((1,1,1)\) appears within the integer row span of the matrices: For Heawood and M\"obius-Kantor, \((1,1,1)\) does appear in the span, while for~$K_{3,3,}$, Pappus, Nauru, and F26A, it does not.
\end{remark}

\subsubsection{The parallelepiped domain}
Our next goal is to show that the underlying topological structure of $\torus{\kk}$ is a \(d\)-dimensional torus.
For $i \in [d+1]$, we define 
\begin{align}
\widetilde p_i = 
m_{i,1} \widetilde w_1 + m_{i,2} \widetilde w_2 + \dots + m_{i,d+1} \widetilde w_{d+1} \in V,
\end{align}
where $(m_{i,1},\dots,m_{i,d+1})$ is the $i$th row of the matrix $M_\kk$.
The first $d$ vectors $\widetilde p_1,\dots,\widetilde p_d$ are lineraly independent, and the last one is the negative of the sum of the others
\begin{align}
\widetilde p_{d+1}= -(\widetilde p_1+\dots+\widetilde p_d).
\end{align}
Indeed, adding all \(d+1\) many \(\widetilde p_{i}\) gives \(\widetilde w_1 + \widetilde w_2 + \dots + \widetilde w_{d+1}\) which is 0.
In particular, 
\begin{align}\label{eq_sublattice_spanvectors_parallelepiped}
\sublatticeAffine{\kk} = \Zspan{\widetilde p_1, \dots, \widetilde p_{d+1}}
= \Zspan{\widetilde p_1, \dots, \widetilde p_d}
\end{align}
where $\Zspan{\cdot}$ denotes the set of integer linear combinations of the elements of the given set.  
Note that the vectors $\widetilde p_i$ depend on $\kk$, but we omit $\kk$ in our notation for simplicity. 

\begin{definition}[The parallelepiped domain]\label{def_parallelepipeddomain}
The \defn{parallelepiped domain} of $\torus{\kk}$ is
% associated to $\kk$ is the set
\begin{align}
\parallelepipedDomainAffine{\kk} = 
\left \{
a_1\widetilde p_1 + \dots + a_d\widetilde p_d: \,
0\leq a_1 \leq 1 \text{ with } a_i\in \mathbb{R} 
\right \}.
\end{align}
\end{definition}

Figure~\ref{fig_parallelepiped_permutahedron_domains} (left) shows an example of the parallelepiped domain of $\torus{(3,1,2)}$.

\begin{figure}[htb]
\centering
    \begin{tabular}{cc}
    \begin{tikzpicture}[scale=0.5,every node/.append style={circle,draw=black,inner sep=1.8pt,align=center}]
    \def\con{0.8660254};
    \draw[draw=white,fill=black!10] (-1*\con,-5.5) -- (5*\con,0.5) -- (-1*\con,3.5) -- (-7*\con,-2.5) -- (-1*\con,-5.5);
    \node[fill=olive] (olv) at (-2*\con, 2){};
    \node[fill=red] (red) at (0*\con, 2){};
    \node[fill=indigo] (ind) at (-3*\con, 0.5){};
    \node[fill=bluishgreen] (blg) at (-1*\con, 0.5){};
    \node[fill=orange] (org) at (1*\con, 0.5){};
    \node[fill=skyblue] (sky) at (3*\con, 0.5){};
    \node[fill=blue2] (bl2) at (-4*\con, -1){};
    \node[fill=magenta] (mag) at (-2*\con, -1){};
    \node[fill=purple] (pur) at (0*\con, -1){};
    \node[fill=blue] (blu) at (2*\con, -1){};
    \node[fill=teal] (tea) at (-5*\con, -2.5){};
    \node[fill=green] (gre) at (-3*\con, -2.5){};
    \node[fill=indigo2] (in2) at (-1*\con, -2.5){};
    \node[fill=turquoise] (tur) at (1*\con, -2.5){};
    \node[fill=purple2] (pr2) at (-2*\con, -4){};
    \node[fill=brown] (bro) at (0*\con, -4){};
    \draw  (-3*\con,3.5) -- (olv) -- (-4*\con, 2) -- (ind) -- (-5*\con, 0.5) -- (bl2) -- (-6*\con, -1) -- (tea) -- (-6*\con, -4);
    \draw (-3*\con,-5.5) -- (pr2);
    \draw (bro) -- (2*\con, -4) -- (tur) --  (3*\con, -2.5) -- (blu) -- (4*\con, -1) -- (sky) -- (4*\con, 2);
    \draw (red) -- (1*\con,3.5);
    \draw[draw=white, fill=white] (-7*\con,-2.5) -- (-1*\con,3.5) -- (-7*\con,3.5) --cycle;
    \draw[draw=white, fill=white] (-7*\con,-2.5) -- (-1*\con,-5.5) -- (-7*\con,-5.5) --cycle;
    \draw[draw=white, fill=white] (5*\con,0.5) -- (-1*\con,-5.5) -- (5*\con,-5.5) --cycle;
    \draw[draw=white, fill=white] (5*\con,0.5) -- (-1*\con,3.5) -- (5*\con,3.5) --cycle;
    \node[fill=yellow] (yel1) at (2*\con, 2){};
    \node[fill=yellow] (yel2) at (-4*\con, -4){};
    \node[fill=lightblue] (lbu1) at (-1*\con,3.5){};
    \node[fill=lightblue] (lbu2) at (5*\con, 0.5){};
    \node[fill=lightblue] (lbu3) at (-7*\con, -2.5){};
    \node[fill=lightblue] (lbu4) at (-1*\con,-5.5){};
    \draw (red) -- (lbu1) -- (olv) -- (blg) -- (mag) -- (pur) -- (org) -- (sky) -- (yel1) -- (org) -- (blg) -- (red) -- (org) -- (blu) -- (pur) -- (blg) -- (ind) -- (olv) -- (red) -- (yel1);
    \draw (lbu2) -- (sky) -- (blu) -- (tur) -- (pur) -- (in2) -- (mag) -- (ind) -- (bl2) -- (mag) -- (gre) -- (in2) -- (tur) -- (bro) -- (in2) -- (pr2) -- (bro) -- (lbu4) -- (pr2) -- (yel2) -- (gre) -- (tea) -- (yel2);
    \draw (lbu3) -- (tea) -- (bl2) -- (gre) -- (pr2);
    \end{tikzpicture}
    &
    \begin{tikzpicture}[scale=0.5,every node/.append style={circle,draw=black,inner sep=1.8pt,align=center}]
    \def\con{0.8660254};
    \begin{scope}[shift={(-0.5*\con,-0.25)}]
    \draw[draw=white,fill=black!10] (-2*\con,-5) -- (-4*\con,-1) -- (-2*\con,4) -- (2*\con,5) -- (4*\con,1) -- (2*\con,-4) -- (-2*\con,-5);
    \end{scope}
    \node[fill=lightblue] (lbu1) at (-1*\con,3.5){};
    \node[fill=teal] (tea) at (1*\con, 3.5){};
    \node[fill=olive] (olv) at (-2*\con, 2){};
    \node[fill=red] (red) at (0*\con, 2){};
    \node[fill=indigo] (ind) at (-3*\con, 0.5){};
    \node[fill=bluishgreen] (blg) at (-1*\con, 0.5){};
    \node[fill=orange] (org) at (1*\con, 0.5){};
    \node[fill=skyblue] (sky) at (3*\con, 0.5){};
    \node[fill=blue2] (bl2) at (-4*\con, -1){};
    \node[fill=magenta] (mag) at (-2*\con, -1){};
    \node[fill=purple] (pur) at (0*\con, -1){};
    \node[fill=blue] (blu) at (2*\con, -1){};
    \node[fill=green] (gre) at (-3*\con, -2.5){};
    \node[fill=indigo2] (in2) at (-1*\con, -2.5){};
    \node[fill=turquoise] (tur) at (1*\con, -2.5){};
    \node[fill=purple2] (pr2) at (-2*\con, -4){};
    \node[fill=brown] (bro) at (0*\con, -4){};
    \node[fill=yellow] (yel1) at (2*\con, 2){};
    \draw (2*\con, 5) -- (tea) -- (0*\con, 5) -- (lbu1) -- (-3*\con,3.5) -- (olv) -- (-4*\con, 2) -- (ind) -- (-5*\con, 0.5) -- (bl2) -- (-5*\con, -1);
    \draw (-3*\con,-5.5) -- (pr2);
    \draw (1*\con, -5.5) -- (bro) -- (2*\con, -4) -- (tur) --  (3*\con, -2.5) -- (blu) -- (4*\con, -1) -- (sky) -- (4*\con, 2) -- (yel1) -- (tea) -- (3*\con, 3.5) -- (yel1);
    \draw (red) -- (tea) -- (lbu1) -- (-2*\con, 5);
    
    \draw (red) -- (lbu1) -- (olv) -- (blg) -- (mag) -- (pur) -- (org) -- (sky) -- (yel1) -- (org) -- (blg) -- (red) -- (org) -- (blu) -- (pur) -- (blg) -- (ind) -- (olv) -- (red) -- (yel1);
    \draw (4*\con, 0.5) -- (sky) -- (blu) -- (tur) -- (pur) -- (in2) -- (mag) -- (ind) -- (bl2) -- (mag) -- (gre) -- (in2) -- (tur) -- (bro) -- (in2) -- (pr2) -- (bro) -- (-1*\con,-5.5) -- (pr2) -- (-4*\con, -4) -- (gre) -- (-5*\con, -2.5) -- (bl2) -- (gre) -- (pr2);

    \draw[draw=white, fill=white] (2*\con, 5.1) -- (1.5*\con,4.75) --     (-2.5*\con,3.75) -- (-2*\con, 5.1)--cycle;

    \draw[draw=white, fill=white] (-2.1*\con, 5) -- (-2.5*\con,3.75) --     (-4.5*\con,-1.25) -- (-5*\con, -1) -- (-5.1*\con, 0.6) --cycle;
    
    \draw[draw=white, fill=white] (-5*\con, -1) -- (-4.5*\con,-1.25) --     (-2.5*\con,-5.25) -- (-3*\con,-5.6) -- (-5.1*\con, -2.5)--cycle;
    
    \draw[draw=white, fill=white] (-3*\con,-5.6) -- (-2.5*\con,-5.25) --     (1.5*\con,-4.25) -- (1.1*\con, -5.6)--cycle;
    
    \draw[draw=white, fill=white] (1*\con, -5.6) -- (1.5*\con,-4.25) --     (3.5*\con,0.75) --  (4*\con, 0.5)--  (4.1*\con, -1)--cycle;
    
    \draw[draw=white, fill=white] (4*\con, 0.5) -- (3.5*\con,0.75) --     (1.5*\con,4.75) -- (2*\con, 5.1) -- (4.1*\con, 2) --cycle;

    \end{tikzpicture}
    \end{tabular}
\caption{The parallelepiped domain and the permutahedron domain of~$\torus{(3,1,2)}$.}
\label{fig_parallelepiped_permutahedron_domains}
\end{figure}
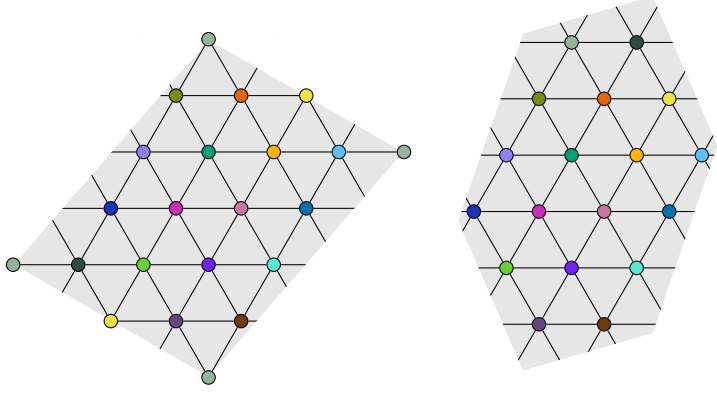

\begin{proposition}\label{prop_torus}
If $\kk=(k_1,\dots,k_{d+1})\in \N^{d+1}$ for some $d\geq 1$, then $\torus{\kk}$ is a triangulated $d$-dimensional torus on $D_\kk=\det M_\kk$ many vertices if \(k_i\geq 1\) for all \(i\). 
\end{proposition}

\begin{proof}
By~\eqref{eq_sublattice_spanvectors_parallelepiped} and the fact that $\widetilde p_1,\dots, \widetilde p_d$ are linearly independent, the parallelepiped domain~$\parallelepipedDomainAffine{\kk}$ is a fundamental domain of the quotient $\torus{\kk}= \complexAffine{d} / \sublatticeAffine{\kk}$. 
Therefore the topological structure of the quotient is that of a parallelepiped with opposite facets identified, a classical presentation of the \(d\)-torus.

From Lemma \ref{lem_quotient_size_affine}, it is clear that $\torus{\kk}$ has $D_\kk=\det M_\kk$ many vertices.

Finally, this complex is a triangulation as long as no face is repeated and no vertex appears more than once in each face.
Both of these are ensured by the condition that \(k_i\geq 1\) for all \(i\).
The shortest nontrivial path from a vertex to itself is of length \(3\), the smallest difference between minimum and maximum in elements of \(\sublatticeAffine{\kk}\).
Paths of length \(1\) could lead to repeated vertices, and paths of length \(2\) could lead to bi-edges.
With both of these impossible, we conclude that $\torus{\kk}$ is a simplicial complex.
\end{proof}

\begin{remark}
In this section, we relaxed the condition on \(d\) to include \(1\).
While cycles are not particularly interesting, there is no problem including them.

A more interesting extension is when  $\kk$ has \(0\)s in it.
For the proof that $\torus{\kk}$ is a simplicial complex, we require no short cycles.
If short cycles exist, then $\torus{\kk}$ is no longer a simplicial complex, but rather a delta complex, which relaxes the repeated face and vertex rule.
Indeed, such delta complexes are worthy of study on their own merits, as done in Garber and Sikiri\'c~\cite{garber_sikiric_periodictriangulations_2020}.
There also exist choices for $\kk$ which include \(0\)s and yet remain simplicial complexes.
One such choice is \(2,0,2,2,0,2\).
An interested reader is encouraged to work out the conditions on \(\kk\) which are necessary and sufficient for  $\torus{\kk}$ to be a simplicial complex.
\end{remark}

\subsubsection{The permutahedron domain}\label{sec_permutahedrondomain_affine}
The choice of the fundamental domain to be a parallelepiped, as in the previous section, is useful to understand the topology of~$\torus{\kk}$. 
However, there is another natural domain that we would like to mention in this section. 
In dimension 2, the usual representation of a torus is a square (or rhombus) with opposite sides identified. 
Another, not so common, representation of the torus is a hexagon with opposite sides identified by translation (preserving their orientation).
In this section, we show an analog result in higher dimensions. 

\begin{proposition}\label{prop_permutahedron_torus}
The permutahedron $\perm{d}$ with opposite facets identified by translation is a topological $d$-dimensional torus. 
\end{proposition}

In order to show this result, we present a fundamental domain for $\torus{\kk}$ which is a permutahedron. 

\begin{definition}\label{def_permutahedrondomainAffine}
The \defn{permutahedron domain} of $\torus{\kk}$ is
\begin{align}
\permutahedronDomainAffine{\kk} = 
\frac{1}{d+1}
\conv\left \{
(a_1,\dots,a_{d+1})_{\widetilde p}
: \,
\{a_1,\dots,a_{d+1}\}=[d+1]
\right \},
\end{align}
where
\begin{align}
(a_1,\dots,a_{d+1})_{\widetilde p} =
a_1 \widetilde p_1+\dots+a_{d+1}\widetilde p_{d+1}.
\end{align}
\end{definition}

An  example is illustrated on Figure~\ref{fig_parallelepiped_permutahedron_domains} (right).

Our definition of the permutahedron domain is a slight modification of the definition of the classical permutahedron $\perm{d}$. Indeed, the projection of $\perm{d}$ to the subspace $V$ is 
\begin{align}
\frac{1}{d+1}
\conv\left \{
(a_1,\dots,a_{d+1})_{w}
: \,
\{a_1,\dots,a_{d+1}\}=[d+1]
\right \},
\end{align}
where
\begin{align}
(a_1,\dots,a_{d+1})_{w} =
a_1 w_1+\dots+a_{d+1}w_{d+1}.
\end{align}

This can be shown by checking that if $\{a_1,\dots,a_{d+1}\}=[d+1]$ then
\begin{align}
\frac{1}{d+1} (a_1,\dots,a_{d+1})_w 
&= (a_1,\dots,a_{d+1}) - \frac{1+\dots+(d+1)}{d+1}(1,\dots,1)\\
&= (a_1,\dots,a_{d+1}) - \operatorname{barycenter}(\perm{d}).
\end{align}

As a consequence, $\permutahedronDomainAffine{\kk}$ is an affine transformation of $\perm{d}$ (sending the barycenter to zero and $w_i$ to $\widetilde p_i$), 
and so $\permutahedronDomainAffine{\kk}$ is combinatorially a permutahedron. 

Furthermore, the facets of $\permutahedronDomainAffine{\kk}$ can also be labeled by ordered partitions~$[B_1,B_2]$ of~$[d+1]$ consisting of two proper subsets. 
The vertices of this facet are of the form $(a_1,\dots,a_{d+1})_{\widetilde p}$ where
the following two conditions are satisfied
\begin{align}
\left \{ 
a_i
\right \}_{i\in B_1}
&=\{1,2,\dots,|B_1|\},
\\
\left \{ 
a_i
\right \}_{i\in B_2}
&=\{|B_1|+1,|B_1|+2,\dots,d+1\}.
\end{align}

The two facets $[B_1,B_2]$ and $[B_2,B_1]$ are opposite to each other and, as in~\Cref{lem_rotation_partition_one}, we have 
\begin{align}\label{eq_facettranslation_permutahedrondomain}
[B_2,B_1] = [B_1,B_2]+ \sum_{i\in B_1} \widetilde p_i.
\end{align}

Since $\sum_{i\in B_1} \widetilde p_i \in \sublatticeAffine{\kk}$, the permutahedron $\permutahedronDomainAffine{\kk}$ is a fundamental domain for $\torus{\kk}$ with opposite facets identified by the translation~\eqref{eq_facettranslation_permutahedrondomain}. 

\begin{proof}[Proof of~\Cref{prop_permutahedron_torus}]
As just explained, $\torus{\kk}$ is a triangulation of the permutahedron $\permutahedronDomainAffine{\kk}$ with opposite facets identified by translation. 
Since $\torus{\kk}$ is a torus by~\Cref{prop_torus} and $\perm{d}$ is an affine transformation of $\permutahedronDomainAffine{\kk}$, then the result follows. 
\end{proof}

\subsection{The Heawood complex}\label{sec_heawoodComplex}
We can also make some identifications in the infinite complex of faces of the permutahedral tiling to obtain a finite polytopal complex. 

For this, we use the lattices $\sublattice{\kk}\subset \lattice{d}$ defined in Equations~\eqref{eq_lattice_one} and~\eqref{eq_sublattice}, which are defined in terms of linear combinations of the vectors $w_i$ from Equation~\eqref{eq_vector_wi}. They satisfy 
\begin{align}\label{eq_lattice_affineLattice}
    \sublattice{\kk} = (d+1)\sublatticeAffine{\kk},
    \qquad
    \lattice{d} = (d+1)\latticeAffine{d},
    \qquad
    w_i = (d+1) \widetilde w_i.
\end{align}

We say that two faces $\B,\B'$ of the permutahedral tiling $\permutahedralTiling{d}$ are \defn{$\kk$-equivalent}, and write \mbox{$\B\equivalent{\kk} \B'$}, if 
\begin{align}
\B'=\B+v \text{ for some } v\in \sublattice{\kk}.
\end{align}
That is, the face $\B'$ can be obtained by translating the face $\B$ by a vector $v\in \sublattice{d}$. 

\begin{definition}[The Heawood complex]
\label{def_heawood_complex}
Let $\kk=(k_1,\dots,k_{d+1})\in \N^{d+1}$ be a sequence of positive integers for some $d\geq 2$. The \defn{Heawood complex} 
\[
\heawoodComplex{\kk}= \permutahedralTiling{d} / \sublattice{\kk}
\]
is the polytopal complex of $\kk$-equivalent classes of faces of $\permutahedralTiling{d}$. 
In other words,~$\heawoodComplex{\kk}$ is the polytopal complex obtained by identifying faces of $\permutahedralTiling{d}$ up to translation by vectors in $\sublattice{\kk}$.   
A face class is contained in another if there is an element representative of the first that is contained in an element representative of the second.
\end{definition}

\begin{lemma}\label{lem_quotient_size}
The quotient $\lattice{d}/\sublattice{\kk}$ is finite. Its number of elements counts the numbers of full dimensional faces of the Heawood complex~$\heawoodComplex{\kk}$, and is equal to
\[
|\lattice{d}/\sublattice{\kk}| =
\det M_\kk
= \prod (k_i+1) -\prod{k_i}
. 
\]
\end{lemma}
\begin{proof}
    The formula is straightforward from~\Cref{lem_quotient_size_affine} and~\eqref{eq_lattice_affineLattice}. 
    Furthermore, recall that the full dimensional cells of the permutahedral tiling $\permutahedralTiling{d}$ are of the form $\perm{d}+v$ for $v\in \lattice{d}$. Thus, the full dimensional cells of the Heawood complex are of the form $\perm{d}+v$ for $v\in \lattice{d}/\sublattice{\kk}$, and the number $|\lattice{d}/\sublattice{\kk}|$ is counting them. 
\end{proof}

Recall that two polytopal complexes are said to be \defn{dual} to each other if their posets of non-empty faces are dual to each other, i.e. there is a bijection between their faces that reverses the containment order. 
The \defn{dual graph} of a simplicial complex is the graph on the set of facets of the complex, whose edges connect two facets that share a codimension 1 face.  

\begin{proposition}\label{prop_duality_heawood_torus}
The following hold:
\begin{enumerate}
    \item The Heawood graph $\heawoodGraph{\kk}$ is the edge graph of the Heawood complex~$\heawoodComplex{\kk}$.\label{prop_duality_heawood_torus_item1}
    \item The Heawood complex and the torus are dual complexes:
    \begin{align*}
        \heawoodComplex{\kk} \cong \torus{\kk}^*.
    \end{align*}
    \label{prop_duality_heawood_torus_item2}
    \item The Heawood graph $\heawoodGraph{\kk}$ is the dual graph of the torus $\torus{\kk}$.\label{prop_duality_heawood_torus_item3}
\end{enumerate}
\end{proposition}

\begin{proof}
    Item~\eqref{prop_duality_heawood_torus_item1} follows by definition. 
    
    For Item~\eqref{prop_duality_heawood_torus_item2}, recall that $\permutahedralTiling{d}$ and $\complexAffine{d}$ are dual complexes. The vertices of $\complexAffine{d}$ are the elements $\widetilde v \in \latticeAffine{d}$. Such a vertex $\widetilde v$ is dual to maximal dimensional cell  $\perm{v}\in\permutahedralTiling{d}$ given by 
    \[
    \perm{v}=\perm{d} + v,
    \]
    where $v=(d+1)\widetilde v\in \lattice{d}$. 
    We denote by $\phi:\complexAffine{d}\rightarrow \permutahedralTiling{d}$ the containment reversing map that sends a face $F\in \complexAffine{d}$ to its dual face $\B\in \permutahedralTiling{d}$. This map satisfies 
    \[
    \phi(F+\widetilde v) = \phi(F) +v = \phi(F) +(d+1)\widetilde v.
    \]

    Since $\sublattice{\kk}=(d+1)\sublatticeAffine{\kk}$, then  $\phi$ sends the $\kk$-equivalent class of $F\in \complexAffine{d}$ to the $\kk$-equivalent class of $\B \in \permutahedralTiling{d}$. 
    Thus, $\phi$ induces a map between the faces of the quotients $\torus{\kk}=\complexAffine{d}/\sublatticeAffine{\kk}$ and $\heawoodComplex{\kk}=\permutahedralTiling{d}/\sublattice{\kk}$. 
    This ``quotient map" is a well defined bijection that reverses the containment order. This finishes our proof of Item~\eqref{prop_duality_heawood_torus_item2}.

    Item~\eqref{prop_duality_heawood_torus_item3} is a consequence of Items~\eqref{prop_duality_heawood_torus_item1} and~\eqref{prop_duality_heawood_torus_item2}.
\end{proof}

\subsubsection{The fundamental tile and the permutahedron domain}
As our figures suggest, the Heawood complex $\heawoodComplex{\kk}$ can be obtained by making some identifications on the boundary of a finite fundamental piece of the infinite permutahedral tiling. 
There are several ways to choose such a fundamental piece. 
In this section, we present two natural choices, the fundamental tile and the permutahedron domain. 

Our first choice is related to the fundamental vector representatives from~\Cref{def_fundamentalvectors}.
We define 
\begin{align}\label{eq_fundamentalvectors}
\fundamentalvectors{\kk} = 
\left\{a_1 w_1+\dots + a_{d+1} w_{d+1} :
\begin{array}{l}
    \ 0 \leq a_i \leq k_i\in \Z^{d+1}  \\ 
    \text{at least one } a_i = 0
\end{array}
\right\}.
\end{align}

Since $w_i=(d+1)\widetilde w_i$, then the vectors in $\fundamentalvectors{\kk}$ are just the fundamental vectors in~$\fundamentalvectorsAffine{\kk}$ from~\Cref{def_fundamentalvectors}, dilated by a factor of $d+1$. 
Similarly as in~\Cref{lem_quotient_size_affine}, we have 

\begin{lemma}\label{lem_fundamentalvector_representatives}
The vectors in $\fundamentalvectors{\kk}$ are element representatives of the classes of~$\lattice{d}/\sublattice{\kk}$. 
\end{lemma}

\begin{proof}
This follows directly from~\Cref{lem_quotient_size_affine} and~\eqref{eq_lattice_affineLattice}.
\end{proof}

\begin{definition}
The \defn{fundamental tile} $\fundamentaltile{\kk}$ is the union of the permutahedra of the form $\perm{d}+v$ with $v\in \fundamentalvectors{\kk}$.
\end{definition}

Several examples of the fundamental tiles (including some translations or expansions) are illustrated in Figures~\ref{fig_figure111},~\ref{fig_fundamentaltile_translates},~\ref{fig_figureFour1111},~\ref{fig_figure1111expand} and~\ref{fig_figure2223expand}.
 See also Figure~\ref{fig_fundamental_tile_and_permutahedron_domain} (left).

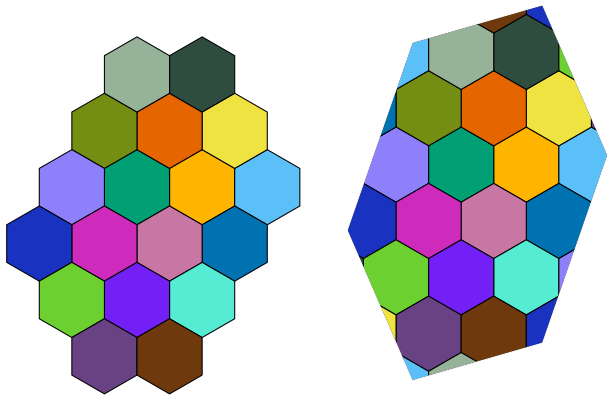
\begin{figure}[htb]
\centering
    \begin{tabular}{cc}
    \begin{tikzpicture}[scale=0.5]
    \def\con{0.8660254};
    \draw[fill=orange] (0*\con,0) -- (0*\con,1) -- (1*\con,1.5) -- (2*\con,1) -- (2*\con,0) -- (1*\con,-0.5) -- (0*\con,0);
    \draw[fill=skyblue] (2*\con,0) -- (2*\con,1) -- (3*\con,1.5) -- (4*\con,1) -- (4*\con,0) -- (3*\con,-0.5) -- (2*\con,0);
    \draw[fill=bluishgreen] (-2*\con,0) -- (-2*\con,1) -- (-1*\con,1.5) -- (0*\con,1) -- (0*\con,0) -- (-1*\con,-0.5) -- (-2*\con,0);
    \draw[fill=yellow] (1*\con,1.5) -- (1*\con,2.5) -- (2*\con,3) -- (3*\con,2.5) -- (3*\con,1.5) -- (2*\con,1) -- (1*\con,1.5);
    \draw[fill=red] (1*\con,1.5) -- (1*\con,2.5) -- (0*\con,3) -- (-1*\con,2.5) -- (-1*\con,1.5) -- (0*\con,1) -- (1*\con,1.5);
    \draw[fill=blue] (1*\con,-1.5) -- (1*\con,-0.5) -- (2*\con,0) -- (3*\con,-0.5) -- (3*\con,-1.5) -- (2*\con,-2) -- (1*\con,-1.5);
    \draw[fill=purple] (1*\con,-1.5) -- (1*\con,-0.5) -- (0*\con,0) -- (-1*\con,-0.5) -- (-1*\con,-1.5) -- (0*\con,-2) -- (1*\con,-1.5);
    \draw[fill=olive] (-1*\con,1.5) -- (-1*\con,2.5) -- (-2*\con,3) -- (-3*\con,2.5) -- (-3*\con,1.5) -- (-2*\con,1) -- (-1*\con,1.5);
    \draw[fill=indigo] (-4*\con,0) -- (-4*\con,1) -- (-3*\con,1.5) -- (-2*\con,1) -- (-2*\con,0) -- (-3*\con,-0.5) -- (-4*\con,0);
    \draw[fill=magenta] (-1*\con,-1.5) -- (-1*\con,-0.5) -- (-2*\con,0) -- (-3*\con,-0.5) -- (-3*\con,-1.5) -- (-2*\con,-2) -- (-1*\con,-1.5);
    \draw[fill=teal] (0*\con,3) -- (0*\con,4) -- (1*\con,4.5) -- (2*\con,4) -- (2*\con,3) -- (1*\con,2.5) -- (0*\con,3);
    \draw[fill=turquoise] (0*\con,-3) -- (0*\con,-2) -- (1*\con,-1.5) -- (2*\con,-2) -- (2*\con,-3) -- (1*\con,-3.5) -- (0*\con,-3);
    \draw[fill=lightblue] (-2*\con,3) -- (-2*\con,4) -- (-1*\con,4.5) -- (0*\con,4) -- (0*\con,3) -- (-1*\con,2.5) -- (-2*\con,3);
    \draw[fill=indigo2] (-2*\con,-3) -- (-2*\con,-2) -- (-1*\con,-1.5) -- (0*\con,-2) -- (0*\con,-3) -- (-1*\con,-3.5) -- (-2*\con,-3);
    \draw[fill=blue2] (-3*\con,-1.5) -- (-3*\con,-0.5) -- (-4*\con,0) -- (-5*\con,-0.5) -- (-5*\con,-1.5) -- (-4*\con,-2) -- (-3*\con,-1.5);
    \draw[fill=green] (-4*\con,-3) -- (-4*\con,-2) -- (-3*\con,-1.5) -- (-2*\con,-2) -- (-2*\con,-3) -- (-3*\con,-3.5) -- (-4*\con,-3);
    \draw[fill=purple2] (-1*\con,-4.5) -- (-1*\con,-3.5) -- (-2*\con,-3) -- (-3*\con,-3.5) -- (-3*\con,-4.5) -- (-2*\con,-5) -- (-1*\con,-4.5);
    \draw[fill=brown] (1*\con,-4.5) -- (1*\con,-3.5) -- (0*\con,-3) -- (-1*\con,-3.5) -- (-1*\con,-4.5) -- (0*\con,-5) -- (1*\con,-4.5);
    \end{tikzpicture}
    &

    \begin{tikzpicture}[scale=0.5]
    \def\con{0.8660254};
    \draw[fill=orange] (0*\con,0) -- (0*\con,1) -- (1*\con,1.5) -- (2*\con,1) -- (2*\con,0) -- (1*\con,-0.5) -- (0*\con,0);
    \draw[fill=skyblue] (2*\con,0) -- (2*\con,1) -- (3*\con,1.5) -- (4*\con,1) -- (4*\con,0) -- (3*\con,-0.5) -- (2*\con,0);
    \draw[fill=bluishgreen] (-2*\con,0) -- (-2*\con,1) -- (-1*\con,1.5) -- (0*\con,1) -- (0*\con,0) -- (-1*\con,-0.5) -- (-2*\con,0);
    \draw[fill=yellow] (1*\con,1.5) -- (1*\con,2.5) -- (2*\con,3) -- (3*\con,2.5) -- (3*\con,1.5) -- (2*\con,1) -- (1*\con,1.5);
    \draw[fill=red] (1*\con,1.5) -- (1*\con,2.5) -- (0*\con,3) -- (-1*\con,2.5) -- (-1*\con,1.5) -- (0*\con,1) -- (1*\con,1.5);
    \draw[fill=blue] (1*\con,-1.5) -- (1*\con,-0.5) -- (2*\con,0) -- (3*\con,-0.5) -- (3*\con,-1.5) -- (2*\con,-2) -- (1*\con,-1.5);
    \draw[fill=purple] (1*\con,-1.5) -- (1*\con,-0.5) -- (0*\con,0) -- (-1*\con,-0.5) -- (-1*\con,-1.5) -- (0*\con,-2) -- (1*\con,-1.5);
    \draw[fill=olive] (-1*\con,1.5) -- (-1*\con,2.5) -- (-2*\con,3) -- (-3*\con,2.5) -- (-3*\con,1.5) -- (-2*\con,1) -- (-1*\con,1.5);
    \draw[fill=indigo] (-4*\con,0) -- (-4*\con,1) -- (-3*\con,1.5) -- (-2*\con,1) -- (-2*\con,0) -- (-3*\con,-0.5) -- (-4*\con,0);
    \draw[fill=magenta] (-1*\con,-1.5) -- (-1*\con,-0.5) -- (-2*\con,0) -- (-3*\con,-0.5) -- (-3*\con,-1.5) -- (-2*\con,-2) -- (-1*\con,-1.5);
    \draw[fill=teal] (0*\con,3) -- (0*\con,4) -- (1*\con,4.5) -- (2*\con,4) -- (2*\con,3) -- (1*\con,2.5) -- (0*\con,3);
    \draw[fill=turquoise] (0*\con,-3) -- (0*\con,-2) -- (1*\con,-1.5) -- (2*\con,-2) -- (2*\con,-3) -- (1*\con,-3.5) -- (0*\con,-3);
    \draw[fill=lightblue] (-2*\con,3) -- (-2*\con,4) -- (-1*\con,4.5) -- (0*\con,4) -- (0*\con,3) -- (-1*\con,2.5) -- (-2*\con,3);
    \draw[fill=indigo2] (-2*\con,-3) -- (-2*\con,-2) -- (-1*\con,-1.5) -- (0*\con,-2) -- (0*\con,-3) -- (-1*\con,-3.5) -- (-2*\con,-3);
    \draw[fill=blue2] (-3*\con,-1.5) -- (-3*\con,-0.5) -- (-4*\con,0) -- (-5*\con,-0.5) -- (-5*\con,-1.5) -- (-4*\con,-2) -- (-3*\con,-1.5);
    \draw[fill=green] (-4*\con,-3) -- (-4*\con,-2) -- (-3*\con,-1.5) -- (-2*\con,-2) -- (-2*\con,-3) -- (-3*\con,-3.5) -- (-4*\con,-3);
    \draw[fill=purple2] (-1*\con,-4.5) -- (-1*\con,-3.5) -- (-2*\con,-3) -- (-3*\con,-3.5) -- (-3*\con,-4.5) -- (-2*\con,-5) -- (-1*\con,-4.5);
    \draw[fill=brown] (1*\con,-4.5) -- (1*\con,-3.5) -- (0*\con,-3) -- (-1*\con,-3.5) -- (-1*\con,-4.5) -- (0*\con,-5) -- (1*\con,-4.5);

    \begin{scope}[shift={(6*\con,6)}]
    \draw[fill=blue2] (-4*\con,-1) --  (-5*\con,-1) -- (-5*\con,-1.5) -- (-4*\con,-2) -- (-4*\con,-1);
    \draw[fill=green] (-4*\con,-3) -- (-4*\con,-2) -- (-3*\con,-3.5) -- (-4*\con,-3);
    \draw[fill=purple2] (-3*\con,-3.5) -- (-3*\con,-4.5) -- (-2*\con,-5) -- (-3*\con,-3.5);
    \end{scope}

    \begin{scope}[shift={(-6*\con,-6)}]
    \draw[fill=teal] (1*\con,4.5) -- (2*\con,4) -- (2*\con,3) -- (1*\con,4.5);
    \draw[fill=yellow] (2*\con,3) -- (3*\con,2.5) -- (3*\con,1.5) -- (2*\con,3);
    \draw[fill=skyblue] (3*\con,1.5) -- (4*\con,1) -- (4*\con,0.5) -- (3*\con,0.5) -- (3*\con,1.5);
    \end{scope}

    \begin{scope}[shift={(0*\con,-9)}]
    \draw[fill=lightblue] (-2*\con,3.5) -- (-2*\con,4) -- (-1*\con,4.5) -- (0*\con,4) -- (-2*\con,3.5);
    \end{scope}
    
    \begin{scope}[shift={(0*\con,9)}]
    \draw[fill=brown] (1*\con,-4.5) -- (1*\con,-4) -- (-1*\con,-4.5) -- (0*\con,-5) -- (1*\con,-4.5);
    \end{scope}

    \begin{scope}[shift={(6*\con,-3)}]
    \draw[fill=blue2] (-4*\con,0) -- (-5*\con,-0.5) -- (-5*\con,-1.5) -- (-4*\con,-1.5) -- (-4*\con,0);
    \draw[fill=indigo] (-4*\con,0) -- (-4*\con,1) -- (-3*\con,1.5) -- (-4*\con,0);
    \end{scope}
    
    \begin{scope}[shift={(-6*\con,3)}]
    \draw[fill=skyblue] (3*\con,1) -- (4*\con,1) -- (4*\con,0) -- (3*\con,-0.5) -- (3*\con,1);
    \draw[fill=blue] (2*\con,-2)  -- (3*\con,-0.5) -- (3*\con,-1.5) -- (2*\con,-2);
    \end{scope}
    
    \draw (3.5*\con,0.75) --(1.5*\con,4.75) --(-2.5*\con,3.75) --(-4.5*\con,-1.25) -- (-2.5*\con,-5.25) --(1.5*\con,-4.25) --   (3.5*\con,0.75);

    \draw[draw=white, fill=white] (3.5*\con,0.75) --(1.5*\con,4.75) -- (0,5.1) -- (4.1*\con,5.1) -- (4.1*\con,0) -- cycle;

    \draw[draw=white, fill=white]  (1*\con,5.1) -- (1.5*\con,4.75) --(-2.5*\con,3.75) -- (-4*\con, 3) -- (-4*\con,5.1) -- cycle;

    \draw[draw=white, fill=white] (-2.5*\con,3.75) --(-4.5*\con,-1.25) -- (-5.1*\con,-1) -- (-5.1*\con,4) -- cycle;

    \draw[draw=white, fill=white] (-5.1*\con,1) -- (-4.5*\con,-1.25) -- (-2.5*\con,-5.25)-- (-2*\con,-5.6)-- (-5.1*\con,-5.6) -- cycle;
    \draw[draw=white, fill=white] (-3*\con,-5.6) -- (-2.5*\con,-5.25) --(1.5*\con,-4.25)-- (3*\con,-4) -- (3*\con,-5.6) -- cycle;
    \draw[draw=white, fill=white] (4.1*\con,-5) -- (1.5*\con,-4.25) --   (3.5*\con,0.75)-- (4.1*\con,1) --cycle;
    \end{tikzpicture}
    \end{tabular}
\caption{The fundamental tile and the permutahedron domain of~$\heawoodComplex{(3,1,2)}$.}
\label{fig_fundamental_tile_and_permutahedron_domain}
\end{figure}

As an observation, note that the translations of the form $\fundamentaltile{\kk}+v$ for $v\in \sublattice{\kk}$ tile the affine space containing $\perm{d}$. 
Also, the faces of the infinite permutahedral tiling restricted to $\fundamentaltile{\kk}$ give a finite representation of the Heawood complex. Withing this finite part, only faces in the boundary of $\fundamentaltile{\kk}$ are identified by translation of elements in~$\sublattice{\kk}$. Such identification can be interpreted according to the way the translations of the fundamental tile $\fundamentaltile{\kk}$ by elements $v\in \sublattice{\kk}$ glue together. 

\medskip
Our second choice for a fundamental domain for the Heawood complex $\heawoodComplex{\kk}$ is a permutahedral domain, which is similar to the permutahedral domain of $\torus{\kk}$ in~\Cref{sec_permutahedrondomain_affine}. 
For $i\in [d+1]$, we denote by $p_i$ the vector
\begin{align}
p_i=(d+1)\widetilde p_i = 
m_{i,1}  w_1 + m_{i,2}  w_2 + \dots + m_{i,d+1}  w_{d+1},
\end{align}
where $(m_{i,1},\dots,m_{i,d+1})$ is the $i$th row of the matrix $M_\kk$.

\begin{definition}\label{def_permutahedrondomain}
The \defn{permutahedron domain} of $\heawoodComplex{\kk}$ is
\begin{align}
\permutahedronDomain{\kk} = 
\frac{1}{d+1}
\conv\left \{
(a_1,\dots,a_{d+1})_{p}
: \,
\{a_1,\dots,a_{d+1}\}=[d+1]
\right \},
\end{align}
where
\begin{align}
(a_1,\dots,a_{d+1})_{p} =
a_1 p_1+\dots+a_{d+1}p_{d+1}.
\end{align}
\end{definition}

As the name suggest, the permutahedron domain is a permutahedron. It is a fundamental domain of the Heawood complex where its opposite facets are identified by translation.

An example of the permutahedron domain $\permutahedronDomain{(3,1,2)}$ of $\heawoodComplex{(3,1,2)}$ is illustrated in Figure~\ref{fig_fundamental_tile_and_permutahedron_domain}~(right). A three dimensional example for $\heawoodComplex{(1,1,1,1)}$ is shown in~\Cref{fig_permutahedraltiling3D}~(right). 

We remark that the gray hexagons in~\Cref{fig_heawoods_dimension2} show alternative permutahedron domains for the examples in that figure. Note that the permutahedron domain from~\Cref{def_permutahedrondomain} is not necessarily equal to those. For instance, compare the gray hexagon in~\Cref{fig_heawoods_dimension2_e} with the permutahedral domain in~\Cref{fig_fundamental_tile_and_permutahedron_domain} (right); they are not exactly equal because the vertices of the permutahedral domain $\permutahedronDomain{\kk}$ are not necessarily vertices of the Heawood complex (they might be slightly inside the hexagons). However, one can be obtained from the other by a slight surgery.  

Although our definition of permutahedral domain in~\Cref{def_permutahedrondomain} is elegant and simple, it would also be interesting to have a modification of it whose vertices belong to the Heawood complex. We leave this as an open problem for the interested readers. 

\begin{question}
Is there a fundamental domain of the Heawood complex $\heawoodComplex{\kk}$ which is a permutahedron (or another nice polytope) and whose vertices are vertices of the fundamental tile~$\fundamentaltile{\kk}$?
\end{question}

As a motivating example, it would be nice to find vertices of the fundamental tile~$\fundamentaltile{(1,1,1,1)}$ in~\Cref{fig_permutahedraltiling3D}, whose convex hull is a permutahedron domain of~$\heawoodComplex{(1,1,1,1)}$. Note that the vertices of our permutahedron domain $\permutahedronDomain{(1,1,1,1)}$ presented on the right of~\Cref{fig_permutahedraltiling3D} are not vertices of the fundamental tile.  

\section{Proof of Theorems~\ref{thm_maintheorem} and~\ref{thm_heawoodgraph}}
\subsection{Proof of Theorem~\ref{thm_maintheorem}}
Let $\kk=(k_1,\dots,k_{d+1})\in \N^{d+1}$ for some $d\geq 2$ as above.
We already proved that the Heawood graph $\heawoodGraph{\kk}$ is the dual graph of the $d$-dimensional torus $\torus{\kk}$ in~\Cref{prop_duality_heawood_torus}~\eqref{prop_duality_heawood_torus_item3}. 
It remains to show that its $f$-vector $f(\torus{\kk})=(f_0,f_1,\dots ,f_d)$ is determined by
\begin{align*}
f_i=i! \, \stirling{d+1}{i+1} \, D_{\kk},    
\end{align*}
where $\stirling{n}{k}$ is the Stirling number of the second kind, counting the number of ways to partition a set of $n$ objects into $k$ non-empty subsets, and $D_\kk=\det M_\kk$.
By the duality in~\Cref{prop_duality_heawood_torus}~\eqref{prop_duality_heawood_torus_item2}, this is equivalent to the following proposition.

\begin{proposition}\label{prop_counting_faces}
    The number of codimension $i$ faces of the Heawood complex $\heawoodComplex{\kk}$ is equal to
    \begin{align*}
    i! \, \stirling{d+1}{i+1} \, D_{\kk}.    
\end{align*}
\end{proposition}

In this section, we concentrate on proving this proposition.
The following rotation lemmas on faces of the permutahedron will be useful. The reader is invited to verify examples in \Cref{fig_perm_2D}.

\begin{lemma}[First Rotation Lemma]\label{lem_rotation_partition_one}
    The faces of the permutahedron $\perm{d}$ satisfy
    \[
    [B_2,\dots,B_k,B_1] = [B_1,B_2,\dots, B_k] + \sum_{b\in B_1}w_b.
    \]
\end{lemma}
\begin{proof}
    The vertices of $[B_1,B_2,\dots, B_k]$ are the permutations $(x_1,\dots,x_{d+1})$ of $[d+1]$ such that 
    \[
    \{x_a\}_{a\in B_i}=\{b_{i-1}+1,\dots , b_i\}
    \]
     where $b_i:=|B_1|+\dots +|B_i|$ and $b_0:=0$.
    
    We translate this face by the vector  
    \[
    v:= \sum_{b\in B_1}w_b = 
    (d+1-|B_1|)\sum_{a\in B_1} e_a -
    |B_1|\sum_{a\notin B_1} e_a. 
    \]
    The resulting vector $(x_1',\dots,x'_{d+1}):=(x_1,\dots,x_{d+1})+v$ satisfies

    \begin{align*}
    x_a' &= x_a+ d+1-|B_1|, \text{ for } a\in B_1 \\
    x_a' &= x_a-|B_1|,\text{ for } a\notin B_1. 
    \end{align*}
    Equivalently, it is a permutation $(x_1',\dots,x_{d+1}')$ of $[d+1]$ such that 
    \begin{align*}
    \{x_a'\}_{a\in B_1}&=\{b_k'+1,\dots , d+1\} \\
    \{x_a'\}_{a\in B_i}&=\{b_{i-1}'+1,\dots , b_i'\}, \text{ for } i\geq 2,
    \end{align*}
    where $b_1':=0$ and
    \begin{align*}
    b_i'&=b_{i}-|B_1|=|B_2|+\dots +|B_i|, \text{ for } i\geq 2.
    \end{align*}    
    The elements $(x_1',\dots,x_{d+1}')$ are precisely the vertices of the face $\mathbf{B}'=[B_2,\dots,B_k,B_1]$.
\end{proof}

Given two ordered partitions $\B=[B_1,\dots,B_k]$ and $\B'=[B_1',\dots,B_k']$ of $[d+1]$, we say that $\B'$ is a \defn{rotation} of $\B$ (or vice versa) if there is a value $j$ such that $B_i'=B_{i+j}$ for all $i$, where the subindex $i+j$ is considered mod $k$. 

\begin{lemma}[Second Rotation Lemma]\label{lem_rotation_partition_two}
    Let $[B_1,\dots,B_k]$ and $[B_1',\dots,B_k']$ be two faces of the permutahedron $\perm{d}$.  
    If 
    \[
    [B_1',\dots , B_k']=[B_1,\dots , B_k]+v
    \]
    for some $v\in\lattice{d}$ then $[B_1',\dots , B_k']$ is a rotation of $[B_1,\dots , B_k]$.
\end{lemma}

\begin{proof}
    We will prove first the case where all the blocks of both ordered partitions are singletons, and then argue that this suffices to prove the general case.

    \medskip
    \emph{Proof of the singleton blocks case.}
    Let $\B=[i_1,\dots , i_{d+1}]$ and $\B'=[i_1',\dots,i_{d+1}']$ be two ordered partitions of $[d+1]$ consisting of singleton blocks (blocks of size one), and assume that 
    \[
    [i_1',\dots,i_{d+1}'] = [i_1,\dots , i_{d+1}] + v
    \]
    for some $v\in \lattice{d}$. We aim to prove that $[i_1',\dots,i_{d+1}']$ is a rotation of $[i_1,\dots , i_{d+1}]$. 

    By~\Cref{lem_rotation_partition_one}, we can apply rotations to $[i_1',\dots, i_{d+1}']$ to obtain a new ordered partition $[j_1,\dots, j_{d+1}]$ such that $j_1=i_1$. This rotation is obtained by translating by a vector $v'\in \lattice{d}$:
    \[
    [i_1',\dots,i_{d+1}']+v' = [j_1,\dots,j_{d+1}].
    \]
    In particular
    \[
    [j_1,\dots,j_{d+1}] = [i_1,\dots , i_{d+1}] + v''
    \]
    for $v''=v+v'\in \lattice{d}$ and $j_1=i_1$. We want to show that $j_\ell=i_\ell$ for all $\ell$. 

    Let $\x=(x_1,\dots,x_{d+1})$ (resp. $\y=(y_1,\dots,y_{d+1})$) be the coordinate of the vertex $[i_1,\dots,i_{d+1}]$ (resp. $[j_1,\dots,j_{d+1}]$) in the permutahedron. Recall that these are given by their inverse permutations. In particular, $x_{i_1}=y_{i_1}=1$. 

    The vector $v''= (v_1'',\dots,v_{d+1}'') =\y-\x \in \lattice{d}$, and $v_{i_1}''=0$.
    Furthermore, $|v_{i_\ell}''|=|y_{i_\ell}-x_{i_\ell}|\leq d$ because $\x$ and $\y$ are permutations of $[d+1]$.
    By definition of the lattice $\lattice{d}$ in~\eqref{eq_lattice_two}, we know that 
    \[
    v_{i_\ell}'' - v_{i_1}'' = v_{i_\ell}'' \in (d+1)\Z. 
    \]
    Therefore, $v_{i_\ell}''=0=y_{i_\ell}-x_{i_\ell}$ for all $\ell$. 
    Thus, $\x=\y$ and $j_\ell=i_\ell$ for all $\ell$ as we wanted.
    This finishes the proof of the singletons blocks case.

    \medskip
    \emph{Proof of the general blocks case.}
    Let $\B=[B_1,\dots , B_k]$ and $\B'=[B_1',\dots,B_k']$ be two ordered partitions of $[d+1]$ such that 
    \[
    [B_1',\dots,B_k'] = [B_1,\dots , B_k] + v
    \]
    for some $v\in \lattice{d}$. We aim to prove that $[B_1',\dots,B_k']$ is a rotation of $[B_1,\dots , B_k]$.

    The vertices of the face $\B$ are ordered partitions of singleton blocks (permutations) that refine $\B$, that is each $B_\ell$ is the union of consecutive singleton blocks. Let $[i_1,\dots,i_{d+1}]$ be one of such vertices, and $[i_1',\dots,i_{d+1}']=[i_1,\dots,i_{d+1}]+v$ be the translated vertex that belongs to $\B'$.
    
    By the proof of the singleton blocks case, we know that $[i_1',\dots,i_{d+1}']$ is a rotation of $[i_1,\dots,i_{d+1}]$. Moreover, we can see from the proof that
    \[
    v = \sum_{j=1}^m w_{i_j}
    \]
    for some fixed value $m$ (such that $i_{m+1}=i_1'$).

    Now, the only way that $[B_1,\dots, B_k]+v$ can be equal to $[B_1',\dots, B_k']$ is that $\{i_1,i_2,\dots,i_m\}=\bigcup_{j=1}^r B_j$ for some $r$. 
    Thus, by~\Cref{lem_rotation_partition_one}, $\B'$ is obtained from $\B$ by rotating the first $r$ blocks to the end. 
\end{proof}

% \begin{proposition}
%     Every $d+1-k$ dimensional face of the permutahedral tiling~$\permutahedralTiling{d}$ is of the form 
%     \[
%     [B_1,\dots,B_k] + v
%     \]
%     for some ordered partition $[B_1,\dots,B_k]$ of $[d+1]$ with $1\in B_1$ and $v\in \lattice{d}$.
% \end{proposition}
% \begin{proof} 
%     We can consecutively apply rotations to the blocks of an ordered partition until reaching a new partition containing 1 in its first block.  
%     Therefore, by the Rotation Lemma~\ref{lem_rotation_partition_one}, every $d+1-k$ dimensional face of the permutahedron $\perm{d}$ can be written as 
%     \[
%     [B_1,\dots,B_k] + (a_1'w_1+\dots+a_{d+1}'w_{d+1})
%     \]
%     for some ordered partition $[B_1,\dots,B_k]$ of $[d+1]$ with $1\in B_1$ and $a_1',\dots,a_{d+1}'\in \Z$. Since the faces of the permutahedral tiling $\permutahedralTiling{d}$ are just translates of faces of $\perm{d}$ by the lattice $\lattice{d}$ spanned by the vectors $w_1,\dots,w_{d+1}$, then the result follows.
% \end{proof}
\begin{proposition}
    Every $d+1-k$ dimensional face of the permutahedral tiling~$\permutahedralTiling{d}$ can be obtained \emph{uniquely} as a translate  
    \begin{align}\label{eq_face_translates}
    [B_1,\dots,B_k] + v        
    \end{align}
    for some ordered partition $[B_1,\dots,B_k]$ of $[d+1]$ with $1\in B_1$ and $v\in \lattice{d}$.
\end{proposition}
\begin{proof} 
    We prove this result in two steps. First, we prove that every face of $\permutahedralTiling{d}$ is of the form~\eqref{eq_face_translates} with $1\in B_1$ and $v\in \lattice{d}$. Second, we prove that such a writing is unique. 
    
    For the first part, note that we can consecutively apply rotations to the blocks of an ordered partition until reaching a new partition containing 1 in its first block.  
    Therefore, by the First Rotation Lemma~\ref{lem_rotation_partition_one}, every $d+1-k$ dimensional face of the permutahedron $\perm{d}$ can be written as 
    \[
    [B_1,\dots,B_k] + (a_1'w_1+\dots+a_{d+1}'w_{d+1})
    \]
    for some ordered partition $[B_1,\dots,B_k]$ of $[d+1]$ with $1\in B_1$ and $a_1',\dots,a_{d+1}'\in \Z$. Since the faces of the permutahedral tiling $\permutahedralTiling{d}$ are just translates of faces of $\perm{d}$ by the lattice $\lattice{d}$ spanned by the vectors $w_1,\dots,w_{d+1}$, then every face of $\permutahedralTiling{d}$ is of the form
    \[
    [B_1,\dots,B_k] + v
    \]
    for some ordered partition $[B_1,\dots,B_k]$ of $[d+1]$ with $1\in B_1$ and $v\in \lattice{d}$.

    For the second part, uniqueness, we need to show that if
    \[
    [B_1,\dots,B_k] + v = [B_1',\dots,B_k'] + v' 
    \]
    with $1\in B_1\cap B_1'$ and $v,v'\in \lattice{d}$ then $B_i=B_i'$ for all $i$ and $v=v'$.

    By the Second Rotation Lemma~\ref{lem_rotation_partition_two}, we deduce that $[B_1',\dots,B_k']$ is a rotation of~$[B_1,\dots,B_k]$. Since $1\in B_1\cap B_1'$, then $B_i=B_i'$ for all $i$. From this one deduces that~$v=v'$. 
\end{proof}

As a straight forward corollary we obtain the following classification of the faces of the Heawood complex.

\begin{corollary}\label{cor_faces_HeawoodComplex}
    Every $d+1-k$ dimensional face of the Heawood complex $\heawoodComplex{d}$    
    can be obtained \emph{uniquely} as     
    \begin{align}\label{eq_face_translates_HeawoodComplex}
    [B_1,\dots,B_k] + v        
    \end{align}
    for some ordered partition $[B_1,\dots,B_k]$ of $[d+1]$ with $1\in B_1$ and $v\in \lattice{d}/\sublattice{\kk}$.    
\end{corollary}

We are now ready to prove~\Cref{prop_counting_faces}.

\begin{proof}[Proof of~\Cref{prop_counting_faces}]
    By~\Cref{cor_faces_HeawoodComplex}, the codimension $i$ faces of the Heawood complex can be obtained uniquely as 
    \[
    [B_1,\dots,B_{i+1}]+v
    \]
    for some ordered partition $[B_1,\dots,B_{i+1}]$ of $[d+1]$ with $1\in B_1$ and $v\in \lattice{d}/\sublattice{\kk}$. The number of such ordered partitions is $i! \, \stirling{d+1}{i+1}$ and the number of such $v's$ is~$D_{\kk}$~(Lemma~\ref{lem_quotient_size}). The result follows.  
\end{proof}

\subsection{Proof of Theorem~\ref{thm_heawoodgraph}}
By~\Cref{prop_duality_heawood_torus}, the Heawood graph~$\heawoodGraph{\kk}$ is the dual graph of the $d$-dimensional triangulated torus $\torus{\kk}$. 
Therefore, counting the number of vertices and edges of the graph is equivalent to counting the number of facets and codimension 1 faces of the torus, respectively. 
By~\Cref{thm_maintheorem}, we obtain that the number of vertices of $\heawoodGraph{\kk}$ is~$d! D_\kk$ and the number of edges is $\frac{(d+1)!}{2}D_\kk$.

Vertex transitivity of the graph $\heawoodGraph{\kk}$ follows from the translational symmetry of~$\complexAffine{d}$ (any vertex can be translated to any other vertex). 
We explore the symmetries more in the following section, through the lens of the automorphism group.

\section{The automorphism groups}
In this section, we study the automorphism groups of the triangulated torus \(\torus{\kk}\) and the generalized Heawood graph \(\heawoodGraph{\kk}\).
Any automorphism of \(\torus{\kk}\) is naturally an automorphism of \(\heawoodGraph{\kk}\), since the graph is determined by incidences of faces in the complex. 
A simplicial automorphism preserves incidences of facets, and so directly acts on the graph.
There may be automorphisms of the graph which do not arise from automorphisms of \(\torus{\kk}\), but we will see those are quite rare.
We focus first on the automorphisms of \(\complexAffine{d}\).

We partially follow the notation of K\"uhnel and Lassman~\cite{kuhnel_lassmann_symmetrictorii_1988} for subgroups of the automorphism group, starting with translation.
We represent this group of automorphisms by \(T\).
In this infinite setting, this group is generated by \(d\) different translations, translating by each \(w_i\) for \(1 \leq i \leq d\).
After the quotient, the number of generators for this group could be reduced, in some cases as much as to a single generator, translating by \(w_1\).

We decompose the remaining automorphisms in a slightly unusual way.
We denote by \(R\) the rotation action which replaces each \(w_i\) by \(-w_i\). 
This is orientation-preserving in even dimensions and orientation-reversing in odd dimensions.

Finally we have the group \(\mathfrak S_{d+1}\), the full symmetry group acting on the \(w_i\).
Of course there are many subgroups of this group, but we focus on the cyclic subgroup, represented by \(C_{d+1}\), which rotates the \(w_i\) to \(w_{i+1}\).

We now turn to the main statements about the automorphism groups of the quotient.

\begin{theorem}
\label{thm_torus_automorphism}
For \(\kk = (k,k,\ldots,k) \in \N^{d+1}\) with \(k \geq 1\) and \(d \geq 2\), the automorphism group of \(\torus{\kk}\) is generated by \(T,R\) and \(C_{d+1}\), with cardinality \(2 (d+1) \det M_k \).
\end{theorem}

\begin{proof}
From Lemma~\ref{lem_quotient_size_affine}, we see that the subgroup \(T\) is of order \(\det M_k\).
The subgroup \(R\) is of order \(2\) and \(C_{d+1}\) is of order \(d+1\).
Putting these together, the group of automorphisms generated by \(T, R\) and \(C_{d+1}\) has order \(2(d+1)\det M_k\).

Suppose \(\sigma\) is an automorphism of \(\torus{\kk}\).
It maps faces to faces and preserves incidences of faces.
In particular, \(\sigma\) maps the chain of faces 
\[(\{00\ldots0\}, \{00\ldots0,10\ldots0\}, \ldots, \{00\ldots0,10\ldots0,\ldots,11\ldots10\})\] 
to some chain of faces.
The image of the first face in the chain, the vertex \(\sigma(\{00\ldots0\})\), uniquely determines an element of \(T\), \(t_v\).
We can then write \(\sigma = t_v \circ \tau \), where \(\tau\) is an automorphism which fixes the vertex \(00\ldots0\).

Once the image of a facet under \(\tau\) is determined, the rest of the automorphism is fixed, since each facet which neighbors a fixed facet has no freedom in the label of its last vertex.
Since \(\tau\) fixes \(00\ldots0\), it is determined by how it maps a basis of~\(\sublatticeAffine{\kk}\).
The basis we choose are the rows of~\(M_k\).
The inner product of any vector in \(M_k\) with itself is \(2k^2+2k+1\).
The only vectors in \(\sublatticeAffine{\kk}\) with this property are the vectors of \(M_k\) and their negatives.
The inner product of any two distinct vectors in \(M_k\) is either 0 or \(-k^2-k\).
The inner product of a vector in \(M_k\) with a vector in \(-M_k\) is one of 0, \(k^2+k\) or \(-2k^2-2k-1\).
Since \(\tau\) must preserve the inner product, we conclude that it must map the basis \(M_k\) to either \(M_k\), or \(-M_k\), but not some mix between the two.
Finally, among all permutations of \(M_k\), \(\tau\) must cyclically shift the vectors of \(M_k\), since there is a natural successor vector within \(M_k\).

Therefore \(\tau\) is an element of the group \(<R,C_{d+1}>\).
Recalling that \(\sigma = t_v \tau\), we see that \(\sigma\) is an element of the group generated by \(T, R\) and \(C_{d+1}\), and indeed this group is the automorphism group of \(\torus{\kk}\).
\end{proof}

For the most part, the automorphisms of the graph \(\heawoodGraph{\kk}\) exactly match the automorphsims of \(\torus{\kk}\), with one exception. 
We state this as two theorems, but provide a joint proof of both theorems.

\begin{theorem}\label{thm_heawood_graph_automorphisms}
For \(\kk = (k,k,\ldots,k) \in \N^{d+1}\) with \(k \geq 1\) and \(d \geq 2\) except \(\kk = (1,1,1)\), the automorphism group of \(\heawoodGraph{\kk}\) is generated by \(T,R\) and \(C_{d+1}\), with cardinality \(2 (d+1) \det M_k \).
\end{theorem}

\begin{theorem}
For \(\kk = (1,1,1) \), the automorphism group of \(\heawoodGraph{\kk}\) is generated by \(T, R, C_{d+1},\) and an additional automorphism group \(W\), with cardinality 336.
\end{theorem}

\begin{proof}
    As a reminder, any automorphism of \(\torus{\kk}\) induces an automorphism of \(G_{\kk}\).
    What remains to be seen is that unless \(k = (1,1,1)\), there are no additional automorphisms of the graph which are not also automorphisms of \(\torus{\kk}\).
    Any automorphism of the graph must map cycles to cycles, and particularly map 4-cycles to 4-cycles and 6-cycles to 6-cycles.
    
    If \(d>2\), the only 6-cycles of \(\heawoodGraph{\kk}\) are those corresponding to the hexagons within the permutahedra generated by the simple roots \(a_i, a_{i+1}\).
    Likewise the only 4-cycles are the squares generated by \(a_i, a_j\) with \(|j-i| \geq 2\).
    Therefore, an automorphsim of \(\heawoodGraph{\kk}\) must map the 2-faces back to themselves.
    A simple complex can be determined solely from the 0-,1-, and 2-faces from a direct argument proved by Perles \cite{perles_problems_1970}.
    Consequently, an automorphsim of \(\heawoodGraph{\kk}\) must map all faces back to themselves, and in fact be equivalent to a simplicial automorphism of \(\torus{\kk}\).

    If \(d=2\), there are no 4-cycles in \(\heawoodGraph{\kk}\).
    Instead, we consider only the 6-cycles.
    If the only 6-cycles are those corresponding to hexagons of \(\complexAffine{2} \), we can rely on the previous argument to see that the automorphisms of \(\heawoodGraph{\kk}\) are those of \(\torus{\kk}\).
    The only case where additional 6-cycles appear is  when \(\kk = (1,1,1)\). 

    In this case there are additional automorphisms.
    We denote this group by \(W\).
    One such automorphism is \((3\overline{1}4,105)(4\overline{1}3,501)(132,024)(231,420)\).
    Direct computation, with computer assistance, shows that the automorphism group in this case has order 336.
\end{proof}

We now turn to the case where the \(k_i\) are not all the same, breaking some of the symmetry.
Here, there is perfect correspondence between the automorphisms of \(\torus{\kk}\) and those of \(\heawoodGraph{\kk}\).

\begin{theorem}
For \(\kk = (k_1,k_2,\ldots,k_{d+1}) \in \N^{d+1}\) with not all \(k_i\) equal and \(d \geq 2\), the automorphism group of \(\torus{\kk}\) is generated by \(T,R\) and some subgroup of \(C_{d+1}\), with cardinality \(2 \ell \det M_k \), where \(\ell\) is the cardinality of the subgroup of \(C_{d+1}\).
\end{theorem}

\begin{proof}
    From Lemma~\ref{lem_quotient_size_affine}, the number of vertices of \(\torus{\kk}\) is \(\det M_k\), so \(T\) is of order  \(\det M_k\).
    The order of \(R\) is 2.

    Following the proof of Theorem~\ref{thm_torus_automorphism},
    we consider an atuomorphism \(\sigma\) of \(\torus{\kk}\).
    Again, \(\sigma = t_v \circ \tau\), where \(t_v\) is an element of \(T\) and \(\tau\) fixes \(00\ldots0\).
    Likewise,  \(\tau\) must map the basis \(M_k\) to some basis of \(\sublatticeAffine{\kk}\).
    In this case, however, the lengths of the various vectors of \(M_k\) are not all the same.
    Whatever action \(\tau\) has on the rows of \(M_k\), it must preserve the successor of each row, like in the previous proof.
    As before, \(\tau\) could swap to \(-M_k\), which accounts for the subgroup \(R\).
    Therefore \(\tau\) is contained in \(<R,C_{d+1}>\), however not all elements of \(C_{d+1}\) are permitted.
    The subgroup of  \(C_{d+1}\) which is permitted is the one that acts on the entries of \(\kk\) without changing the values at any position.
\end{proof}

To ensure the ideas of the proof are clear enough, we present a small collection of examples.
If \(\kk = (2,1,2,1)\), then the subgroup of \(C_{4}\) which fixes the values in \(\kk\) is \(C_2\), and if \(\kk =  (2,3,4,3,2,3,4,3,2,3,4,3)\) the subgroup of \(C_{12}\) is \(C_3\).

\begin{theorem}
For \(\kk = (k_1,k_2,\ldots,k_{d+1}) \in \N^{d+1}\) with not all \(k_i\) equal and \(d \geq 2\), the automorphism group of \(\heawoodGraph{\kk}\) is  generated by \(T,R\) and some subgroup of \(C_{d+1}\), with cardinality \(2 \ell \det M_k \), where \(\ell\) is the cardinality of the subgroup of \(C_{d+1}\).
That is, the automorphism group of \(\heawoodGraph{\kk}\) is identical to the automorphism group of \(\torus{\kk}\).
\end{theorem}

\begin{proof}
    When \(d>2\), the proof in the all equal \(k_i\) setting applies, and we conclude that the automorphism groups of  \(\heawoodGraph{\kk}\)  and \(\torus{\kk}\) are the same.

    If \(d=2\), we consider the 6-cycles as we did previously.
    If all three \(k_i \geq 2\), the only 6-cycles are the 2-faces.
    If two \(k_i \geq 2\), again the only 6-cycles are the 2-faces.
    Since not all \(k_i\) are equal, the only remaining case is when \(k_i= (1,1,n)\).
    In this case, there are some 6-cycles which do not correspond to 2-faces.
    Six non-equivalent 6-cycles contain \(v\), for any fixed vertex \(v\) of \(\heawoodGraph{(1,1,n)}\).
    Since \(T\) is an automorphism of both \(\heawoodGraph{(1,1,n)}\) and  \(\torus{\kk}\), we choose \(v = 123\) without loss of generality.
    We now list the six 6-cycles containing \(123\), and illustrate them in Figure~\ref{fig_sixsixcycles}.

\begin{center}
\begin{tabular}{lllllll}
    123,&024,&015,&105,&204,&213,&123\\
    123,&132,&042,&$\overline{1}43$,&$\overline{1}34$,&024,&123\\
    123,&213,&312,&321,&231,&132,&123\\
    123,&213,&312,&402,&$4\overline{1}3$,&$5\overline{2}3$,&$6\overline{2}2$\\
    123,&213,&204,&$3\overline{1}4$,&$4\overline{1}3$,&$5\overline{2}3$,&$6\overline{2}2$\\
    123,&213,&312,&402,&501,&$6\overline{1}1$,&$6\overline{2}2$\\
\end{tabular} 
\end{center}

    Directly from the structure of these six 6-cycles, we can determine which cycles come from faces.

For any pair of neighorbors of \(123\), there is a unique cycle which describes that face.
There is a unique cycle containing \(024\) and \(132\), so it must be the distinguished cycle.
We distinguish the face containing \(213\) and \(132\) by the cycle which contains two unique vertices among all the cycles, \(321\) and \(231\).
Likewise we distinguish the face containing \(213\) and \(024\) by the cycle containing two unique vertices, \(105\) and \(015\).
We are therefore able to determine the 2-faces of the complex from the graph, and any graph automorphism which fixes \(123\) must also fix the 2-faces.
    After repeating Perles's argument to this automorphism which maps 2-faces to 2-faces, we find this automorphism is a simplical automorphism, and therefore the automorphisms of  \(\heawoodGraph{\kk}\) are exactly the automorphisms of \(\torus{\kk}\).
\end{proof}

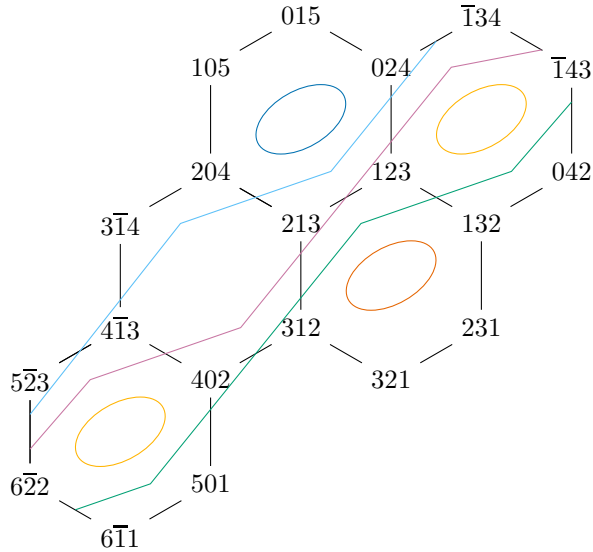
\begin{figure}[htb]
    \centering
     \begin{tikzpicture}%
	[x={(240:2cm)},
	y={(0:2cm)},
	z={(120:2cm)},
    scale=0.4]

\node (123) at (1,2,3) {123};
\node (213) at (2,1,3) {213};
\node (312) at (3,1,2) {312};
\node (321) at (3,2,1) {321};
\node (231) at (2,3,1) {231};
\node (132) at (1,3,2) {132};
\node (042) at (0,4,2) {042};
\node (-143) at (-1,4,3) {$\overline{1}43$};
\node (-134) at (-1,3,4) {$\overline{1}34$};
\node (024) at (0,2,4) {024};
\node (015) at (0,1,5) {015};
\node (105) at (1,0,5) {105};
\node (204) at (2,0,4) {204};
\node (402) at (4,0,2) {402};
\node (4-13) at (4,-1,3) {$4\overline{1}3$};
\node (3-14) at (3,-1,4) {$3\overline{1}4$};
\node (5-23) at (5,-2,3) {$5\overline{2}3$};
\node (6-22) at (6,-2,2) {$6\overline{2}2$};
\node (501) at (5,0,1) {501};
\node (6-11) at (6,-1,1) {$6\overline{1}1$};

\draw (123) -- (024) -- (015) -- (105) -- (204) -- (213);
\draw (123) -- (132) -- (042) -- (-143) -- (-134) -- (024);
\draw (123) -- (213) -- (312) -- (321) -- (231) -- (132);
\draw (312) -- (402) -- (4-13) -- (5-23) -- (6-22);
\draw (213) -- (204) -- (3-14) -- (4-13) -- (5-23) -- (6-22);
\draw (402) -- (501) -- (6-11) -- (6-22);

%drawn from top right to bottom left
\draw[purple] (-3/3,11/3,10/3) -- (-1/3,8/3,11/3) -- (10/3,1/3,7/3) -- (14/3,-4/3,8/3) -- (17/3,-6/3,7/3);
\draw[bluishgreen] (-2/3,12/3,8/3) -- (1/3,10/3,7/3) -- (5/3,5/3,8/3) -- (16/3,-2/3,4/3) -- (18/3,-9/6,9/6) ;
\draw[skyblue] (-3/6,15/6,12/3) -- (4/3,4/3,10/3) -- (8/3,-1/3,11/3) -- (16/3,-6/3,8/3);

\draw[orange](0,3,3) circle (2/3); 
\draw[orange](5,-1,2) circle (2/3); 
\draw[red](0,0,0) circle (2/3); 
\draw[blue](-3,-3,0) circle (2/3); 
\end{tikzpicture}  
            
    \caption{The six 6-cycles of \(\heawoodGraph{(1,1,n)}\).}
    \label{fig_sixsixcycles}
\end{figure}

\section{Hyperbolic setting}
The work we have done here can be repeated in higher levels of generality.
For any triangulation with a group action, and a quotient of that group, a simplicial complex can be obtained. 
When the group and quotient have high degrees of symmetry, the resulting complex will also be highly symmetric.
Likewise, the dual graph will exhibit similarly symmetric properties.
As a concrete example, we use the Klein quartic. 

\begin{example}[\((2,3,7)\)-hyperbolic triangle group]
We refer to the labels in Figure~\ref{fig_klein_quartic}.
We represent the \((2,3,7)\)-hyperbolic triangle group by equilateral triangles infinitely tiling hyperbolic space, with 7 triangles to a vertex.
This group is generated by \(S,F_1,F_2\), where \(S\) is the rotation by \(\pi\) around the center of the edge \(ab\), \(F_1\) is the reflection about the line \(ab\), and \(F_2\) is the reflection about the line \(aj\). 
The composition \(F_1F_2\) gives a rotation by \(2 \pi /7\) around \(a\). 
To match with other conventions, we label \(T = F_1F_2\).
The group generated by \(S,T\) alone is the orientation preserving \((2,3,7)\)-hyperbolic triangle group.
We can see quite naturally that \(S^2 = Id\) and that \(T^7 = Id\). 
Further inspection reveals that \((ST)^3 = Id\) as well. 
\end{example}

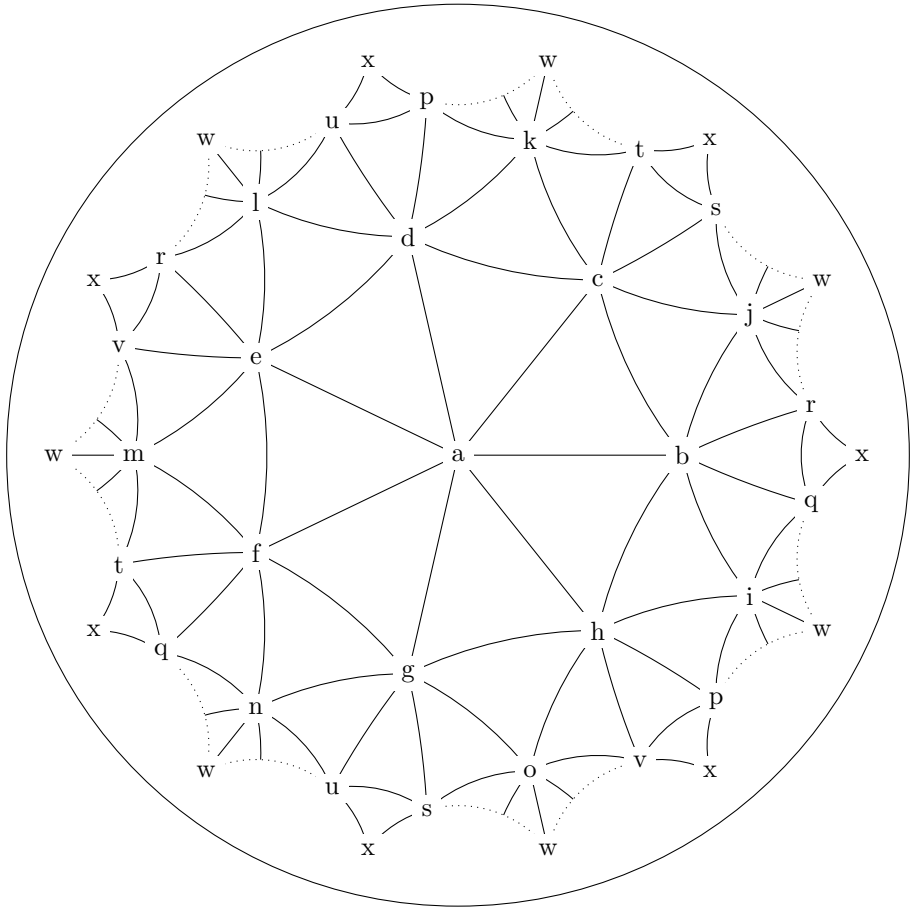
\begin{figure}[htb]
\centering
 \begin{tikzpicture}[ hyperbolic disc radius=1cm,
every to/.style={hyperbolic disc}, scale=6]
\draw (0,0) circle[radius=\pgfkeysvalueof{/tikz/hyperbolic disc radius}];

%c and b are used, but these are x and y coordinates
%a name collision is probably happening somewhere
%coordinates are coming from the output of the python code here:
%https://gitlab.com/parclytaxel/Dounreay/-/tree/ed3e4f1c6e2617855d1de483c36b42928c6b12e6/heptille
\tikzmath{
\c1 = 0.49697042; 
\b1 = 0; 
\c2 = 0.30985599; 
\b2 = 0.38854712; 
\c3 = 0.64702455; 
\b3 = 0.31159060;
\c4 = 0.57199748;
\b4 = 0.54434242;
\c5 = 0.73029830;
\b5 = 0.48451033;
\c6 = 0.80682641;
\b6 = 0.38854712;
\c7 = 0.78221864;
\b7 = 0.10781369;
\b{107} = -0.10781369;
\c8 = 0.83413898;
\b8 = 0.26888295;
\c9 = 0.89550976;
\b9 = 0;
\c{10} = 0.68691809;
\b{10} = 0.41851483;
\c{11} = 0.75549450;
\b{11} = 0.27611446;
}

\foreach \a in {0,...,6} {
\draw[rotate=\a*360/7] (0,0) -- (\c1,\b1);
\draw[rotate=\a*360/7] (\c1,\b1) to (\c3,\b3);
\draw[rotate=\a*360/7] (\c1,\b1) to (\c2,\b2);
\draw[rotate=\a*360/7] (\c2,\b2) to (\c3,\b3);
\draw[rotate=\a*360/7] (\c3,\b3) to (\c4,\b4);
%\draw[rotate=\a*360/7] (\c5,\b5) to (\c4,\b4);
\draw[rotate=\a*360/7] (\c7,\b7) to (\c3,\b3);
\draw[rotate=\a*360/7] (\c{11},\b{11}) to (\c3,\b3);
%\draw[rotate=\a*360/7] (\c7,\b7) to (\c8,\b8);
\draw[rotate=\a*360/7] (\c1,\b1) to (\c7,\b7);
\draw[rotate=\a*360/7] (\c1,\b1) to (\c7,\b{107});
\draw[rotate=\a*360/7] (\c7,\b{107}) to (\c7,\b7);
%\draw[rotate=\a*360/7] (\c8,\b8) to (\c6,\b6);
%\draw[rotate=\a*360/7] (\c6,\b6) to (\c5,\b5);
\draw[rotate=\a*360/7] (\c3,\b3) to (\c{10},\b{10});
\draw[rotate=\a*360/7] (\c3,\b3) -- (\c6,\b6);
\draw[rotate=\a*360/7] (\c9,\b9) to (\c7,\b{107});
\draw[rotate=\a*360/7] (\c9,\b9) to (\c7,\b7);
\draw[rotate=\a*360/7, dotted] (\c7,\b7) to (\c6,\b6);
\draw[rotate=\a*360/7, dotted] (\c4,\b4) to (\c6,\b6);
}

\node[fill=white] at (0,0) {a};
\node[fill=white] at (\c1,\b1) {b};
\node[fill=white] at (\c2,\b2) {c};
\node[fill=white] at ($(0,0)!1!360/7*2:(\c1,\b1)$) {d};
\node[fill=white] at ($(0,0)!1!360/7*3:(\c1,\b1)$) {e};
\node[fill=white] at ($(0,0)!1!360/7*4:(\c1,\b1)$) {f};
\node[fill=white] at ($(0,0)!1!360/7*5:(\c1,\b1)$) {g};
\node[fill=white] at ($(0,0)!1!360/7*6:(\c1,\b1)$) {h};
\node[fill=white] at ($(0,0)!1!360/7*6:(\c3,\b3)$) {i};
\node[fill=white] at ($(0,0)!1!360/7*0:(\c3,\b3)$) {j};
\node[fill=white] at ($(0,0)!1!360/7*1:(\c3,\b3)$) {k};
\node[fill=white] at ($(0,0)!1!360/7*2:(\c3,\b3)$) {l};
\node[fill=white] at ($(0,0)!1!360/7*3:(\c3,\b3)$) {m};
\node[fill=white] at ($(0,0)!1!360/7*4:(\c3,\b3)$) {n};
\node[fill=white] at ($(0,0)!1!360/7*5:(\c3,\b3)$) {o};
\node[fill=white] at ($(0,0)!1!360/7*6:(\c7,\b7)$) {p};
\node[fill=white] at ($(0,0)!1!360/7*0:(\c7,\b7)$) {r};
\node[fill=white] at ($(0,0)!1!360/7*1:(\c7,\b7)$) {t};
\node[fill=white] at ($(0,0)!1!360/7*2:(\c7,\b7)$) {u};
\node[fill=white] at ($(0,0)!1!360/7*3:(\c7,\b7)$) {v};
\node[fill=white] at ($(0,0)!1!360/7*4:(\c7,\b7)$) {q};
\node[fill=white] at ($(0,0)!1!360/7*5:(\c7,\b7)$) {s};
\node[fill=white] at ($(0,0)!1!360/7*6:(\c7,\b{107})$) {v};
\node[fill=white] at ($(0,0)!1!360/7*0:(\c7,\b{107})$) {q};
\node[fill=white] at ($(0,0)!1!360/7*1:(\c7,\b{107})$) {s};
\node[fill=white] at ($(0,0)!1!360/7*2:(\c7,\b{107})$) {p};
\node[fill=white] at ($(0,0)!1!360/7*3:(\c7,\b{107})$) {r};
\node[fill=white] at ($(0,0)!1!360/7*4:(\c7,\b{107})$) {t};
\node[fill=white] at ($(0,0)!1!360/7*5:(\c7,\b{107})$) {u};
\node[fill=white] at ($(0,0)!1!360/7*6:(\c6,\b6)$) {w};
\node[fill=white] at ($(0,0)!1!360/7*0:(\c6,\b6)$) {w};
\node[fill=white] at ($(0,0)!1!360/7*1:(\c6,\b6)$) {w};
\node[fill=white] at ($(0,0)!1!360/7*2:(\c6,\b6)$) {w};
\node[fill=white] at ($(0,0)!1!360/7*3:(\c6,\b6)$) {w};
\node[fill=white] at ($(0,0)!1!360/7*4:(\c6,\b6)$) {w};
\node[fill=white] at ($(0,0)!1!360/7*5:(\c6,\b6)$) {w};
\node[fill=white] at ($(0,0)!1!360/7*5:(\c9,\b9)$) {x};
\node[fill=white] at ($(0,0)!1!360/7*6:(\c9,\b9)$) {x};
\node[fill=white] at ($(0,0)!1!360/7*0:(\c9,\b9)$) {x};
\node[fill=white] at ($(0,0)!1!360/7*1:(\c9,\b9)$) {x};
\node[fill=white] at ($(0,0)!1!360/7*2:(\c9,\b9)$) {x};
\node[fill=white] at ($(0,0)!1!360/7*3:(\c9,\b9)$) {x};
\node[fill=white] at ($(0,0)!1!360/7*4:(\c9,\b9)$) {x};

\end{tikzpicture}
\caption{The triangulated Klein quartic.}
\label{fig_klein_quartic}

\end{figure}

To transform this infinite group with a triangulation into a finite simplicial complex, we need to take a quotient of the group which results in a finite group.
The quotient is then by \((ST^3)^4 = Id\), which gives the Klein quartic.

Reading the set of triangles in Figure~\ref{fig_klein_quartic}, we find the following facet list.

\begin{center}
\begin{tabular}{ccccccc}
[a,b,c],&[a,c,d],&[a,d,e],&[a,e,f],&[a,f,g],&[a,g,h],& [a,h,b], \\\
[b,c,j],&[c,d,k],&[d,e,l],&[e,f,m],&[f,g,n],&[g,h,o],& [h,b,i], \\\
[c,j,s],&[d,k,p],&[e,l,r],&[f,m,t],&[g,n,u],&[h,o,v],&[b,i,q],\\\
[c,s,t],&[d,p,u],&[e,r,v],&[f,t,q],&[g,u,s],&[h,v,p],&[b,q,r],\\\
[c,t,k],&[d,u,l],&[e,v,m],&[f,q,n],&[g,s,o],&[h,p,i],&[b,r,j],\\\
[t,k,m],&[u,l,n],&[v,m,o],&[q,n,i],&[s,o,j],&[p,i,k],&[r,j,l],\\\
[x,q,r],&[x,r,v],&[x,v,p],&[x,p,u],&[x,u,s],&[x,s,t],&[x,t,q],\\\
[w,i,n],&[w,n,l],&[w,l,j],&[w,j,o],&[w,o,m],&[w,m,k],&[w,k,i].\\\
\end{tabular}
\end{center}

The automorphism group of this complex is order \(336 = 7 \cdot 24 \cdot 2\), which is seen by the \(7\)-fold rotational symmetry, the transitive action on 24 vertices, and an orientation-reversing action.
The dual graph also has automorphism group of order \(336 = 7 \cdot 24 \cdot 2\). 
Much as in the euclidean case, there are no graph automorphisms that do not come from simiplical automorphisms, since the small cycles of the graph only appear as links of vertices.

The Klein quartic is only one example. 
There are many more group/subgroup pairs in the hyperbolic setting and beyond.
However, there are not any easily parameterized families of such examples in other settings.
Instead, an interested reader may wish to extend this work to the setting of Hurwitz surfaces, a special class of quotients of the \((2,3,7)\) triangle group.
It is certainly possible that surprises await in the \((2,3,n)\) triangle group.
In particular there may be another example where the graph automorphism group is larger than the simplicial complex automorphism group.

\section{Previous appearances in the literature}
We were not the first to consider such symmetrical tori, and offer this brief overview of previous work.

The first time such a torus was considered is in the work of Heawood, where a 7-region map of the torus was given~\cite{heawood_torus_1984}. 
The dual triangulation is the torus \(T_{1,1,1}\).

\url{https://gdz.sub.uni-goettingen.de/id/PPN600494829_0024?tify=%7B%22pages%22%3A%5B350%5D%2C%22view%22%3A%22toc%22%7D}

The extension to higher dimensions begins with K\"uhnel and Lassmann. 
They first introduced a 15 vertex triangulation of the 3-torus which is \(T_{1,1,1,1}\)~\cite{kuhnel_lassmann_threedtorus_1984}, and then generalized this to all dimensions, \(T_{1,1,\ldots,1}\)~\cite{kuhnel_lassmann_symmetrictorii_1988}.

%Intermediate work for E8?

Our construction uses the lattice \(Z^n\), but this isn't strictly necessary.
Grigis extended the torus of K\"uhnel and Lassmann to a triangulation for every generic lattice, which is then strongly related to \(T_{1,1,\ldots,1}\), although with a different underlying triangulation~\cite{grigis_delaunaytriangulations_2009}.

The triangulations obtained by Grigis are essentially geometrically motivated re-triangulations of the same complex discovered by K\"uhnel and Lassmann.
Sikiri\'c and Garber extended this idea, and gave even more re-triangulations, without relying on an underlying lattice~\cite{garber_sikiric_periodictriangulations_2020}. 
However, they also essentially only considered triangulations of \(T_{0,0,\ldots,0}\), a one-point torus.

Some efforts were made extending the Heawood graph in two dimensions as well.
Any graph which can be obtained by a quotient of $\affineGraph{2}$ can be described as a generalized Heawood graph.
These graphs are all vertex-transitive and 3-regular.
However, not all of these graphs are necessarily edge-transitive.
Those which are edge-transitive appear in the Foster Census, and in fact make up a significant subset of the census.
The first graphs that appear in the Census and are torus quotients of $\affineGraph{2}$ appear in Figure~\ref{fig_fostercensus}.

\begin{figure}[htb]
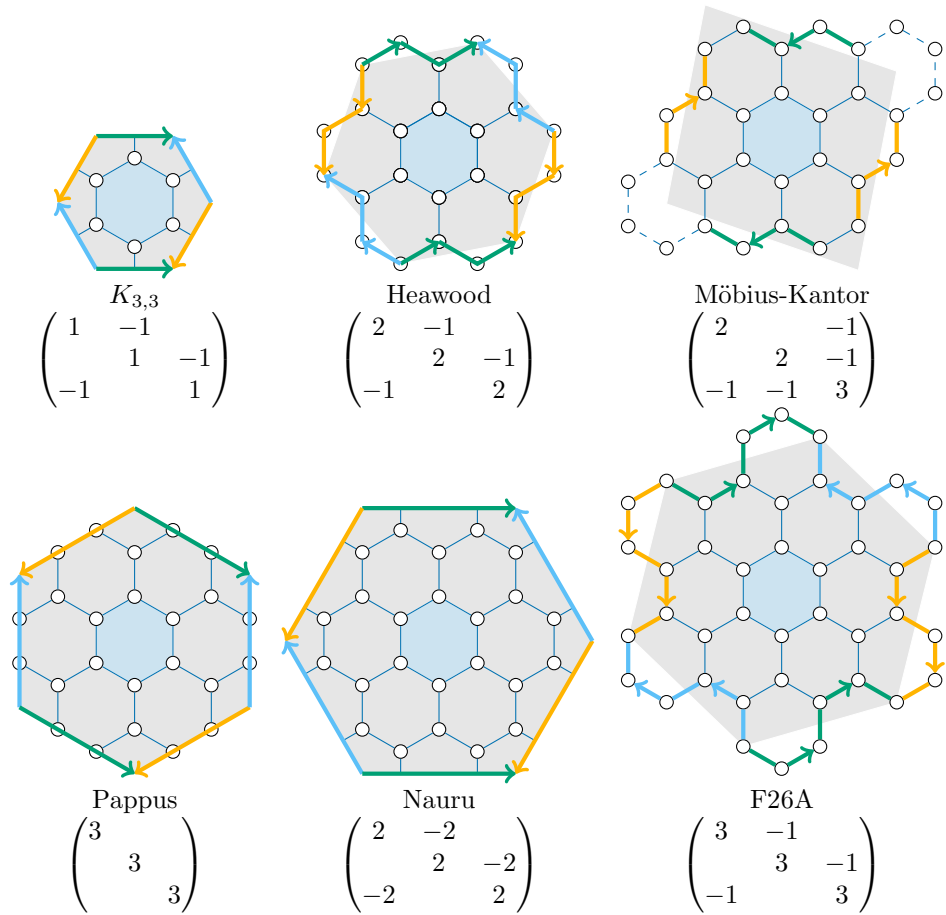

    \centering
    \include{figures/fig_fostercensus}
    \caption{The first six graphs in the Foster Census which can be presented as a torus quotient.}
    \label{fig_fostercensus}
\end{figure}

This brief summary of previous work, in the new language of our construction, suggests many ways to continue this work.

\section{Discussion and open questions}
\subsection{Map coloring}
Returning to the original motivation for Heawood, we can consider map coloring of higher dimensional spaces, with particular interest in the torus.
It is well known that if the regions of the map in dimension higher than two have no restriction, then an unbounded number of colors may be required.
It remains an interesting problem to find constraints on the regions so that the number of colors required to color any map is finite.
A possible restriction would be to require each region of a map to be a Voronoi cell.
By further requesting that the points generating the Voronoi cell are in sufficiently generic position, the result is a simple celluation, and therefore dual to a simplicial complex.
However, even in this case, the number of colors may be unbounded, due to a result of Gonska~\cite[Lemma 1.3.3]{gonska_thesis_2012}.
They showed that the cyclic polytope can be realized as a Delauney triangulation.
The dual complex is therefore a Voronoi celluation, and can be completed to a sphere or torus.
Since the graph of the cyclic polytope is a complete graph for \(d>3\), the resulting Voronoi cellulation will have unbounded map color number.

There is a model which gives bounded coloring number for higher dimensional settings.
If the regions of the map are all non-overlapping spheres, then the coloring number is bounded~\cite{bagchi_higher_2013}.
They rely on the fact that the \emph{kissing number}, the number of balls of the same radius can lie tangent to a single ball, is bounded.
They show that the map coloring number of a $d$-sphere is between $d+2$ and $\tau(d)+1\leq 3^d$ for $d\geq 2$, where $\tau(d)$ is the \emph{kissing number} of the $d$-dimensional Euclidean space.
This model unfortunately does not fill space, and so does not implicitly describe a simplicial complex.
A simplicial complex can be assigned to it by taking the flag complex of the adjacency graph, but this will not in general give a triangulation of the underlying space.

\begin{question}
    Is there a model for triangulations of \(T^d\) for \(d>2\) which have bounded chromatic number for the underlying graph?
\end{question}

A quite restrictive positive answer to this question is the Voronoi celluation where the points are elements of a lattice.
This is relatively unsatisfying, however, since the number of triangulations generated this way is finite in each dimension.

\subsection{Minimal triangulations}

Although it was not recognized at the time, Heawood's map of the torus is dual to the minimal triangulation of the 2-torus.
The generalization of this triangulation given in~\cite{kuhnel_lassmann_symmetrictorii_1988} was directly motiviated by the desire to find minimal triangulations of the torus in all dimensions.
Among all known triangulations of the torus, this generalization, $\torus{(1,\ldots,1)}$, has the fewest vertices.
This has lead to the following conjecture.

\begin{conjecture}[Minimal triangulation, Cf.~{\cite[Conjecture 21]{lutz_triangulated_2005}}]
\label{conj_minimaltriangulation}
The minimal number of vertices required to triangulate the \(d\)-torus is \(2^{d+1}-1\).
In particular, $\torus{(1,\ldots,1)}$ (with $d+1$ ones in the sub-index) is a minimal triangulation of the \(d\)-torus. 
\end{conjecture}

Though not stated as a conjecture, or even a question, this is hinted at in the introductions of \cite{kuhnel_lassmann_symmetrictorii_1988} and \cite{grigis_fourtorus_1998}.

This conjecture is true for \(d=1,2\).
In~\cite{lutz_triangulated_2005}, Lutz shows that there is no triangulation of the \(3\)-torus that is vertex-transitive with fewer vertices. 
They also further show that this triangulation uses a minimal number of edges.
Finally, Lutz attempted to use their software package BISTELLER to find a smaller triangulation for \(d=3,4\), unsuccessfully.
Despite all this strong evidence, the conjecture remains open, and for higher \(d\), there are no good tools available.

\subsection{Graph properties of \texorpdfstring{\(\heawoodGraph{\kk}\)}{Hk}}
The Heawood graph has several interesting properties, and many of these generalize to \(\heawoodGraph{\kk}\).
The Heawood graph is vertex-transitive, as is \(\heawoodGraph{\kk}\).

The Heawood graph has chromatic number 2, as does \(\heawoodGraph{\kk}\) when \(d\) is even.
Every element of \(\sublattice{\kk}\) can be labeled in an alternating way, there are no odd cycles.
When \(d\) is even, this labeling is constant under the displacement by the \(w_i\), and therefore any quotient has a consistent labeling of each vertex.
When \(d\) is odd, the labeling of \(\sublattice{\kk}\) is not consistent under translations by \(w_i\), so \(\kk\) can be chosen such that \(\heawoodGraph{\kk}\) is not bipartite.

The Heawood graph is edge-transitive, however \(\heawoodGraph{\kk}\) is not edge-transitive in higher dimensions.
From the vertex transitivity, we find that any edge of \(\heawoodGraph{\kk}\) of the form \(e_i-e_j\) can be transformed into any other edge of the same form.
The automorphism group which is a subgroup of of \(C_{d+1}\) determines the extent of the equivalence classes of edges under automorphisms of \(\heawoodGraph{\kk}\).
For example, in  \(\heawoodGraph{(1,1,1,1)}\), there is no automorphism which transforms edges of the form \(e_1-e_3\) into edges of the form \(e_1-e_2\), since each of \(T,R,\) and \(C\) doesn't change the cyclic distance between the coordinates altered along the edge.

\begin{conjecture}
    For any \(\kk\), \(\heawoodGraph{\kk}\) is Hamiltonian, with the cycle given by alternating edges \(e_i-e_{i+1}\) and \(e_{i-1}-e_i\) for some \(i\).
\end{conjecture}

It is easy to construct a \(\kk\) so that the above conjecture is not true for a particular \(i\).
When \(\kk=(1,3,2)\), choosing \(i=1\) does not give a Hamiltonian cycle. The path which starts at \(123\) continues \(213,312,402,501,6\overline{1}1,7\overline{1}0,\ldots,(12)\overline{3}\overline{3},(13)\overline{4}\overline{3}\).
Expanding \((1,0,-1) M_\kk  = (3,-3,-3)\), then realizing this point as \(3w_1-3w_2-3w_3 = (12, -6, -6)\), we see that \(123\) and \((13)\overline{4}\overline{3}\) differ by an integer linear combination of the rows of \(M_\kk\).
So this cycle returns to its starting point after only 12 vertices.
From Lemma~\ref{lem_quotient_size_affine}, we can compute directly that there are 36 vertices of \(\heawoodGraph{(1,3,2)}\).
However in this example, choosing \(i=2\) does give a Hamiltonian cycle.
Is there always a choice of \(i\) which will give a Hamiltonian cycle?

\subsection{Identifying opposite facets on polytopes}
As we have shown in~\Cref{prop_permutahedron_torus}, the result of identifying opposite facets on a permutahedron by translation is a topological torus. 
This result is quite intriguing because even in dimension~2, identifying opposite edges of regular $2n$-gons gives rise to interesting surfaces. 

For instance, identifying opposite 
edges by translation (preserving the orientation) of either a 4-gon or a 6-gon gives rise to a torus (with one hole), while identifying opposite edges (preserving the orientation) of an 8-gon gives rise to a 2-hole torus. 
One simple way of seeing this is to compute the Euler characteristic and use the classification theorem for closed surfaces according to their genus and orientability. 
Identifying opposite edges of an 8-gon gives rise to $1$ vertex, $4$ edges, and $1$ face. 
So, the Euler characteristic of the resulting space is $1-4+1=-2$. 
Since the result is orientable, one can deduce that it is a 2-hole torus.  
Continuing the pattern, a \(4n\)-gon with opposite sides identified is a n-holed torus, as is a \((4n+2)\)-gon with opposite sides identified.

It would be interesting to understand the topology of other families of polytopes with opposite facets identified by translation. 
Natural families that fit into this context are 
\begin{enumerate}
\item \emph{Permutahedra arising from finite Coxeter groups}. An example of this, is the type $B$ permutahedron, which is the dual of the barycentric subdivision of a cube. For instance, the type $B_2$ permutahedron is an 8-gon, and identifying opposite facets by translation gives a 2-hole torus. 
\item \emph{Postnikov's generalized permutahedra}~\cite{postnikov_permutohedra_2009} obtained by removing some pairs of opposite facets of the classical permutahedron $\perm{d}$. That is, removing facets corresponding to proper subsets of $[d+1]$ in a collection $\mathcal{A}\subset 2^{[d+1]}$ that is closed under taking complements. 
% as Minkowski sums 
% \[
% \sum_{A\in \mathcal A} \Delta_A,
% \]
% for some $\mathcal A \subseteq 2^{[n]}$, a collection of subsets of $[n]$, that is closed under taking complements, where 
% $
% \Delta_A = \conv \{e_a: \, a\in A\}.
% $
% An example of this is the cyclohedron. 
\item The realization of the cyclohedron in~\cite{postnikov_permutohedra_2009}, which is obtained as a Minkowski sum 
\[
\sum_{A\in \mathcal A} \Delta_A,
\]
where $\mathcal{A}$ is the collection of cyclic intervals of the form $[i,j]$ and $[1,i]\cup [j,n]$ for $1\leq i\leq j \leq n$, and 
$
\Delta_A = \conv \{e_a: \, a\in A\}
$.
\item \emph{Zonotopes} in general, which are Minkowski sums of line segments. 
%This generality includes all our previous examples.
\end{enumerate}
Importantly, we refer to polytopes whose opposite facets are translates of each other, although the slightly more general setting where opposite facets are combinatorially equivalent is also worth investigating.
Poincar\'e's homology sphere is an example in this more general setting;
it is obtained by identifying opposite facets of the dodecahedron using a clockwise twist to align the facets.

\bibliographystyle{plain}
\bibliography{biblio}

\end{document}